\documentclass[10pt]{amsart}

\usepackage[a4paper,top=3cm,bottom=3cm,left=2cm,right=2cm]{geometry}
\usepackage[T1]{fontenc} 
\usepackage[utf8]{inputenc}
\usepackage[english]{babel}

\usepackage{amssymb}
\usepackage{amsthm}
\usepackage{amsmath}
\usepackage{amsfonts}
\usepackage{amscd}
\usepackage{bbm}
\usepackage{enumerate}
\usepackage{times}
\usepackage[mathscr]{eucal}
\usepackage{indentfirst}
\usepackage{verbatim}
\usepackage{lipsum}
\usepackage{mathrsfs}
\usepackage{bm}
\usepackage{dsfont}
\usepackage{latexsym}
\usepackage{yhmath}
\usepackage{braket}
\usepackage{hyperref}

\usepackage{subfig}

\usepackage[normalem]{ulem}

\usepackage[svgnames]{xcolor}
\hypersetup{
	colorlinks = true,
	citecolor = teal, 
	linkcolor=red,
	urlcolor=blue}

\usepackage[style=numeric,backend=bibtex,natbib=true,maxnames=5,giveninits=true]{biblatex}

\usepackage{orcidlink}

\usepackage{caption}
\usepackage{float}

\allowdisplaybreaks
\makeindex

\newtheorem{theorem}{Theorem}[section]
\newtheorem{claim}[theorem]{Claim}
\newtheorem{lemma}[theorem]{Lemma}
\newtheorem{proposition}[theorem]{Proposition}
\newtheorem{definition}[theorem]{Definition}
\newtheorem{example}[theorem]{Example}

\newtheorem{remark}[theorem]{Remark}
\newtheorem{corollary}[theorem]{Corollary}
\newtheorem{assumption}[theorem]{Assumption}

\numberwithin{equation}{section}
\allowdisplaybreaks

\newcommand{\dif}{\mathrm{d}}
\newcommand{\E}{\mathbb{E}}
\newcommand{\R}{\mathbb{R}}
\newcommand{\C}{\mathbb{C}}

\newcommand{\N}{\mathbf{N}}
\newcommand{\supp}{\mathrm{supp}}
\newcommand{\ii}{\mathrm{i}}
\newcommand{\ee}{\mathrm{e}}
\newcommand{\rd}{\mathrm{d}}

\DeclareMathOperator{\dist}{dist}
\DeclareMathOperator{\sign}{sign}
\DeclareMathOperator{\spec}{spec}

\newcommand{\other}[1]{\widetilde{#1}}
\newcommand{\Brwn}{\mathfrak{B}}
\newcommand{\Tr}{\mathrm{Tr}}

\newcommand{\norm}[1]{\left\lVert#1\right\rVert}
\newcommand{\vertiii}[1]{{\left\vert\kern-0.3ex\left\vert\kern-0.3ex\left\vert #1 
		\right\vert\kern-0.3ex\right\vert\kern-0.3ex\right\vert}}

\newcommand{\tin}{t_{\mathrm{init}}}
\newcommand{\tfin}{t_{\mathrm{final}}}

\newcommand{\etaf}{\eta_{\mathfrak{f}}}
\newcommand{\kapd}{\varkappa}

\newcommand{\subD}{\mathcal{D}^\mathrm{sub}}
\newcommand{\abvD}{\mathcal{D}^\mathrm{abv}}
\newcommand{\globD}{\mathcal{D}^{\mathrm{glob}}}
\newcommand{\outD}{\mathcal{D}^{\mathrm{out}}}
\newcommand{\bddD}{\mathcal{D}^\mathrm{bdd}}

\newcommand{\smallsymbol}[1]{\scalebox{0.65}{#1}}
\newcommand{\dift}{{\text{\smallsymbol{$\Delta$}} t}}

\newcommand{\arb}{\nu} 
\newcommand{\tar}{\xi} 
\newcommand{\step}{\delta} 
\newcommand{\scl}{\varepsilon} 
\newcommand{\sscl}{\zeta} 
\newcommand{\excl}{\theta} 
\newcommand{\deloc}{\omega} 
\newcommand{\rig}{\chi} 

\title{Cusp universality for correlated random matrices}

\date{\today}

\DeclareFieldFormat*{title}{#1}
\AtEveryBibitem{\clearfield{month}}
\AtEveryBibitem{\clearfield{number}}
\AtEveryBibitem{\clearfield{publisher}}
\AtEveryBibitem{\clearfield{url}}
\DeclareFieldFormat*{eprint}{
	\hfil\penalty90\hfilneg\space \ifhyperref{
		\href{https://arxiv.org/abs/#1}{arXiv\addcolon#1} 
	}{
		arXiv\addcolon\nolinkurl{#1}
	} 
}
\DeclareFieldFormat*{doi}{
	\hfil\penalty90\hfilneg\space \ifhyperref{
		\href{https://doi.org/#1}{DOI\addcolon\addnbspace#1}
	}{
		DOI\addcolon\addnbspace\nolinkurl{#1}
	} 
}
\begin{document}
	\begin{minipage}{0.85\textwidth}	\vspace{0.5cm}
	\end{minipage}
	\begin{center}
		\large\bf Cusp Universality for Correlated Random Matrices
	\end{center}
	\vspace{0.75cm}
	
	\renewcommand{\thefootnote}{\fnsymbol{footnote}}

	\noindent
	\mbox{}%
	\hfill%
	\begin{minipage}{0.25\textwidth}
		\centering
		{L\'aszl\'o Erd\H{o}s}\footnotemark[1]~\orcidlink{0000-0001-5366-9603}\\
		\footnotesize{\textit{lerdos@ist.ac.at}}
	\end{minipage}
	\hfill%
	\begin{minipage}{0.25\textwidth}
		\centering
		{Joscha Henheik}\footnotemark[1]~\orcidlink{0000-0003-1106-327X}\\
		\footnotesize{\textit{joscha.henheik@ist.ac.at}}
	\end{minipage}
	\hfill%
	\begin{minipage}{0.25\textwidth}
		\centering
		{Volodymyr Riabov}\footnotemark[1]~\orcidlink{0009-0007-4989-7524}\\
		\footnotesize{\textit{vriabov@ist.ac.at}}
	\end{minipage}
	\hfill%
	\mbox{}%
	\footnotetext[1]{Institute of Science and Technology Austria, Am Campus 1, 3400 Klosterneuburg, Austria.  Supported by the ERC Advanced Grant "RMTBeyond" No.~101020331.
	}
	
	\renewcommand*{\thefootnote}{\arabic{footnote}}
	\vspace{0.25cm}
	
	\begin{center}
		\begin{minipage}{0.91\textwidth}\footnotesize{ {\bf Abstract.}} 
			For correlated real symmetric or complex Hermitian random matrices, we prove that the local eigenvalue statistics at any cusp singularity are universal. Since the density of states typically exhibits only square root edge or cubic root cusp singularities, our result completes the proof of the Wigner-Dyson-Mehta universality conjecture in all  spectral  regimes for a very general class of random matrices. 
			Previously only the bulk and  the  edge universality were established in this generality \cite{edgelocallaw},  while cusp universality was proven only for Wigner-type matrices with independent entries \cite{Cusp1, Cusp2}.	
			As our main technical input, we prove an optimal local law at the cusp  using the \emph{Zigzag strategy}, a recursive tandem of the characteristic flow method and a Green function comparison argument. 
			Moreover, our proof of the optimal local law holds uniformly in the spectrum, thus we also provide a significantly simplified alternative proof of the local eigenvalue universality in the previously studied bulk \cite{slowcorr} and edge \cite{edgelocallaw} regimes.
		\end{minipage}
	\end{center}

	\vspace{3mm}
	
	{\small
		\footnotesize{\noindent\textit{Date}: \today}\\
		\footnotesize{\noindent\textit{Keywords and phrases}:  Universality, cusp singularity, local law, correlated random matrix
	}
	
	\vspace{3mm}
	
	\thispagestyle{headings} 
	\normalsize

	\section{Introduction} \label{sec:intro}
	The celebrated Wigner-Dyson-Mehta (WDM) conjecture asserts that the local eigenvalue statistics of large random matrices become \emph{universal}: they depend only on the symmetry class of the matrix and not on the precise details of its distribution. This remarkable effect is extremely  robust and manifests in all spectral regimes. 
	The correlation functions of the eigenvalues are governed by one of three universal determinantal processes, whose kernel functions depend on the local shape of the eigenvalue density. 
	As  proven by Dyson, Gaudin and Mehta \cite{mehta1967random} for the Gaussian GOE/GUE ensembles, the local statistics of the eigenvalues in the \emph{bulk} of the spectrum
	are driven by the \emph{sine kernel}. 
	At the spectral edges, where the density of states vanishes like a square root, Tracy and Widom \cite{tracy1994level, tracy1996orthogonal} computed  that the correlation functions for GOE/GUE are given by  the \emph{Airy kernel}.  
	As was first observed by Wigner \cite{wigner1955characteristic}, and formalized as a  conjecture for standard Wigner matrices by Dyson and Mehta in the 1960’s, these statistics hold well beyond the Gaussian ensembles.
	After  the first proofs  for standard Wigner matrices \cite{bourgade2014edge, erdHos2010bulk, erdHos2011universality, soshnikov1999universality, TVacta, tao2010random}, these universality results in the bulk and at the edge saw rapid development and were gradually extended\footnote{
		In another direction of generalization, sparse matrices \cite{adlam2015spectral, erdHos2012spectral, huang2015bulk, lee18}, adjacency matrices of regular graphs \cite{bauerschmidt2017bulk}, band matrices \cite{bourgade2017universality, bourgade2020random, sodin2010spectral}, and dynamically defined matrices \cite{adhikari2023universal} have also been considered. In parallel to that, universal statistics in the bulk and at the edge have been established for invariant $\beta$-ensembles (see, e.g., \cite{anderson1958absence, bekerman2015transport, bourgade2014edge, bourgade2014universality, deift1999uniform, deift2007universality, krishnapur2016universality, pastur2003edge, pastur2008bulk, shcherbina2009edge, shcherbina2014change, valko2009continuum}) and their discrete analogs \cite{baik2007discrete, borodin2000asymptotics, guionnet2019rigidity, johansson2001discrete}, although  often  using   very different methods.
	} to ensembles of ever greater generality: for Wigner matrices with diagonal \cite{lee2015edge, lee2016bulk} and non-diagonal deformations \cite{knowles2017anisotropic}, Wigner-type ensembles with not necessarily identically distributed but still independent entries \cite{univWigtype}, and even to random matrices allowing for substantial correlations among the entries \cite{ajanki2019stability, slowcorr, edgelocallaw}.

	The third and final class of universal local statistics emerges at the \emph{cusp-like} singularities of the density with cubic-root behavior.  
	There, the eigenvalues form a \emph{Pearcey process}, which was first identified by Br\'{e}zin and Hikami for a Gaussian unitary (GUE) matrix with a special deterministic deformation \cite{brezin1998level, brezin1998universal}.
Compared to the bulk and edge, the cusp regime is  less understood and universality in this most delicate spectral regime was established only recently in \cite{Cusp1, Cusp2}, however only for a special class of random matrices. More precisely, these proofs were restricted to Wigner-type ensembles with  independent entries and  diagonal deformations, and did not cover the  broadest class of correlated ensembles, for which bulk and edge universality had already been proven.

	\bigskip 
	Our main result completes the picture by proving the universality of the local eigenvalues statistics at the cusp for random matrices with correlated entries and  an   arbitrary deformation,  as stated in our main result,    Theorem~\ref{thm:univ}.
	The proof follows the \emph{three-step strategy}, a general method for proving universality of local spectral statistics, summarized in \cite{erdHos2017dynamical}. 
	The first step in this strategy is the \emph{local law},  which asserts   that the resolvent $G(z) = (H-z)^{-1}$ at $z = E + \ii \eta \in \mathbb{H}$ of the random matrix $H$ concentrates around a deterministic matrix $M(z)$ as the dimension of the matrix tends to infinity.
	This concentration estimate holds for $\eta$ just above the local eigenvalue spacing at $E$, resolving   the empirical distribution of eigenvalues  at   this scale. 
	The second step is to establish universality for ensembles with a tiny Gaussian component, and the third step is a perturbative argument that removes the Gaussian component. 
	Crucially, the optimal local law is used as a key input for both the second and third steps. 
	These latter two steps have proven to be extremely robust and essentially model-independent tools \cite{slowcorr, edgelocallaw, Cusp1, Cusp2}. 
	Nevertheless, the critical first step, the proof of the \emph{local law}, remains highly model-dependent.   
	
	As our main technical result, Theorem~\ref{thm:main}, we prove the \emph{optimal average and isotropic local laws} for correlated random matrices. 
	These local laws assert that for any fixed $\tar > 0$, any deterministic matrix $B$ and test vectors $\bm x, \bm y$, the bounds 
	\begin{equation} \label{eq:LLintro}
		\big| \big\langle \big(G(z) - M(z)\big)B \big\rangle\big| \lesssim N^\tar\frac{\Vert B \Vert_{\rm hs}}{N \eta} \quad \text{and} \quad \big|  \big(G(z) - M(z)\big)_{\bm x \bm y}\big| \lesssim N^\tar\sqrt{\frac{\rho(z)}{N \eta}} \Vert \bm x \Vert \, \Vert \bm y \Vert
	\end{equation}
	hold with very high probability.	
	Here $N$ is the dimension of the random matrix $H$, $\langle \cdot \rangle := N^{-1}\Tr[\cdot]$ denotes the normalized trace, and $\rho(z): = \pi^{-1} \langle \Im M(z) \rangle > 0$ is the \emph{self consistent density of states.} 
	Moreover, Theorem~\ref{thm:main} provides further optimal improvements to the right-hand sides of \eqref{eq:LLintro} for spectral parameters $z = E + \ii \eta$ with energy $E$ outside of the self-consistent spectrum.  
	We point out that the local laws in \eqref{eq:LLintro} are optimal in terms of their dependence on $\rho(z)$ and the (normalized) Hilbert--Schmidt norm of the observable matrix $B$. 
	In many cases, such as for low-rank observables, the Hilbert--Schmidt norm $\Vert B \Vert_{\rm hs}:= \langle BB^*\rangle^{1/2}$ is much smaller than the operator norm $ \Vert B \Vert$, which has  traditionally been used in previous single-resolvent local laws  \cite{slowcorr, edgelocallaw, Cusp1}.
	Thus, our local law \eqref{eq:LLintro} unifies and improves upon the previous local laws, even in the Wigner-type case.

	\bigskip
	Traditional  proofs of the local laws relied on solving an approximate self-consistent equation for the difference $G-M$.
	They consisted of two parts: a stability analysis of the underlying deterministic \emph{Dyson equation} and a probabilistic estimate on the fluctuations. 
	Both steps become quite cumbersome beyond the simple Wigner matrices. 
	In particular, for  general Wigner-type \cite{univWigtype, Cusp1} and correlated random matrices \cite{slowcorr, edgelocallaw}, the stability analysis became intricate \cite{ajanki2019quadratic, AEK2020}, and 
	the probabilistic part relied on sophisticated Feynman graph expansions.
	Recently, a completely new approach, the \emph{Zigzag strategy}  \cite{mesoCLT, Cipolloni2023Edge, cipolloni2023universality, OTOC, WigTypeETH, edgeiid}, has been developed.
	This approach consists of an  iterated application of two steps in tandem (cf.~Figure \ref{fig:zigzag_fig} below): the \emph{characteristic flow method} \cite{bourgade2021extreme, adhikari20, huang2019rigidity, landon2024single, adhikari2023local, landon2022almost, aggarwal2023edge}, coined the \emph{zig-step}, and a Green function comparison (GFT) argument driven by   an Ornstein-Uhlenbeck flow, called the \emph{zag-step}.
	Remarkably,   the Zigzag strategy circumvents many of the  difficulties that arise along  the more traditional local law proofs.  
	It even removes the   key obstacles that previously hindered the proof of the optimal local law at the cusp for the most general correlated matrices. 
	We now explain this crucial aspect in more detail.

	For traditional proofs of the local laws, the bulk regime is the easiest since the  underlying  Dyson equation is stable when $\rho(z)$ is separated away from zero. 
	In the regime where the density $\rho(z)$ vanishes, this stability deteriorates -- specifically, the corresponding stability factor behaves like $\rho(z)^{-1}$ at a square-root edge and as $\rho(z)^{-2}$ at a cubic-root cusp.
	This blow-up had to be compensated by a fine control on the error term in the approximate Dyson equation. 
On the probabilistic side, obtaining the optimal very-high-probability estimate on the fluctuation error required a high moment calculation that exploited various \emph{fluctuation averaging} mechanisms, even in the simplest bulk regime.  
	In the edge regime, an additional factor $\rho(z)$ needed to be extracted, which essentially relied on the emergence of the imaginary part of the resolvent via the \emph{Ward identity}, $GG^* = \Im G/\eta$. 
	However, for cusp singularities, an additional \emph{second order} cancellation  effect was necessary.
	This delicate effect, coined the \emph{cusp fluctuation averaging} \cite{Cusp1}, arises from a finite set of critical Feynman subdiagrams, called  the   \emph{$\sigma$-cells}. 
	Roughly speaking, a $\sigma$-cell consists of four resolvents interconnected through the deterministic approximation $M$   and the correlation four-tensor of the matrix elements. 
	In the case of Wigner-type matrices with diagonal deformations, $M$ becomes a \emph{diagonal} matrix, leading to a simplification of the original \emph{matrix} Dyson equation into a \emph{vector} equation.
	Moreover,  since the entries of a Wigner-type matrix are independent, the correlation tensor  is reduced to a matrix acting on the diagonal.   
	These substantial simplifications facilitated the intricate extraction of $\sigma$-cells,  effectively capturing the second order cancellation effect.
	Identifying the analog of the $\sigma$-cells for correlated matrices, when $M$ is no longer diagonal and the correlation is a full-fledged four-tensor remains out of reach.

	In this paper, we leverage   the \emph{Zigzag strategy} to conveniently avoid the complicated graphical expansions and, more importantly, circumvent the extraction of $\sigma$-cells. The only stability input required is a trivial bound of the form $\rho(z)/\eta$, that is precisely tracked by the Ward identity. 
	The characteristic flow at   the heart of the Zigzag strategy has previously proved itself to be effective in dealing with a \emph{first order} blow-up of the stability factor, such as at the edge of Wigner matrices \cite{Cipolloni2023Edge}, and in capturing the $z_1 - z_2$ decorrelation effect for the Hermitizations of non-Hermitian i.i.d.~matrices \cite{cipolloni2023universality, nonHermdecay}. 
	The current work demonstrates that the Zigzag strategy is even capable of circumnavigating general \emph{second order} instabilities arising at the cusp. 
	Evidence of this feature of the characteristic flow has already been observed for unitary Brownian motion \cite{adhikari2023universal} and in a special non-Hermitian setting \cite{edgeiid}, where an additional symmetry was available.
	
	Besides unraveling this remarkable power of the Zigzag approach in full generality, our paper is the first to implement the method in a correlated setting,
	 which requires adjustments to the Zigzag dynamics. 
	The GFT argument at the core of the zag step requires an a-priori bound on the resolvent as an input, which typically stems from a single resolvent local law. This, however, would render our argument circular. Hence, to remedy the situation, we augment the zag step with an internal induction\footnote{This argument is reminiscent of \cite{knowles2017anisotropic} and we also refer to \cite{taovuGFT} for an alternative approach. }  (\emph{bootstrap}) in $\eta$.
	Furthermore, our result has two additional features: (i) for the averaged law in \eqref{eq:LLintro}, we obtain the optimal estimate on the observable $B$  in terms of its Hilbert--Schmidt norm, and (ii) we extend the Zigzag approach beyond the typical \emph{above the scale}  regime of  $N \eta \rho(z) \ge  N^{\varepsilon} $ (see Section~\ref{sec:outside}).
	We emphasize that, in addition to covering the missing cusp regime, our proof also provides a unified approach to optimal local laws for the most general class of random matrices with correlated entries, completely eliminating any dependence of the proof on the specific spectral regime.
	The price we pay for our simple and self-contained Zigzag proof of the local law is assuming \emph{fullness} of the correlated random matrix (cf.~Assumption~\ref{ass:full}), rather than the slightly weaker \emph{flatness} condition (cf.~\cite[Assumption~(E)]{slowcorr}). 
	However, this stronger assumption is justified because fullness is necessary for deducing universality using the three-step strategy, regardless of how the local law is proven.

	\subsection*{Notations and conventions}
	
	We use the notation $[N]$ to represent the index set $\{1,\dots, N\}$. The letters $a$, $b$, $j$, and $k$ are used to denote integer indices, while $\alpha$ (with various subscripts) denotes elements of $[N]^2$. All unrestricted summations of the form $\sum_a$ and $\sum_{\alpha}$ are understood to run over $a \in [N]$ and $\alpha \in [N]^2$, respectively. 
	
	We denote vectors in $\mathbb{C}^{N\times N}$ using boldface letters, e.g., $\bm x$. The scalar product on $\mathbb{C}^N$ is defined by $\langle \bm x, \bm y\rangle := \sum_{j=1}^N \overline{x_j}y_j$, and the corresponding Euclidean norm is denoted by $\norm{\bm x} := \langle \bm x, \bm x\rangle^{1/2}$.
	
	Matrices are denoted by capital letters. Unless explicitly stated otherwise, all matrices we consider are $N\times N$. For a matrix $A \in \C^{N \times N}$, the angle brackets $\langle A\rangle := N^{-1}\Tr[A]$ denote its normalized trace. We use the following notations for the matrix norms:
	\begin{equation*}
		\norm{A}_{\max} := \max_{a,b} |A_{ab}|, \quad \norm{A} := \sup_{\norm{\bm x}=1} \norm{A\bm x}, \quad \norm{A}_{\mathrm{hs}} := \bigl\langle |A|^2 \bigr\rangle^{1/2},
	\end{equation*}  
	where $|A|^2 := AA^*$. Furthermore, for any $a\in [N]$ and vectors $\bm x$ and $\bm y$, we use the following notation:
	\begin{equation*}
		A_{\bm x \bm y} := \langle \bm x, A\bm y\rangle, \quad A_{\bm x a} := \langle \bm x, A \bm{e}_a\rangle, \quad A_{a \bm y} := \langle\bm{e}_a , A \bm y\rangle,
	\end{equation*}
	where $\bm{e}_a$ is the standard $a$-th basis vector of $\mathbb{C}^N$.
	We denote the complex upper half-plane by $\mathbb{H}$, that is, $\mathbb{H} := \{z \in \mathbb{C} : \Im z > 0\}$, and its closure by $\overline{\mathbb{H}} := \mathbb{H}\cup\mathbb{R}$.  For a complex number $z \in \mathbb{C}$, we use the notation $\langle z \rangle := 1 +|z|$.

	We use $c$ and $C$ to denote unspecified, positive constants—small and large, respectively—that are independent of $N$ and may change from line to line. Various tolerance exponents are denoted by Greek letters such as $\scl, \tar, \step, \sscl, \mu, \arb$.
	The notation $\tar \ll \scl$ means that there exists a small absolute constant $c > 0$ such that $\tar \le c \scl$.  We use $\arb > 0$ to denote arbitrary small tolerance exponents.
	
	For two positive quantities $\mathcal{X}$ and $\mathcal{Y}$, we write $\mathcal{X} \lesssim \mathcal{Y}$ if there exists a constant $C > 0$ that depends only on the \emph{model parameters} in Assumptions~{\ref{ass:boundedexp}--\ref{ass:Mbdd}} (unless explicitly stated otherwise), such that $\mathcal{X} \le C \mathcal{Y}$. We use the notation $\mathcal{X} \sim \mathcal{Y}$ if both $\mathcal{X} \lesssim \mathcal{Y}$ and $\mathcal{Y} \lesssim \mathcal{X}$ hold. For an arbitrary quantity $\mathcal{X}$ and a positive quantity $\mathcal{Y}$, we use the notation $\mathcal{X} = \mathcal{O}(\mathcal{Y})$ to indicate that $|\mathcal{X}|\lesssim \mathcal{Y}$.
	
	Let $\Omega := \{\Omega^{(N)}(u) \, \lvert\, N\in\mathbb{N},\, u\in\mathcal{U}^{(N)} \}$  be a family of events depending on $N$ and possibly on a parameter $u$ that varies over some parameter set $\mathcal{U}^{(N)}$. We say that $\Omega$ holds \textit{with very high probability} (w.v.h.p.) uniformly in $u \in \mathcal{U}^{(N)}$ if, for any $D > 0$,
	\begin{equation*}
		\sup\limits_{u \in \mathcal{U}^{(N)}} \mathbb{P}\bigl[\Omega^{(N)}(u)\bigr] \ge 1 - N^{-D},
	\end{equation*}
	for any $N \ge N_0(D)$. We often discard the explicit dependence of $\Omega^{(N)}$ and $\mathcal{U}^{(N)}$ on $N$, and simply refer to $\Omega$ as a very-high-probability event. A bound is said to hold w.v.h.p. if it holds on a very-high-probability event.

	\section{Main results} \label{sec:main}
	We consider real symmetric or complex Hermitian random matrices $H$ of the form 
	\begin{equation} \label{eq:HAW}
		H = A + W\,, \qquad \E W = 0\,, 
	\end{equation}
	where $A \in \C^{N \times N}$ is a bounded deterministic matrix (cf.~Assumption~\ref{ass:boundedexp} below) and $W$ has sufficiently fast decaying correlations between its matrix elements (cf.~Assumption~\ref{ass:cumulants} below). 
	
	For any random matrix $H$, we define the \textit{self-energy operator} $\mathcal{S}_H$ corresponding to $H$  by its action on any deterministic matrix $X \in \mathbb{C}^{N\times N}$,
	\begin{equation} \label{eq:self_energy_def}
		\mathcal{S}_H[X] := \E \bigl[(H-\E H)X(H-\E H)\bigr].
	\end{equation}
	
	The Matrix Dyson Equation (MDE) with a {\it data pair} $(A, \mathcal{S})$ is given by
	\begin{equation} \label{eq:MDE}
		- M(z)^{-1} = z - A + \mathcal{S}\bigl[M(z)\bigr] \,
	\end{equation}
	for the unknown matrix valued function $M(z)$, $z\in \C\setminus \R$.
	It is well known (Theorem  2.1 \cite{ajanki2019quadratic})  that the MDE
	has a unique solution under the constraint that $( \Im z)\Im M(z) > 0$, where $\Im M = \tfrac{1}{2\ii}(M-M^*)$.
	The corresponding \textit{self-consistent density of states} (scDOS) $\rho$ is a probability density function on the real line defined via the Stieltjes inversion formula, 
	\begin{equation} \label{eq:scDOS}
		\rho(x) := \lim_{\eta\to+0} \frac{1}{\pi}\bigl\langle \Im M(x+\ii\eta) \bigr\rangle.
	\end{equation} 
	We define $\rho(z) := \pi^{-1} \langle \Im M(z) \rangle$ to be the harmonic extension of the scDOS to the complex upper-half plane. With a slight abuse of notation, we also refer to $\rho(z)$ as scDOS. As shown in \cite{AEK2020}, under suitable assumptions (which are formulated precisely in Section~\ref{subsec:assump} below) on the data pair $(A, \mathcal{S})$ and the solution $M$ of the MDE \eqref{eq:MDE}, the scDOS $\rho$ is $1/3$-Hölder continuous. Furthermore, the set where the scDOS is positive, $  \{x \in \R : \rho(x) > 0\}$, splits into finitely many connected components, that are called \emph{bands}. Inside the bands, the density is real-analytic with a square root growth behavior at the \emph{edges}. If two bands touch, however, a cubic root \emph{cusp} emerges. These are the only two possible types of singularities. Precise universal asymptotic formulas in the \emph{almost cusp regime} are given, e.g., in \cite[Eqs.~(2.4a)--(2.4e)]{Cusp1}.  
	
	As the main result of this paper, Theorem~\ref{thm:univ}, we show the universality of the local eigenvalue statistics of correlated real symmetric and complex Hermitian random matrices at cusp-like singularities. As mentioned in the introduction, the proof of cusp universality follows the \emph{three-step strategy} \cite{erdHos2017dynamical}, the first step of which is a \emph{local law} (see Theorem~\ref{thm:main}) identifying the empirical eigenvalue distribution on a scale slightly above the typical eigenvalue spacing, with very high probability. After precisely formulating the assumptions that we impose on the random matrix \eqref{eq:HAW} in Section~\ref{subsec:assump}, we present our novel local law in Section~\ref{subsec:locallaw}. Afterwards, in Section~\ref{subsec:corollaries}, we formulate our main result on cusp universality and other consequences of the local law, such as eigenvector delocalization and eigenvalue rigidity. 
	
	\subsection{Assumptions} \label{subsec:assump}
	In this section, we precisely formulate the assumptions, under which our main result, Theorem~\ref{thm:main}, holds, and comment on them. 
	
	\begin{assumption}[Bounded expectation] \label{ass:boundedexp}
		There exists a constant $C_A > 0$ such that $\Vert A \Vert \le C_A$, uniformly in $N$. 
	\end{assumption}
	\begin{assumption}[Finite moments] \label{ass:moments}
		For every $p \in \N$, there exists a constant $\mu_p$ such that $\E |\sqrt{N} w_\alpha|^p \le \mu_p$ for all $\alpha \in [N]^2$.
	\end{assumption}
	
	Before formulating our assumption on the correlation structure of the random matrix $W$, we introduce some custom notation to keep the definition of the norms of the (normalized) \emph{cumulants}\footnote{Let $\bm{w} = (w_1, ..., w_k)$ be a random vector. Recall that its joint cumulants, $\kappa_{\bm m}$ with $\bm m \in \N_0^k$, are traditionally given as the coefficients of the $\log$-characteristic function
		\begin{equation*}
			\log \E \ee^{\ii \bm{w} \cdot \bm t} = \sum_{\bm{m}} \kappa_{\bm m}  \frac{(\ii \bm t)^{\bm m}}{\bm m!} \,. 
		\end{equation*}
		For $\bm w = (\sqrt{N}w_{\alpha_1}, ... , \sqrt{N}w_{\alpha_k})$ we use the notation $\kappa(\alpha_1, ..., \alpha_k) \equiv \kappa(\sqrt{N}w_{\alpha_1}, ... , \sqrt{N}w_{\alpha_k}) := \kappa_{(1,...,1)}$ %
		and note that, by construction, $\kappa(\alpha_1, ... , \alpha_k)$ is invariant under permutations of its arguments. For example, for $k=2$, $\kappa(\alpha_1, \alpha_2) = N \E [w_{\alpha_1} w_{\alpha_2}]$.
	},
	\begin{equation} \label{eq:cumulants}
		\kappa(\alpha_1, ... , \alpha_k) \equiv \kappa(\sqrt{N}w_{\alpha_1}, ... , \sqrt{N}w_{\alpha_k})\,,
	\end{equation}
	relatively compact. If, instead of an index $a \in [N]$, we write a dot $(\cdot)$ in a scalar quantity, then we consider it as an $N$-vector indexed by the coordinate in place of the dot. As an example, $\kappa(a_1 \cdot, a_2 b_2)$ is an $N$-vector, whose $i$-entry is $\kappa(a_1 i, a_2 b_2)$ and $\Vert \kappa(a_1 \cdot, a_2 b_2)\Vert$ is its Euclidean (vector) norm. Similarly, $\Vert X(*,*)\Vert $ refers to the operator norm of the $N^2\times N^2$ matrix with entries $X(\alpha_1,\alpha_2)$. We also introduce a combination of these conventions. In particular, $\big\Vert \Vert \kappa(\bm x *, \cdot *)\Vert \big\Vert $ denotes the operator norm $\Vert Y \Vert$ of the matrix $Y$ with entries $Y(i,j) = \Vert \kappa(\bm x i, \cdot j) \Vert = \Vert \sum_a x_a \kappa(ai, \cdot j)\Vert $. Since the operator norm is invariant under transposition of the matrix, this does not lead to ambiguity regarding the order of $i$ and $j$. Note that we use dot $(\cdot)$ as a placeholder for the variable related to the inner norm, and star $(*)$ for the outer norm. 
	
	The following assumption on the correlation structure of $W$ is formulated in the real symmetric case. For complex Hermitian matrices, we require the cumulant norms introduced below to be bounded for all choices of real and imaginary in each of the arguments of a cumulant, i.e.~for $\kappa(\alpha_1^{\mathfrak{X}_1}, ... ,  \alpha_k^{\mathfrak{X}_k}) = \kappa(\sqrt{N} \mathfrak{X}_1 w_{\alpha_1}, ... , \sqrt{N} \mathfrak{X}_k w_{\alpha_k})$ and all choices of $\mathfrak{X}_i \in \{\Re , \Im \}$ (see~\cite[Appendix C]{slowcorr} for a more detailed discussion). %
	\begin{assumption}[Correlation structure] \label{ass:cumulants}
		The correlations among the matrix entries $(w_\alpha)_\alpha$ of $W$ satisfy the following. 
		\begin{itemize}
			\item[(i)] The {cumulants} $\kappa(\alpha_1, ... , \alpha_k)$ have bounded  %
			matrix norms (viewed as an $N^2 \times N^2$ matrix), 
			i.e.~for all $k \ge 2$ there exists a constant $C_k > 0$ such that\footnote{ We remark that the constants $C_k$ in the bounds \eqref{eq:summcum}--\eqref{eq:kappa_3_av_norm} could also be replaced by $C_{k, \arb} N^\arb$ for any $\arb > 0$, where $C_{k,\arb}$ is a positive constant. All our proofs hold under this more general condition, but we omit it for simplicity.}
			\begin{equation} \label{eq:summcum}
				\vertiii{\kappa}_k := \bigg\Vert \sum_{\alpha_1, ... , \alpha_{k-2}} |\kappa(\alpha_1, ... , \alpha_{k-2}, *, *) | \bigg\Vert \le C_k \,. 
			\end{equation}
			Moreover, we suppose that 
			\begin{equation} \label{eq:summcum3c}
				\vertiii{\kappa}_{2}^\mathrm{iso} := \inf_{\kappa = \kappa_{\mathrm{c}} + \kappa_{\mathrm{d}}} \bigl( \vertiii{\kappa_{\mathrm{c}}}_c + \vertiii{\kappa_{\mathrm{d}}}_d\bigr) \le C_2 \,, 	%
			\end{equation}
			where the infimum is taken over all decompositions of $\kappa$ in two functions $\kappa_{\rm c}, \kappa_{\rm d}$,  %
			where the subscripts stand for ``direct" and ``cross" (see~\cite[Remark~2.8]{slowcorr} for an explanation of this terminology) and the corresponding norms are defined as  %
			\begin{equation*}
				\vertiii{\kappa}_d := \sup_{\norm{\bm x} \le 1} \bigl\lVert \, \lVert \kappa(\bm x* ,\cdot *) \rVert  \, \bigr\rVert, \quad \text{and} \quad   \vertiii{\kappa}_c := \sup_{\norm{\bm x} \le 1} \bigl\lVert \, \lVert \kappa(\bm x* , * \cdot) \rVert  \, \bigr\rVert \,. 
			\end{equation*}
		Finally, we assume that
		\begin{equation} \label{eq:kappa_3_av_norm}
			\vertiii{\kappa}_3^\mathrm{av} :=  N^{-3/2} \sup\limits_{\substack{X,Y,Z\in\C^{N\times N} : \\ \Vert X \Vert, \Vert Y \Vert \le 1\,, \ \Vert Z \Vert_{\rm hs} \le 1}}  \ \sum\limits_{ab}\sum\limits_{a_1b_1}\sum\limits_{a_2b_2} \bigl\lvert \kappa(ab,a_1b_1,a_2b_2) \bigr\rvert |X_{b_1a_2}|\, |Y_{b_2a_3}|\, |Z_{b_3a_1}| \le C_3\,. 
		\end{equation}
		
		\item[(ii)] There exists a positive $\mu > 0$, such that for every $\alpha$ there exists an index set $\mathcal{N}(\alpha)$ of cardinality $|\mathcal{N}(\alpha)| \le N^{1/2 - \mu}$ with the property that $w_\alpha \perp w_\beta$ for all $\beta \notin \mathcal{N}(\alpha)$. That is, every element is correlated with at most $N^{1/2-\mu}$ other matrix elements and is independent of the rest.
	\end{itemize}
\end{assumption}
The first part of Assumption~\ref{ass:cumulants} is needed to control every finite order term in a cumulant expansion in Proposition~\ref{prop:cumex}, analogously to Assumption (C) in \cite{slowcorr}. The condition in \eqref{eq:kappa_3_av_norm} is needed only since we are dealing with Hilbert--Schmidt norm error terms and thus did not appear in \cite{slowcorr}, where the observables were bounded in terms of their operator norm. In Example~\ref{ex:tree} below, we present a prototypical class of models with a polynomially decaying metric correlation structure satisfying Assumption~\ref{ass:cumulants}~(i). Complementary to Assumption~\ref{ass:cumulants}~(i), the only purpose of the second part of Assumption~\ref{ass:cumulants} is to ensure that the cumulant expansion can be truncated. In \cite{slowcorr}, this was guaranteed by a more complicated and slightly more general condition on the correlation decay (cf.~\cite[Assumption (D)]{slowcorr}). We believe, however, that our proof of Theorem~\ref{thm:main} works under this condition as well, but we refrain from doing so for brevity. 

\begin{assumption}[Fullness] \label{ass:full}
	We say that a random matrix $H$ satisfies the \textit{fullness} condition with a constant $c > 0$ if
	\begin{equation} \label{eq:fullness}
		N\, \E  \bigl[|\Tr [(H - \E H)X]|^2\bigr] \ge c \, \Tr [X^2],
	\end{equation}
	for any deterministic matrix $X$ of the same symmetry class as $H$ (real symmetric or complex Hermitian). 	
	
	We assume that there exists a constant $c_{\rm full} > 0$ such that the random matrix $H$ satisfies the fullness condition as in \eqref{eq:fullness} with the constant $c := c_{\rm full}$.
\end{assumption}

\begin{assumption}[Bounded self-consistent Green function] \label{ass:Mbdd} Fix $C_M, c_M > 0$ and define the set of \emph{admissible energies} as 
	\begin{equation} \label{eq:admE}
		\mathcal{I} \equiv \mathcal{I}_{C_M, c_M} := \{ e \in \R  :  \Vert M(z) \Vert \le C_M\langle z \rangle^{-1} \quad \text{for all} \quad z \in \C \quad \text{with} \quad \Re z \in [e- c_M, e+ c_M] \} \,. 
	\end{equation}
	We assume that $\mathcal{I} \neq \emptyset$. 
\end{assumption}

Recall that we refer to the constants in Assumptions~\ref{ass:boundedexp}--\ref{ass:Mbdd} as \textit{model parameters}.

\begin{example}[Polynomially Decaying Metric Correlation Structure] \label{ex:tree} 
	A prime example of correlated random matrix satisfying the Assumption~\ref{ass:cumulants} (i) is the polynomially decaying model. For second order cumulants, we assume that 
	\begin{subequations}
		\begin{equation} \label{eq:tree_kappa2}
			\bigl\lvert \kappa(a_1b_1, a_2b_2) \bigr\rvert \le \frac{C_2}{1 + d(a_1b_1, a_2b_2)^s},
		\end{equation}
		for some $s > 2$, where we define the distance $d$ on the set of labels $[N]^2$ as
		\begin{equation} \label{eq:label_dist}
			d (a_1b_1, a_2b_2 ) := \min\bigl\{|a_1-a_2| + |b_1-b_2|, |a_1-b_2| + |b_1-a_2|\bigr\}.
		\end{equation}
		For cumulants of order $k\ge 3$, we assume the following decay condition
		\begin{equation} \label{eq:tree_kappak}
			\bigl\lvert \kappa(\alpha_1, \dots, \alpha_k) \bigr\rvert \le C_k \prod_{e \in \mathfrak{T}_{\mathrm{min}}}\frac{1}{1+d(e)^s},
		\end{equation}
		where $\mathfrak{T}_{\mathrm{min}}$ is a minimal spanning tree, i.e., a spanning tree for which the sum of the edge weights is minimal,
		in a complete graph with vertices $\alpha_1, \alpha_2, \dots, \alpha_k$ and edge weights induced by the distance $d$, defined in \eqref{eq:label_dist}.
	\end{subequations}
	The validity of \eqref{eq:summcum}--\eqref{eq:summcum3c} was asserted in Example 2.10 of \cite{slowcorr}, and we verify the new condition \eqref{eq:kappa_3_av_norm} in Appendix~\ref{app:tree}.  
\end{example}

\subsection{Local law} \label{subsec:locallaw}
In this section, we formulate our main technical result, the optimal local laws in Theorem~\ref{thm:main}. These show that $G(z) = (H-z)^{-1}$ is very well approximated by $M(z)$ in the $N \to \infty$ limit, with optimal convergence rate even at all singular points of the scDOS down to the typical eigenvalue spacing. We now define the scale on which the eigenvalues are predicted to fluctuate around a given energy $e_0$.   
\begin{definition}[Local fluctuation scale]
	Let $e_0 \in \mathcal{I}$ be an admissible energy.  We define the self-consistent \emph{fluctuation scale} $\etaf = \etaf(e_0) > 0$ (indicated by subscript $\mathfrak{f}$) at energy $e_0$ via
	\begin{equation} \label{eq:etafdef}
		\int_{-\etaf}^{\etaf} \rho(e_0 + x) \dif x = \frac{1}{N}\,,
	\end{equation}
	if $e_0 \in \supp \rho$. In case that $e_0 \notin \supp \rho$, we define $\etaf$ as the fluctuation scale at a nearby edge. More precisely, let $I$ be the largest interval with $e_0 \in I \subset \R\setminus \supp \rho$ and set $\Delta := \min\{ |I|,1 \}$. Then, $\etaf$   satisfies the scaling relation  
	\begin{equation} \label{eq:etafdefoutside}
		\etaf \sim \begin{cases}
			N^{-2/3}\Delta^{1/9} \quad &\text{if} \quad \Delta > N^{-3/4} \\
			N^{-3/4} \quad &\text{if} \quad \Delta \le N^{-3/4}  \,. 
		\end{cases} 
	\end{equation}
\end{definition}
While for $e_0$ in the \emph{bulk}, where the scDOS satisfies $\rho \sim 1$, we have $\etaf \sim N^{-1}$, it holds that $\etaf \sim N^{-2/3}$ at a regular \emph{edge} and $\etaf \sim N^{-3/4}$ at an exact \emph{cusp}. 

\begin{theorem}[Optimal Local Laws] \label{thm:main}
	Fix small $N$-independent constants $\scl_0, \tar_0 > 0$.
	Let $H \in \C^{N \times N}$ be a real symmetric or complex Hermitian correlated random matrix.
	Suppose that Assumptions~\ref{ass:boundedexp}--\ref{ass:Mbdd} are satisfied, and let $\mathcal{I}$ be the set of admissible energies from \eqref{eq:admE}. 
	Then, uniformly for all $z \in \mathbb{H}$ with $\Re z \in \mathcal{I}$
	and $\dist(z, \supp \rho) \in [N^{\scl_0} \etaf(\Re z), N^{D}]$, the resolvent $G(z) := (H-z)^{-1}$ satisfies the \emph{optimal isotropic local law},
	\begin{subequations}
		\begin{equation} \label{eq:ISOLL}
			\bigl\lvert \bigl(G(z) - M(z)\bigr)_{\bm x \bm y}\bigr\rvert  \le N^{\tar_0} \sqrt{\frac{\rho(z)}{\langle z \rangle^2 N \eta}} \Vert \bm x \Vert \, \Vert \bm y \Vert \,,
		\end{equation}
		for any deterministic vectors $\bm x, \bm y \in \C^N$, and the \emph{optimal average local law},   
		\begin{equation} \label{eq:AVLL}
			\big| \big\langle \big(G(z) - M(z)\big)B\big\rangle\big| \le \frac{N^{\tar_0}}{\langle z \rangle N \dist\bigl(z, \supp \rho\bigr)}\Vert B \Vert_{\rm hs}\,,
		\end{equation}
		for any deterministic matrix $B \in \C^{N \times N}$, both with very high probability.
	\end{subequations}
\end{theorem}

\subsection{Delocalization, rigidity, and universality} \label{subsec:corollaries} 
The local law in Theorem~\ref{thm:main} is the main input for eigenvector delocalization, eigenvalue rigidity, and universality, as stated below. While Corollaries~\ref{cor:deloc}--\ref{cor:rig} and Theorem~\ref{thm:univ} are proven as corollaries to Theorem~\ref{thm:main} in Section~\ref{subsec:proofcor}, the exclusion of eigenvalues outside the support of the scDOS in Theorem~\ref{thm:noeig} is obtained alongside the proof of Theorem~\ref{thm:main} and presented in Section~\ref{sec:outside}. 
\begin{theorem}[No eigenvalues outside the support of the scDOS] \label{thm:noeig}
	Under the assumptions of Theorem~\ref{thm:main} we have the following: Let $e_0 \in \mathcal{I} \setminus \supp \rho$. There exists a constant $c > 0$ such that for any fixed small $N$-independent constant $\excl_0 > 0$
	\begin{equation} \label{eq:noeig}
		\dist\big(\spec H \cap [e_0 - c, e_0 + c], \supp  \rho \big) \le N^{\excl_0}\etaf(e_0),
	\end{equation}
	with very high probability. Here we use the convention that $\dist(\emptyset, ... ) = 0$. 
\end{theorem}

\begin{corollary}[Eigenvector delocalization] \label{cor:deloc}
	Let $\bm u_i \in \C^N$ with $\Vert \bm u_i \Vert = 1$ be a normalized eigenvector of $H$ corresponding to the eigenvalue $\lambda_i$.  Then, under the assumptions of Theorem~\ref{thm:main}, for any small $N$-independent constant $\deloc_0 > 0$, the estimate
	\begin{equation}
		\max_{\substack{i \in [N] : \\ \lambda_i \in \mathcal{I}}} \big| \langle \bm x, \bm u_i \rangle \big| \le \frac{N^{\deloc_0}}{\sqrt{N}}
	\end{equation}
	holds with very high probability, uniformly in deterministic vectors $\bm x \in \C^N$ with $\Vert \bm x \Vert = 1$. 
\end{corollary}

\begin{corollary}[Band rigidity and eigenvalue rigidity] \label{cor:rig}
	Assume the conditions of Theorem~\ref{thm:main} with $\mathcal{I} = \R$ in Assumption~\ref{ass:Mbdd}. Then, the following holds. 
	\begin{itemize}
		\item[(a)] %
		For any $\excl > 0$, whenever $e_0 \in \R \setminus \supp \rho$ with $\dist(e_0, \supp \rho) \ge  N^{\excl} \etaf(e_0)$, the number of eigenvalues less than $e_0$ is deterministic with high probability. 
		More precisely, 
		\begin{equation} \label{eq:intmass}
			\big|\spec H \cap (- \infty, e_0) \big| = N \int_{- \infty}^{e_0} \rho(x) \dif x, \quad \text{ w.v.h.p}.
		\end{equation}
		\item[(b)] Let $\lambda_1 \le ... \le \lambda_N$ denote the ordered eigenvalues of $H$ and assume that $e_0 \in \mathrm{int} (\supp \rho)$. Then, for any small $N$-independent constant $\rig_0 > 0$, it holds that
		\begin{equation}
			\big| \lambda_{k(e_0)}  - e_0\big| \le N^{\rig_0}\etaf(e_0)\,,
		\end{equation}
		with very high probability, where we defined the (self-consistent) \emph{eigenvalue index} as $k(e_0) := \lceil N \int_{- \infty}^{e_0} \rho(x) \dif x \rceil$. 
	\end{itemize}
\end{corollary}

\begin{remark}[Integer mass] \label{rmk:intmass}
	We point out that  \eqref{eq:intmass} entails the nontrivial fact that, whenever $e_0 \notin \supp \rho$ satisfies  $\dist(e_0, \supp \rho) \ge N^{\excl} \etaf(e_0)$ for some $\excl > 0$, the integral $N \int_{- \infty}^{e_0} \rho(x) \dif x$ is always an integer. An immediate consequence is that, for each connected component $[a,b]$ of $\supp \rho$, it holds that $N\int_{a}^b \rho(x) \dif x$ is an integer. That is, each \emph{spectral band} contains that number of eigenvalues with very high probability. For spectral bands which are separated by a distance of order one, this was previously shown in \cite[Corollary 2.9]{edgelocallaw}. Our Corollary~\ref{cor:rig} improves this to the optimal minimal distance $N^\epsilon \etaf(e_0)$. %
\end{remark}

As our last consequence to the optimal local laws in Theorem~\ref{thm:main}, we prove cusp universality in Theorem~\ref{thm:univ} below. Since universality is already known in the bulk \cite{slowcorr} as well as the edge regime \cite{edgelocallaw}, we will henceforth focus on the (approximate) cubic-root cusp. However, the optimal local laws of Theorem \ref{thm:main} can be used as an input for the three-step strategy to yield bulk and edge universality as well. 
From the in-depth analysis of the MDE \eqref{eq:MDE} and its solution in \cite{AEK2020}, we know that the scDOS $\rho$ is described by explicit universal shape functions in the vicinity of local minima with a small value of $\rho$ and near small gaps in the support of $\rho$; see, e.g., \cite[Eqs.~(2.4a)--(2.4e)]{Cusp1} for precise formulas. 

Whenever the local length scale of such an almost cusp shape around a point $\mathfrak{b}$ matches (or is smaller than) the local eigenvalue spacing, i.e.~if  $\mathfrak{b}$ is a small local minimum, satisfying $\rho(\mathfrak{b}) \lesssim N^{-1/4}$, %
or   a midpoint of a gap with width $\Delta \lesssim N^{-3/4}$,   %
then we call the local shape   around $\mathfrak{b}$   a \emph{physical cusp} -- reflecting the fact that   it   %
becomes indistinguishable from an exact cusp when resolved with a precision (slightly) above the local eigenvalue spacing $\sim N^{-3/4}$. In this case, $\mathfrak{b}$ is called a \emph{physical cusp point}.
Besides the local length scale of a physical cusp point   $\mathfrak{b}$,   the specific shape of the scDOS   around $\mathfrak{b}$   is characterized by a single additional parameter $\gamma > 0$, called the \textit{slope parameter}. 

In order to formulate our result on cusp universality in Theorem~\ref{thm:univ}, it is natural to consider the rescaled $k$-point function $p_k^{(N)}$, which is implicitly defined as
\begin{equation} \label{eq:kptfunct}
	\E \, \binom{N}{k}^{-1} \sum_{\{j_1, ... , j_k\} \subset [N]} f(\lambda_{j_1}, ... , \lambda_{j_k}) =: \int_{\R^k} f(\bm x) p_k^{(N)}(\bm x) \, \dif \bm x \,,
\end{equation}
for any test function $f$. Here, the summation is over all distinct subsets of $k$ integers from $[N]$.

\begin{theorem}[Cusp universality for correlated random matrices] \label{thm:univ} Let $H \in \C^{N\times N}$ be a real symmetric or complex Hermitian correlated random matrix as in \eqref{eq:HAW}. Suppose that Assumptions~\ref{ass:boundedexp}--\ref{ass:Mbdd} are satisfied, assume that a \emph{physical cusp point} $\mathfrak{b} \in \mathcal{I}$ lies in the set of admissible energies \eqref{eq:admE}, and let $\gamma > 0$ be the   appropriate   slope parameter at $\mathfrak{b}$. %
	Then, the local $k$-point correlation function at $\mathfrak{b}$ is universal. That is, for every $k \in \N$ there exists a $k$-point correlation function $p_{k, \alpha}^{\rm GOE/GUE}$ such that for any test function $F \in C_c^1(\overline{\Omega})$ on a bounded open set $\Omega \subset \R^k$, it holds that\footnote{Here, $\mathfrak{b}$ is identified with the vector $(\mathfrak{b}, ... , \mathfrak{b}) \in \R^k$. }
	\begin{equation} \label{eq:universal}
		\int_{\R^k} F(\bm x) \left[ \frac{N^{k/4}}{\gamma^k} p_k^{(N)}\left(\mathfrak{b} + \frac{\bm x}{\gamma N^{3/4}}\right) - p_{k, \alpha}^{\rm GOE/GUE}(\bm x)\right] \, \dif \bm x = \mathcal{O}_{k, \Omega}(N^{-c(k)} \Vert F \Vert_{C^1}) \,, 
	\end{equation}
	where 
	the parameter $\alpha$ depends on  $\gamma$, the local length scale and the specific shape  %
	of the scDOS   around $\mathfrak{b}$, %
	i.e., whether %
	it is   an exact cusp, a small gap, or a small minimum (see \cite[Eq.~(2.6)]{Cusp1} or \cite[Eq.~(2.5)]{Cusp2}). 
			The constant $c(k) > 0$ in \eqref{eq:universal} depends only on $k$, and the implicit constant in the error term depends on $k$ and the diameter of the set $\Omega$.
		\end{theorem}
		
		\begin{remark}[On $p_{k, \alpha}^{\rm GUE/GOE}$] \label{rmk:pearcey}
			For the universal $k$-point correlation function $p_{k, \alpha}^{\rm GOE/GUE}$, we have the following. 
			\begin{itemize}
				\item[(i)] In the \emph{complex Hermitian} symmetry class, the $k$-point function takes the determinantal form
				\begin{equation}
					p_{k, \alpha}^{\rm GUE}(\bm x) = \det\big(K_\alpha(x_i, x_j)\big)_{i,j=1}^k \,, 
				\end{equation}
				where the \emph{extended Pearcey kernel} with parameter $\alpha \in \R$ is given by
				\begin{equation}
					K_\alpha(x,y) = \frac{1}{(2 \pi \ii)^2} \int_{\Xi} \dif z \, \int_\Phi \dif w \, \frac{\exp\big(-w^4/4 + \alpha w^2/2 - yw + z^4/4 - \alpha z^2/2 + xz\big)}{w-z} \,. 
				\end{equation} 
				Here, $\Xi$ is a contour consisting of rays from $\pm \ee^{\ii \pi/4}$ to $0$ and rays from $0$ to $\pm \ee^{-\ii \pi/4}$, and $\Phi$ is the ray from $- \ii \infty$ to $ \ii \infty$. See \cite{adler2010airy, brezin1998universal, tracy2006pearcey} and the references in \cite{Cusp1} for more details.
				\item[(ii)] In the \emph{real symmetric} case, the $k$-point correlation function $p_{k, \alpha}^{\rm GOE}$ (possibly only a distribution) is not known explicitly, not even if it is Pfaffian. However, $p_{k, \alpha}^{\rm GOE}$ exists in the dual of $C^1$ as the limit of correlation functions of a suitable one-parameter family of Gaussian comparison models (see Sec.~3 and in particular Eq.~(3.5) of \cite{Cusp2}). 
			\end{itemize}
		\end{remark}

		\section{Zigzag strategy: Proof of the main results} \label{sec:zigzag}
		To streamline the presentation, we assume that the set of admissible energies $\mathcal{I}$, defined in \eqref{eq:admE} of Assumption~\ref{ass:Mbdd}, is the entire real line, that is, $\mathcal{I}=\mathbb{R}$. We discuss the straightforward modifications for general $\mathcal{I}$ in Remark~\ref{rem:non-local}.
		\begin{definition} [Local Laws] \label{def:ll_template}
			Let $H_u$ be a random matrix depending on a parameter\footnote{In applications, the parameter $u$ will typically be  time and the set $ \mathcal{U}$ will be a 
				bounded subinterval of $\R$.} $u \in \mathcal{U}$, and let $M_u$ be the solution to the MDE \eqref{eq:MDE} with the data pair $(\E H_u, \mathcal{S}_{H_u})$, where $\mathcal{S}_{H_u}$ is defined in \eqref{eq:self_energy_def}.
			For all $u \in \mathcal{U}$, let $\mathcal{D}_u\subset \mathbb{H}$ and let $\tar > 0$.
			We say that the resolvent $G_u(z) := (H_u-z)^{-1}$ satisfies the averaged local law and the isotropic local law, respectively, with data $(\mathcal{D}_u, \tar)$ uniformly in $u \in \mathcal{U}$, if and only if the bounds
			\begin{equation} \label{eq:ll_template}
				\biggl\lvert \bigl\langle \bigl(G_u(z) - M_{u}(z)\bigr) B \bigr\rangle  \biggr\rvert \le \frac{N^{3\tar}}{N \eta}, \quad \text{and} \quad \biggl\lvert \bigl(G_u(z) - M_{u}(z)\bigr)_{\bm x \bm y} \biggr\rvert \le N^{\tar} \biggl(\sqrt{\frac{\rho_{u}(z)}{N \eta} } + \frac{1}{N\eta}\biggr),
			\end{equation}
			hold uniformly in $z := E+\ii\eta \in \mathcal{D}_u$   and in $u \in \mathcal{U}$, with very high probability,  for any deterministic vectors $\bm x, \bm y \in \mathbb{C}^N$ with   $\norm{\bm x} = \norm{\bm y} = 1$, and any deterministic matrices $B$ with $\norm{B}_{\mathrm{hs}} = 1$. Here $\rho_u(z) := \tfrac{1}{\pi}\langle \Im M_u(z)\rangle$.
		\end{definition}
		
		The goal of the present section is to prove the local laws in the \textit{above the scale} regime, where $\rho(z)N|\Im z|$ is large.
		Fix a (small) $N$-independent constant $\scl > 0$, a large constant $C_L > 0$, and define the spectral domain $\abvD$ as
		\begin{equation} \label{eq:abvD}
			\abvD \equiv \abvD(\scl, C_L) := \bigl\{ z := E + \ii \eta \in \mathbb{H} \,:\, \rho(z)N\eta \ge N^{\scl}, \, |E| \le C_L, \, \eta \le C_L \bigr\}.
		\end{equation}
		The regime $\rho(z)N \eta \ge N^\varepsilon$ is natural for studying the local laws, since $\rho(E+\ii\eta)N\eta$ is the typical number of eigenvalues in the interval of size $\eta$ around the energy $E$. 
		\begin{theorem} [Local Laws above the Scale] \label{th:Neta_local_laws}
			Fix a (small) $N$-independent constant $\scl > 0$, a large constant $C_L > 0$.
			Let $H$ be a random matrix satisfying the Assumptions~\ref{ass:boundedexp}--\ref{ass:Mbdd}, then the resolvent $G(z) := (H-z)^{-1}$ satisfies the local laws \eqref{eq:ll_template} with data $(\abvD, 2\tar)$, for any fixed tolerance exponent $0 < \tar \le \tfrac{1}{100}\scl$, where $\abvD=\abvD(\scl, C_L)$.
		\end{theorem}
		
		To prove Theorem~\ref{thm:main} in the {\it below the scale} regime, that is, to handle the case when $\rho(z)N|\Im z|$ is small, we proceed in two steps. In the key first step  we use
		the local laws above the scale of Theorem~\ref{th:Neta_local_laws} to prove Theorem~\ref{thm:noeig} that asserts the absence of spectrum outside of the support of the scDOS $\rho$.
		Then the second step  is a routine derivation of \eqref{eq:AVLL} and \eqref{eq:ISOLL} from \eqref{eq:noeig} and \eqref{eq:ll_template}. Both steps are presented in Section~\ref{sec:outside}. In the main part of the proof, we only consider spectral parameters $z$ satisfying $\dist(z, \supp\rho )\lesssim 1$. The easy extension to the regime $\dist(z, \supp\rho ) \gtrsim 1$ and the resulting $\langle z \rangle^{-2}$-decay are briefly addressed in Remark~\ref{rem:far_away}. 
		
		In the sequel, we treat the constants $\scl, C_L$ in \eqref{eq:abvD} as additional model parameters and omit them from the arguments of $\abvD$.

		Throughout the paper, we consistently use the notation $\scl, \tar, \sscl, \step$ to represent positive $N$-independent tolerance exponents, each playing a particular role in the proof.  Specifically, $\scl$ denotes the tolerance exponent from the definition of the domain $\abvD$ (see \eqref{eq:abvD} and \eqref{eq:abvD_t} below); $\tar$ and its multiples represent the target tolerance exponents for the local laws above the scale in \eqref{eq:ll_template}. The exponent $\sscl$ appears in the below-the-scale part of the proof (Section~\ref{sec:outside}). Multiples of $-\sscl$ are used in the exclusion estimate \eqref{eq:excl_template} and in the lower bound on $\rho N \eta$ in \eqref{eq:below_Dt}. The exponent $\step$ refers to the step size used in various inductive arguments. In the sequel, we adhere to the following conventions:
		\begin{equation} \label{eq:exponents}
			\step \ll \tar \ll \scl, \quad \sscl \ll \tar, \quad \step < \mu,
		\end{equation} 
		where $\mu > 0$ is the constant from Assumption~\ref{ass:cumulants} (ii). We also assume that the arbitrary exponent $\arb > 0$ is much smaller than the other tolerance exponents, that is, $\arb \ll \step$ and $\arb \ll \sscl$.

		\subsection{Input: Global Laws}
		Let $\rho(z)$ be the harmonic extension to $\mathbb{H}$ of the scDOS corresponding to a solution of \eqref{eq:MDE}. Given small positive constants  $\scl, \tar > 0$, and a large constant $D>0$, we define the global domain as 
		\begin{equation} \label{eq:globD_def}
			\globD\equiv \globD(D, \scl, \tar, \rho) := \bigl\{ z := E+\ii\eta\in \mathbb{H}\,:\, |E| \le N^D,\, N^{-1+\scl} \le \eta \le N^D,\, \rho(z)^{-1}\eta \ge N^{-\tar/4} \bigr\}.
		\end{equation}
		Effectively, the function $\rho(z)^{-1} \eta$ in \eqref{eq:globD_def} controls the proximity of the spectral parameter $z$ to the support $\rho$.
		
		\begin{proposition} \label{prop:global}
			Let $H$ be a random matrix satisfying the Assumptions~\ref{ass:boundedexp}--\ref{ass:Mbdd}, and let $\rho(z)$ be the scDOS arising from the solution to the MDE \eqref{eq:MDE} corresponding to $H$. Let $\bddD := \mathcal{I}+[-c_M,c_M]+\ii\mathbb{R} \subset \mathbb{C}$, where $\mathcal{I}$ is defined in \eqref{eq:admE}.
			Fix a large constant $D > 0$ and a tolerance exponent $0 < \tar < \tfrac{1}{10}\scl$.
			Then the resolvent $G(z) := (H-z)^{-1}$ satisfies 
			\begin{subequations}
				\begin{equation} \label{eq:global_iso}
					\biggl\lvert \bigl(G(z) - M(z)\bigr)_{\bm x \bm y} \biggr\rvert \le N^{\tar} \Psi(z)\lVert\bm x \rVert  \lVert\bm y \rVert,
				\end{equation} 
				\begin{equation} \label{eq:global_av}
					\biggl\lvert \bigl\langle \bigl(G(z) - M(z)\bigr)B \bigr\rangle \biggr\rvert \le N^{3\tar}\Psi(z)\sqrt{\frac{\langle z\rangle}{N\eta}}\lVert B\rVert_{\mathrm{hs}},
				\end{equation}
			\end{subequations}  
			with very high probability, uniformly in $z := E +\ii\eta \in \globD(D,\varepsilon,\xi,\rho)\cap\bddD$, for any deterministic vectors $\bm x$, $\bm y$ and matrices $B$. Here the control parameter $\Psi(z)$ is defined as
			\begin{equation} \label{eq:iso_psi_def}
				\Psi(z) %
				:= \sqrt{\frac{\rho(z)}{\langle z \rangle^2 N\eta}} + \frac{1}{\langle z \rangle^2 N\eta}, \quad \eta:= \Im z.
			\end{equation}
		\end{proposition}
		We prove Proposition~\ref{prop:global} in Section~\ref{sec:stable}.
		
		\subsection{Local law via Zigzag strategy: Proof of Theorem~\ref{th:Neta_local_laws}}
		
		\subsubsection{ Preliminaries: Two Random Matrix Flows}  
		For any random matrix $H$, we define the covariance tensor $\Sigma_{H}$ corresponding to $H$ by its action on any deterministic matrix $X\in \mathbb{C}^{N\times N}$,
		\begin{equation} \label{eq:corr_def}
			\Sigma_{H}[X] := \E \bigl[ \Tr \bigl[ (H-\E H)X \bigr] (H-\E H)\bigr].
		\end{equation}
		Note that $\Sigma_{H}$ is different from the self-energy operator \eqref{eq:self_energy_def}, but they both carry equivalent information. Moreover,
		it is positive definite on the space of matrices equipped with the usual scalar product $(X, Y)= \langle X^*Y\rangle$
		and we will denote by $\Sigma^{1/2}$ its square root.
		
		Along the proof, we use two distinct flows in the space of $N\times N$ random matrices: the \textit{zig-flow} (standard Ornstein–Uhlenbeck process), defined as
		\begin{equation} \label{eq:OU_flow}
			\rd H_t = -\frac{1}{2}H_t \rd t + \frac{\rd \Brwn_t}{\sqrt{N}},  \qquad t\ge 0;
		\end{equation}
		and the \textit{zag-flow} (modified Ornstein–Uhlenbeck process), distinguished by the superscript $t$, 
		\begin{equation} \label{eq:zag_flow}
			\rd H^t = -\frac{1}{2} \bigl(H^t - \E H^t\bigr) + \Sigma_{H^0}^{1/2}\bigl[\rd \Brwn_t\bigr],    \qquad t \ge 0,
		\end{equation}
		where $\Sigma_{H^0}$ is the covariance tensor of $H^0$, defined according to \eqref{eq:corr_def}.
		In both \eqref{eq:OU_flow} and \eqref{eq:zag_flow}, $\Brwn_t$
		denotes the real symmetric or complex Hermitian Brownian motion, depending on the symmetry class of $H$.

		Note that along the zig-flow \eqref{eq:OU_flow}, the covariance tensor $\Sigma_t := \Sigma_{H_t}$, corresponding to $H_t$ via \eqref{eq:corr_def}, satisfies the ordinary differential equation
		\begin{equation} \label{eq:Sigma_flow}
			\rd \Sigma_t = \bigl(-\Sigma_t + \Sigma_{\mathrm{G}}\bigr)\rd t,
		\end{equation}
		where $\Sigma_{\mathrm{G}}$ is the covariance tensor of a GOE/GUE matrix in the same symmetry class as $H$. That is $\Sigma_{\mathrm{G}}[X] = N^{-1}X$ in the complex Hermitian case, and $\Sigma_{\mathrm{G}}[X] = N^{-1}(X+X^\mathfrak{t})$ in the real-symmetric case, where $X^\mathfrak{t}$ denotes the transpose of $X$. On the other hand, along the zag-flow \eqref{eq:zag_flow}, the expectation and the covariance tensor of $H_t$ (and hence the self-energy $\mathcal{S}_{H_t}$) are preserved. Therefore, the deterministic approximation $M$ remains unchanged along the zag-flow.

		For any $t \ge 0$,  we define the flow maps $\mathfrak{F}_{\mathrm{zig}}^t$ and $\mathfrak{F}_{\mathrm{zag}}^t$ on the space of probability distribution $\mathcal{P}(\mathbb{C}^{N\times N})$ by
		\begin{equation} \label{eq:zig_operator}
			\mathfrak{F}_{\mathrm{zig}}^t\bigl[H \bigr] := H_t, \quad \text{ where } H_t \text{ solves (\ref{eq:OU_flow}) with the initial condition } H_0 = H.
		\end{equation}
		\begin{equation} \label{eq:zag_operator}
			\mathfrak{F}_{\mathrm{zag}}^t\bigl[ H  \bigr] := H^t, \quad \text{ where } H^t \text{ solves (\ref{eq:zag_flow}) with the initial condition } H^0 = H.
		\end{equation}
		The key relation between the flow maps $\mathfrak{F}_{\mathrm{zig}}^t$ and $\mathfrak{F}_{\mathrm{zag}}^t$ is captured by the following lemma.
		\begin{lemma} [Flow Distribution Surjectivity] \label{lemma:OUsurj}
			Let $H$ be a random matrix satisfying the fullness condition \eqref{eq:fullness} with a constant $0 < c < 1$, then there exists a random matrix $\mathfrak{H}_{c,t}(H)$ such that
			\begin{equation} \label{eq:init_cond_for_zig}
				\mathfrak{F}_{\mathrm{zig}}^t\bigl[\mathfrak{H}_{c,t}(H)\bigr] \,\overset{d}{=}\, \mathfrak{F}_{\mathrm{zag}}^{s(t)}\bigl[ H  \bigr], \quad 0 \le t \le -\log(1-c),
			\end{equation}
			
			where the function $s(t) \equiv s_c(t)$ is defined as 
			\begin{equation} \label{eq:init_cond_reverse}
				s(t) \equiv s_c(t) := \log c - \log\bigl(c - 1 + \ee^{-t}\bigr),
			\end{equation}
			and satisfies
			\begin{equation} \label{eq:s(t)_est}
				s(t) \le 2c^{-1} t,\quad 0 \le t \le c/2. 
			\end{equation}
		\end{lemma}
		We defer the proof of Lemma~\ref{lemma:OUsurj} to the Appendix~\ref{app:tech}.

		\subsubsection{Zigzag approach: Iterative application of the characteristic flow and GFT}

		We consider the time-dependent matrix Dyson equation (MDE), 
		\begin{equation} \label{eq:MDEt}
			-M_t(z)^{-1} = z - A_t + \mathcal{S}_t\bigl[M_t(z)\bigr], \quad z \in \mathbb{C}\backslash\mathbb{R},\quad   (\Im z) \Im M_t(z) > 0,
		\end{equation}
		where the data pair $(A_t, \mathcal{S}_t)$ is given as the unique solutions to the differential equations
		\begin{equation} \label{eq:ASevol}
			\rd A_t = - \frac{1}{2} A_t \rd t\,, \quad \rd \mathcal{S}_t = (- \mathcal{S}_t + \langle \cdot \rangle) \rd t \,.%
		\end{equation}
		with the {\it terminal conditions} $A_T = A =\E H$ and $\mathcal{S}_T = \mathcal{S} = \E [ (H-A) (\,\cdot\,) (H-A) ]$, respectively.  
		
		Given $M_t(z)$, we consider the {\it characteristic ODE} for the time dependent spectral parameter $z_t\in \C$ (see Figure~\ref{fig:zflow_fig}), 
		\begin{equation} \label{eq:z_flow}
			\rd z_t = -\frac{1}{2}z_t\rd t - \bigl\langle M_t(z_t) \bigr\rangle \rd t.
		\end{equation}
		\begin{figure}[H]
			\centering
			\includegraphics[width=.495\textwidth]{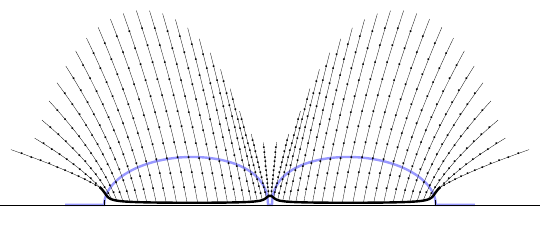} \hfill \includegraphics[width=.495\textwidth]{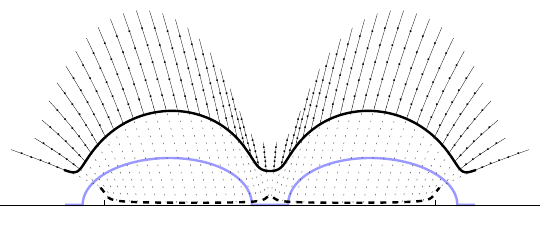}
			\caption{
				The left panel depicts several trajectories of the flow \eqref{eq:z_flow} that terminate at the \textit{scale curve} $\rho_T(z)N\Im z = c$ (solid black line), while the the graph of scDOS $\rho_T$ is superimposed in light blue. 
				The right panel depicts trajectories up to an intermediate time $t\in(0,T)$ with their continuations beyond $t$ shown as thin dotted lines. 
				The pre-image of the scale curve at the time $t$ is depicted as a solid black line, and the scale curve itself is depicted as a dashed black line. 
				The graph of scDOS $\rho_t$ is superimposed in light blue.
				In both panels, the black markers along the trajectories of \eqref{eq:z_flow} are evenly spaced in time.
			}
			\label{fig:zflow_fig}
		\end{figure}
		
		By trivial ODE arguments, for all $0 \le s \le t$, the corresponding (inverse) flow map $\varphi_{s,t} : \overline{\mathbb{H}}\to \overline{\mathbb{H}}$ is defined uniquely by 
		\begin{equation} \label{eq:phi_map}
			\varphi_{s,t}(z_t) := z_s, \quad \text{where } z_s \text{ solves (\ref{eq:z_flow})}.
		\end{equation}
		
		It can be directly checked that along the trajectories of \eqref{eq:z_flow}, the solution to the time-dependent MDE \eqref{eq:MDEt} satisfies
		\begin{equation} \label{eq:m_flow}
			\rd M_t(z_t) = \frac{1}{2}M_t(z_t) \rd t.
		\end{equation}
		
		\begin{lemma} [Time-Dependent Domains] \label{lemma:abvD_t} 
			There exist a constant $C' \sim 1$ such that for any constant $0  < c' \le \pi$ and any terminal time $0 < T \lesssim 1$, the time-dependent domains $\abvD_t$, $t\in [0,T]$, (see Figure~\ref{fig:abv_fig}), defined as	
			
			\begin{equation} \label{eq:abvD_t}
				\begin{split}
					\abvD_t \equiv \abvD_t(\scl, C_L, c',T) := \bigl\{ z := E + \ii \eta \in \mathbb{H} \,:\, ~&\rho_t(z)N\eta \ge N^{\scl}, \, |E| \vee \eta \le C_{L} + C'\cdot(T-t),\\
					&\rho_t(z)^{-1}\eta \ge c' \cdot\bigl(N^{-1+\scl} + T-t\bigr )   \bigr\},
				\end{split}
			\end{equation}
			satisfy $\varphi_{s,t}(\abvD_t) \subset \abvD_s$ for all $0 \le s \le t \le T$, where $\varphi_{s,t}$ is the flow map defined in \eqref{eq:phi_map}.
		\end{lemma}
		We defer the proof of Lemma~\ref{lemma:abvD_t} to Appendix~\ref{app:tech}.
		
		\begin{figure}[h]
			\centering
			\includegraphics[width=.33\textwidth]{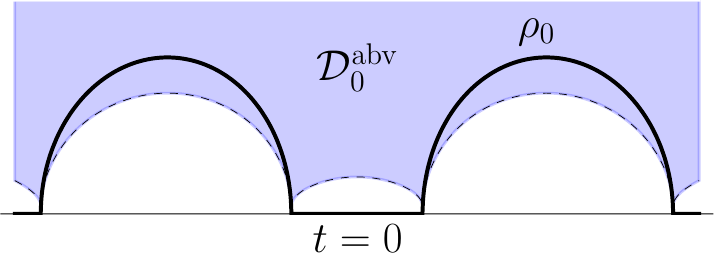}\hfill
			\includegraphics[width=.33\textwidth]{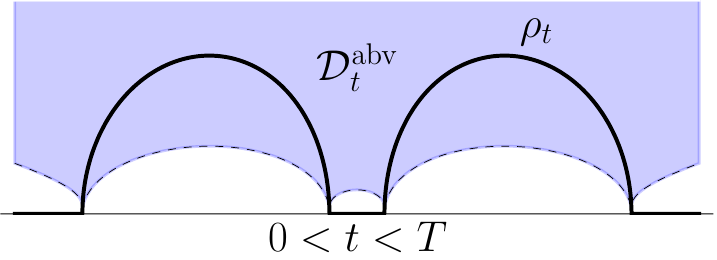}\hfill
			\includegraphics[width=.33\textwidth]{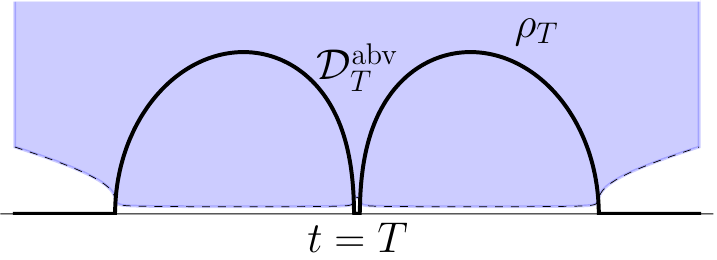}\hfill
			\caption{ The time-dependent domain $\abvD_{t}$, defined in \eqref{eq:abvD_t}, is illustrated in blue at three distinct times: the initial time $t = 0$ (left), an intermediate time $0 < t < T$ (center), and the terminal time $t=T$ (right). The graph of the scDOS $\rho_t$ is superimposed in black on each panel (not to scale).}
			\label{fig:abv_fig}
		\end{figure}

		As in \eqref{eq:globD_def}, the function $\rho_t(z)^{-1}\eta $ in the definition \eqref{eq:abvD_t} effectively controls the distance between $z$ and the support of $\rho_t$. Therefore the time-dependent family
		of domains $\abvD_t$ effectively interpolates between the global regime $\globD$ and the final target domain $\abvD$.
		Indeed,   since $\rho(z) \lesssim 1$, by choosing the constant $c' \sim 1 $ in \eqref{eq:abvD_t} small enough, we can guarantee that 
		$\abvD \subset \abvD_T$, where we recall that $\abvD$ is defined in \eqref{eq:abvD}. On the other hand, it follows from \eqref{eq:globD_def} that by choosing
		\begin{equation} \label{eq:term_time}
			T := C N^{-\tar/4},
		\end{equation} 
		with a sufficiently large constant $C \gtrsim 1$, we can guarantee that $\abvD_0 \subset \globD$, where $\globD$ is defined in \eqref{eq:globD_def}. 
		
		We conduct the proof inductively. Fix a tolerance exponent $ 0 < \tar \ll \scl$, a step size $ 0 < \step \ll \tar$ (recall \eqref{eq:exponents}). For the terminal time $T$ chosen as in \eqref{eq:term_time}, let $K$ be the smallest integer such that $N^{-K\step}T \le N^{-1+\scl}$,
		and define a sequence of times $\{t_k\}_{k=0}^K$ as
		\begin{equation} \label{eq:t_steps}
			t_0 := 0,\quad  t_k := T - N^{-k\step}T, \quad k \in \{1,\dots, K-1\}, \quad t_K := T.
		\end{equation}
		Let $\{\dift_k\}_{k=1}^K$ denote the difference sequence of $\{t_k\}_{k=0}^K$, that is
		\begin{equation}
			\dift_k := t_k - t_{k-1}, \quad k \in \{1,\dots, K\}.
		\end{equation}
		Let $\Sigma_t$ solve the equation \eqref{eq:Sigma_flow} with the terminal condition $\Sigma_T = \Sigma$, where $\Sigma$,
		defined via \eqref{eq:corr_def}, is the covariance tensor of the target matrix $H$, for which
		we eventually prove the local laws in Theorem~\ref{thm:main}.
		Observe that for all $0 \le t \le T$, the solution $\Sigma_t$ satisfies
		\begin{equation} \label{eq:fullness_alongflow}
			\Sigma_t \ge \other{c}\,\Sigma_{\mathrm{G}}, \quad \other{c} := \frac{c_{\mathrm{flat} }}{2} \wedge 1,
		\end{equation}
		where $c_{\mathrm{flat}}$ is the constant in Assumption~\ref{ass:full}.
		Given the target random matrix ensemble $H$, we construct two sequences of random matrices, $\{H_k\}_{k=0}^{K}$ and $\{H^k\}_{k=1}^K$ recursively by
		\begin{equation} \label{eq:H_k_defs}
			H_K := H, \quad H^k := \mathfrak{F}^{s(\dift_k)}_{\mathrm{zag}}\bigl[H_k\bigr], \quad H_{k-1} := \mathfrak{H}_{\other{c},\dift_k}\bigl(H^k\bigr), \quad k\in \{1,\dots, K\},
		\end{equation}
		where $s(t) := s_{\other{\alpha}}(t)$ and $\mathfrak{H}_{\other{c},\dift_k}$ are given by Lemma~\ref{lemma:OUsurj}, and $\other{c}$ is the constant in \eqref{eq:fullness_alongflow}. It follows by a simple 
		backward inductive argument starting at $k=K$ that the covariance tensor of both $H_k$ and $H^k$ is given by $\Sigma_{t_k}$, hence by \eqref{eq:fullness_alongflow}, $H_{k-1}$ is well-defined. 

		\begin{figure}[H]
			\centering
			\begin{minipage}{0.6\textwidth}
				\centering
				\includegraphics[width=\textwidth]{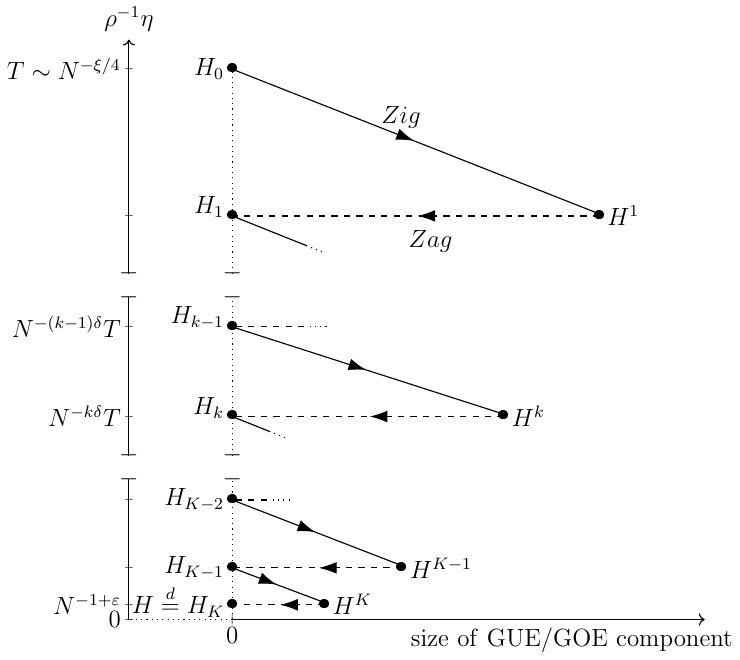}
			\end{minipage}%
			\hfill
			\begin{minipage}{0.4\textwidth}
				\centering
				\captionsetup{width=.85\textwidth, justification=justified}
				\caption[Zigzag caption]{
					Schematic representation of the Zigzag induction.
					The random matrices $H_k, H^k$, as defined in \eqref{eq:H_k_defs}, are situated within an abstract coordinate system. The horizontal axis represents the size of the Gaussian component, while the vertical axis indicates the lower bound on $\rho(z)^{-1}\eta$ in the domains, c.f. \eqref{eq:abvD_t}, where we prove the local laws \eqref{eq:ll_template}.
					Solid arrows denote applications of Proposition~\ref{prop:zig} (referred to as \textit{Zig} steps), and dashed {arrows\footnotemark} indicate applications of Proposition~\ref{prop:zag} (\textit{Zag} steps).
				}
				\label{fig:zigzag_fig}
			\end{minipage}
		\end{figure}
		
		\footnotetext{Note that the zag-flow \eqref{eq:zag_flow} acts in the direction opposite to the dashed arrows, i.e., the size of the Gaussian components increases along both zig- and zag- matrix flows, \eqref{eq:OU_flow} and \eqref{eq:zag_flow}, respectively.}
		
		\begin{proposition}[Zig Step] \label{prop:zig} Fix $k \in \{1, \dots, K\}$, and denote
			\begin{equation} \label{eq:G_t}
				G_{t}(z) := \bigl(\mathfrak{F}_{\mathrm{zig}}^{t-t_{k-1} }[H_{k-1}] - z\bigr)^{-1}, \quad t_{k-1} \le t \le t_k.
			\end{equation}
			Assume that for some $\tar, \arb > 0$ with $\tar + K \arb \ll \scl$ , and $\ell \le 2k$, the resolvent $G_{t}$ satisfies the local laws \eqref{eq:ll_template} with data $(\abvD_{t}, \tar+\ell\arb)$ at time $t = t_{k-1}$, then the resolvent $G_t$ satisfies the local laws \eqref{eq:ll_template} with data $(\abvD_{t}, \tar+(\ell+1)\arb)$ uniformly in $t \in [t_{k-1}, t_k]$.
		\end{proposition}
		
		\begin{proposition}[Zag Step] \label{prop:zag} 
			Fix $k \in \{1, \dots, K\}$. Let $s_k := s(\dift_k)$ be the time defined in \eqref{eq:init_cond_reverse}, let $H_k$ be the random matrix defined in \eqref{eq:H_k_defs}, and denote
			\begin{equation} \label{eq:G^s}
				G^{s}(z) := \bigl(\mathfrak{F}_{\mathrm{zag}}^{s}[H_{k}] - z\bigr)^{-1}, \quad 0 \le s \le s_k.
			\end{equation}
			Assume that for some $\tar, \arb > 0$ with $\tar + K \arb \ll \scl$, and $\ell \le 2k$, the resolvent $G^s$ satisfies the local laws \eqref{eq:ll_template} with data $(\abvD_{t_k}, \tar+\ell\arb)$ at time $s = s_{k}$, then  $G^s$ satisfies the local laws \eqref{eq:ll_template} with data $(\abvD_{t_k}, \tar+(\ell+1)\arb)$ uniformly in $s \in  [0, s_k]$. 
		\end{proposition}
		
		Having formulated the cardinal steps of the Zigzag strategy, we now put them together to prove our key theorem on the local laws above the scale. 
		Note that in the above the scale regime $\rho(z)N\eta \ge N^\varepsilon$, the term $1/(N\eta)$ in the isotropic bound is \eqref{eq:ll_template} is dominated by $\sqrt{\rho/(N\eta)}$, and hence will be ignored in Sections~\ref{sec:zig} and~\ref{sec:zag}.
		
		\begin{proof} [Proof of Theorem~\ref{th:Neta_local_laws}]
			Recall our choice of the constant $c' \sim 1 $ in \eqref{eq:abvD_t} and the terminal time $T \sim N^{-\tar/4} $ in \eqref{eq:term_time} that guarantees the inclusions  $\abvD_0 \subset \globD$ and $\abvD \subset \abvD_T$. Therefore, Proposition~\ref{prop:global} implies that the resolvent $G_0(z) := (H_0-z)^{-1}$ of a random matrix $H_0$, defined in \eqref{eq:H_k_defs}, satisfies the local laws \eqref{eq:ll_template} with data $(\abvD_0, \tar)$. Using Propositions~\ref{prop:zig} and~\ref{prop:zag}  in tandem  $K$ times,  we prove by forward induction on $k$ that for any $\arb > 0$, the resolvent $G_k(z) := (H_k-z)^{-1}$ satisfies the local laws \eqref{eq:ll_template} with data $(\abvD_{t_k}, \tar+2k\arb)$, for all $k \in \{1,\dots, K\}$. Since $H_K = H$ and $\abvD_{t_K} = \abvD_{T} \supset \abvD$, this concludes the proof of Theorem~\ref{th:Neta_local_laws}.
		\end{proof}

		\begin{remark} [On Locality of Assumption~\ref{ass:Mbdd}] \label{rem:non-local}
			In the case of a general set of admissible energies $\mathcal{I}$, defined in \eqref{eq:admE}, our proof holds verbatim, except the spectral domains $\globD, \abvD,\abvD_{t}$ used
			along the proof have to be restricted. More precisely, we need the following modifications: 
			\begin{itemize}
				\item[(i)] we restrict the domain $\abvD_{t}$, defined in  \eqref{eq:abvD_t}, by intersecting it with the region
				\begin{equation}
					\bddD_t := \bigl\{ z \in \mathbb{C} \,:\, \dist(\Re z, \mathcal{I}) \le c_M/2  + C'\cdot(T-t) \bigr\}, \quad 0 \le t \le T;
				\end{equation}
				\item[(ii)] we restrict the domain $\abvD$, defined in \eqref{eq:abvD}, by intersecting it with  $\bddD_T$;
				\item[(iii)] we restrict the global domain $\globD$, defined in \eqref{eq:globD_def} by intersecting it with $ \{ z \in \mathbb{C} \,:\, \dist(\Re z, \mathcal{I}) \le \tfrac{3}{4}c_M  \}$.
			\end{itemize}
		\end{remark}
		
		\subsection{Proofs of Corollaries~\ref{cor:deloc}--\ref{cor:rig} and Theorem~\ref{thm:univ}} \label{subsec:proofcor}

		In this section, we deduce eigenvector delocalization, band rigidity and eigenvalue rigidity, as well as cusp universality from the local law in Theorem~\ref{thm:main}. These arguments are essentially independent of the correlation structure of the random matrix, so we only refer to analogous proofs, which can easily be adjusted to our case with straightforward modifications. 
		
		\begin{proof}[Proof of Corollary~\ref{cor:deloc} on eigenvector delocalization]
			As usual, eigenvector delocalization is an immediate consequence of the optimal isotropic local law from Theorem~\ref{thm:main} for $\Im G$; see \cite[Proof of Corollary 2.4]{slowcorr} or \cite[Proof of Corollary~1.14]{univWigtype} for this argument. 
		\end{proof}
		
		\begin{proof}[Proof of Corollary~\ref{cor:rig} on band rigidity and eigenvalue rigidity]
			The proof of band rigidity was first done for correlated matrices in \cite[Proof of Corollary 2.5 in Section 5]{edgelocallaw} but with $\dist(e_0, \supp \rho) \gtrsim 1$. The adjustments for $\dist(e_0, \supp \rho) \ge N^{\excl} \etaf(e_0)$ are carried out in \cite[Proof of Corollary 2.6]{Cusp1} for the case of Wigner-type matrices (i.e.~without correlations). This argument immediately translates to our setting, hence we omit the details for brevity. 
			
			Armed with band rigidity as in \eqref{eq:intmass}, the proof of Corollary~\ref{cor:rig}~(b)  is conducted in the same way as in \cite[Proofs of Corollaries 1.10 and 1.11]{univWigtype} or \cite[Proof of Corollary 2.6]{Cusp1}. 
		\end{proof}
		
		\begin{proof}[Proof of Theorem~\ref{thm:univ} on cusp universality]
			Given the optimal local law in Theorem~\ref{thm:main}, universality at the cusp follows by the \emph{three-step strategy}. This has already been worked out in the general correlated case in both the complex Hermitian \cite{Cusp1} and real symmetric \cite{Cusp2} symmetry class. More precisely, in \cite[Section~3]{Cusp2} it is spelled out that only the local law in this paper (as the first step of the three-step strategy) required restricting to Wigner-type matrices with independent entries. Our Theorem~\ref{thm:main} provides the necessary local law for correlated matrices. 
		\end{proof}

		\section{Characteristic flow: Proof of Proposition~\ref{prop:zig}} \label{sec:zig}
		First, we collect the necessary properties of the solution $M_t$ to the time-dependent MDE \eqref{eq:MDEt}.
		\begin{lemma}[Preliminary bounds on $M_t$] \label{lemma:Mt} 
			Let $(A, \mathcal{S})$ be a data-pair satisfying the Assumptions~\ref{ass:boundedexp},~\ref{ass:full}, and~\ref{ass:Mbdd}. Then there exists a threshold $T_* \sim 1$ such that for any terminal time $0 < T < T_*$, the solution $M_t$ to the time-dependent MDE \eqref{eq:MDEt}, with the terminal condition on the data pair $(A_T, \mathcal{S}_T) = (A, \mathcal{S})$, satisfies
			\begin{equation} \label{eq:imM}
				\norm{M_t(z)} \lesssim 1, \quad c \rho_t(z) \le \Im M_t(z) \le C \rho_t(z),
			\end{equation}
			uniformly in $z$ with $\Re z \in \mathcal{I}$, where $\mathcal{I}$ is the set of admissible energies from \eqref{eq:admE}. Here the second inequality holds in the sense of quadratic forms, with $1 \lesssim c \le C \lesssim 1$.
		\end{lemma}
		Essentially, at the terminal time $t = T$, the bounds \eqref{eq:imM} follow from the assumptions of the lemma, while at all other times $0 \le t < T$, the equations \eqref{eq:ASevol} guarantee that the data pair $(A_t, \mathcal{S}_t)$ constitutes only a small perturbation around $(A_T, \mathcal{S}_T)$. We give a more detailed proof of Lemma~\ref{lemma:Mt} in Appendix~\ref{app:tech}.

		Equipped with Lemma~\ref{lemma:Mt}, we are ready to prove Proposition~\ref{prop:zig}. 
		We conduct the proof in the complex Hermitian case, the obvious modifications in the real symmetric case\footnote{For a detailed treatment of the real symmetric case in the setting of standard Wigner matrices, we refer the reader to Section 4 of \cite{Cipolloni2023Edge}. The modifications for more general ensembles
			are analogous.  } are left to the reader.
		Throughout the proof we consider the step index $k$ to be fixed, and hence omit it from the subscripts. 
		
		It suffices to prove that the resolvent $G_t$ satisfies the local laws \eqref{eq:ll_template} with data $(\abvD_t, \tar + (\ell + 1)\arb)$ for any fixed $t := \tfin \in [t_{k-1}, t_{k}]$
		and $z \in \abvD_{\tfin}$, since uniformity in $t$ and $z$ can be obtained by a simple grid argument\footnote{
			The grid argument relies on two straightforward observations: 
			First, the resolvent $G_t(z)$ with $|\Im z| \ge N^{-1}$ -- and, therefore, all quantities we consider -- are Lipschitz continuous with a Lipschitz constant $\lesssim N^C$ for some $C>0$ both in $z$ and in $t$. Second, for any $C>0$, the intersection of $N^C$-many very-high-probability events also occurs with very high probability. 
			Therefore, a uniform very-high probability bounds are first established over a sufficiently fine $N^{-C}$ grid in the domain of $z$ or $t$, and then extended to the entire domain by Lipschitz continuity.  }. 
		Let $\tin := t_{k-1}$, and for all $t \in [\tin, \tfin]$, let $z_t := \varphi_{t,\tfin}(z)$, where the map $\varphi$ is defined in \eqref{eq:phi_map}. It follows from Lemma~\ref{lemma:abvD_t} that  $z_t \in \abvD_t$ for all $t \in [\tin, \tfin]$. 
		We denote $G_t := (H_t - z_t)^{-1}$, and $M_t := M_t(z_t)$, where $M_t$ is the solution to \eqref{eq:MDEt}.
		
		Using It\^{o}'s formula, we deduce that for any deterministic $N\times N$ matrix $B$,
		\begin{equation} \label{eq:Gav_evol}
			\rd \bigl\langle (G_t - M_t)B \bigr\rangle = \biggl( \frac{1}{2}\bigl\langle (G_t - M_t)B \bigr\rangle + \bigl\langle G_t - M_t \bigr\rangle  \bigl\langle G_t^2 B \bigr\rangle \biggr) \rd t + \frac{1}{\sqrt{N}} \sum_{ab } \partial_{ab} \bigl\langle G_t B \bigr\rangle \rd \bigl(\Brwn_{t}\bigr)_{ab},
		\end{equation}
		where $\partial_{ab} := \partial_{H_{ab,t}}$ denotes the partial derivative with respect to the matrix entry $H_{ab,t}$.
		In particular, for a fixed pair of deterministic vectors $\bm x, \bm y \in \mathbb{C}^N$ with 
		$\norm{\bm x} = \norm{\bm y} = 1$, setting $B := N \bm{y}\bm{x}^*$ we obtain
		\begin{equation} \label{eq:Giso_evol}
			\rd \bigl(G_t - M_t\bigr)_{\bm x \bm y} = \biggl( \frac{1}{2}\bigl(G_t - M_t\bigr)_{\bm x \bm y} + \bigl\langle G_t - M_t \bigr\rangle \bigl(G_t^2\bigr)_{\bm x \bm y} \biggr) \rd t + \frac{1}{\sqrt{N}} \sum_{ab } \partial_{ab} \bigl(G_t\bigr)_{\bm x \bm y} \rd \bigl(\Brwn_{t}\bigr)_{ab}.
		\end{equation}

		First, we prove that the resolvent $G_{\tfin}$ satisfies the isotropic local law and averaged local law in \eqref{eq:ll_template} for $B:= I$ with data $(\abvD_{\tfin}, \tar + (\ell+\tfrac{1}{2})\arb)$. Define a set of deterministic vectors $\mathcal{V} := \{\bm x, \bm y\}$.
		Define the stopping time $\tau$
		\begin{equation} \label{eq:tau_1_def}
			\begin{split}
				\tau := &~ \sup\biggl\{ \tin \le t \le \tfin \, : \,  \sup_{\tin \le s \le t}\max_{\bm u, \bm v \in \mathcal{V}} \bigl\lvert \sqrt{\rho_s(z_s)^{-1}N \eta_s } \bigl(G_s-M_s\bigr)_{\bm u \bm v} \bigr\rvert \le N^{\tar + (\ell+\tfrac{1}{2})\arb}   \biggr\}\\
				&\wedge \sup\biggl\{ \tin \le t \le \tfin \, : \,  \sup_{\tin \le s \le t} \bigl\lvert N \eta_s  \langle G_s-M_s\rangle \bigr\rvert \le N^{3\tar + 3(\ell+\tfrac{1}{2})\arb}    \biggr\},
			\end{split}
		\end{equation}
		where we denote $\eta_t := \Im z_t > 0$.
		
		Computing the quadratic variation of the martingale term in \eqref{eq:Gav_evol}, we obtain
		\begin{equation} \label{eq:av_mart_compute}
			\begin{split}
				\biggl[\int_{\tin}^\cdot \frac{1}{\sqrt{N}} \sum_{ab } \partial_{ab} \bigl\langle G_s \bigr\rangle \rd \bigl(\Brwn_{s}\bigr)_{ab} \biggr]_{t\wedge\tau} 
				&\le \int_{\tin}^{t\wedge\tau} \frac{\bigl\langle (\Im G_s)^2\bigr\rangle}{N^2\eta_s^2} \rd s \le \int_{\tin}^{t\wedge\tau} \frac{\bigl\langle \Im G_s\bigr\rangle}{N^2\eta_s^3} \rd s  \\ 
				&\le \int_{\tin}^{t\wedge\tau} \frac{\bigl\langle \Im M_s\bigr\rangle + \tfrac{1}{2}\eta_s}{N^2\eta_s^3} \rd s  + \int_{\tin}^{t\wedge\tau} \frac{\bigl\langle \Im G_s - \Im M_s\bigr\rangle}{N^2\eta_s^3} \rd s,
			\end{split}
		\end{equation}
		where in the penultimate step we used the norm bound $\norm{\Im G_s} \le \eta_s^{-1}$, and  in the ultimate step we used the fact that $\eta_s > 0$ in $\abvD_s$. We now estimate the two integrals in the last line of \eqref{eq:av_mart_compute} separately. For the first integral, we use the imaginary part of \eqref{eq:z_flow} to obtain 
		\begin{equation}
			\int_{\tin}^{t\wedge\tau} \frac{\bigl\langle \Im M_s\bigr\rangle + \tfrac{1}{2}\eta_s}{N^2\eta_s^3} \rd s  = \int_{\tin}^{t\wedge\tau} \frac{-\rd\eta_s}{N^2\eta_s^3} \le \frac{1}{N^2\eta_{t\wedge\tau}^2}.
		\end{equation}
		For the second integral, we use the definition \eqref{eq:tau_1_def} of the stopping time $\tau$, and the imaginary part of \eqref{eq:m_flow} to deduce that
		\begin{equation}
			\biggl\lvert \int_{\tin}^{t\wedge\tau} \frac{\bigl\langle \Im G_s - \Im M_s\bigr\rangle}{N^2\eta_s^3} \rd s \biggr\rvert 
			\lesssim \biggl\lvert \int_{\tin}^{t\wedge\tau} \frac{N^{3\tar+ 3(\ell+\tfrac{1}{2}\arb)}}{N^3\eta_s^4} \rd s\biggr\rvert 
			\lesssim \biggl\lvert\int_{\tin}^{t\wedge\tau} \frac{N^{3\tar+ 3(\ell+\tfrac{1}{2}\arb)}}{ N^3\eta_s^4 \langle \Im M_{s}\rangle} \rd \eta_s\biggr\rvert \lesssim \frac{N^{-\scl+3\tar+ 3(\ell+\tfrac{1}{2}\arb)}}{N^2\eta_{t\wedge\tau}^2},
		\end{equation}
		where in the last inequality we used that $\langle \Im M_{s}\rangle N\eta_s \sim \rho_s(z_s)N\eta_s \gtrsim N^{\scl}$ for all $\tin \le s \le \tfin$ by \eqref{eq:abvD_t}. 
		
		Therefore, using the path-wise Burkholder-Davis-Gundy inequality (see Lemma 5.6 in \cite{cipolloni2023universality} and Appendix B.6, Eq. (18) in \cite{martingale}) and the fact that $\tar + K\arb \ll \scl$, we deduce that, with very high probability
		\begin{equation} \label{eq:av_mart_bound}
			\max_{\tin \le s \le t}\biggl\lvert \int_{\tin}^{s\wedge \tau} \frac{1}{\sqrt{N}} \sum_{ab } \partial_{ab} \bigl\langle G_s \bigr\rangle \rd \bigl(\Brwn_{s}\bigr)_{ab}  \biggr\rvert  \le \frac{N^{\arb}}{N\eta_{t\wedge\tau}}.
		\end{equation}
		
		Next, using the Ward identity and the definition \eqref{eq:tau_1_def} of the stopping time $\tau$, we obtain
		\begin{equation} \label{eq:propagator_bound}
			\biggl\lvert \frac{1}{2} + \bigl\langle G_s^2 \bigr\rangle \biggr\rvert 
			\le \frac{1}{2} + \frac{\bigl\langle \Im G_s \bigr\rangle}{\eta_s}
			\le  -\frac{1}{\eta_s}\frac{\rd \eta_s}{\rd s} + \frac{N^{3\tar + 3 (\ell+\tfrac{1}{2})\arb}}{N\eta_s^2}
			\le -\frac{1}{\eta_s}\frac{\rd \eta_s}{\rd s} \biggl(1 + C N^{-\scl + 3\tar + 3 (\ell+\tfrac{1}{2})\arb} \biggr),
		\end{equation}
		where in the last inequality we used the imaginary part of \eqref{eq:z_flow} and the bound $\langle \Im M_s \rangle N\eta_s \sim \rho_s(z_s)N\eta_s \gtrsim N^{\varepsilon}$ from \eqref{eq:abvD_t}. 
		By integrating \eqref{eq:Gav_evol}, it follows from the assumption of Proposition~\ref{prop:zig} at $t = t_{k-1} = \tin$  and \eqref{eq:av_mart_bound} that the bound
		\begin{equation}
			\bigl\lvert \bigl\langle G_{t\wedge\tau} - M_{t\wedge\tau}  \bigr\rangle \bigr\rvert \le -\biggl(1 + C N^{-\scl+\tar + (\ell+\tfrac{1}{2})\arb}\biggr)\biggl(\int_{\tin}^{t\wedge\tau}\frac{\bigl\lvert \bigl\langle G_s - M_s  \bigr\rangle \bigr\rvert}{\eta_s}\frac{\rd \eta_s}{\rd s} \rd s + \frac{N^{3\tar+3\ell\arb}}{N\eta_{t\wedge\tau}}\biggr),
		\end{equation}
		holds with very high probability. Here we used that $\tar \ll \scl$ from \eqref{eq:exponents}, and the assumption that $\ell\arb \le 2K\arb \ll \scl$. 	
		Applying the Gronwall inequality yields the very-high-probability bound,
		\begin{equation} \label{eq:stop_time_av_conclusion}
			\bigl\lvert \bigl\langle G_{t\wedge\tau} - M_{t\wedge\tau}  \bigr\rangle \bigr\rvert \le  \frac{N^{3\tar + 3(\ell+\tfrac{1}{4})\arb}}{N\eta_{t\wedge\tau}},
		\end{equation}
		uniformly in $\tin \le t \le \tfin$.

		Similarly, computing the quadratic variation of the martingale term in \eqref{eq:Giso_evol}, we obtain
		\begin{equation}
			\begin{split}
				\biggl[\int_{\tin}^\cdot \frac{1}{\sqrt{N}} \sum_{ab } \partial_{ab} \bigl (G_s \bigr)_{\bm u\bm v} \rd \bigl(\Brwn_{s}\bigr)_{ab} \biggr]_{t\wedge\tau} 
				&\le \int_{\tin}^{t\wedge\tau} \frac{\bigl(\Im G_s\bigr)_{\bm u \bm u}\bigl(\Im G_s\bigr)_{\bm v \bm v}}{N\eta_s^2} \rd s\\ 
				&\lesssim \int_{\tin}^{t\wedge\tau} \frac{\rho_s(z_s)^2}{N\eta_s^2} \biggl(1 + \frac{N^{\tar+(\ell+1)\arb} }{\sqrt{\rho_s(z_s)N\eta_s} }\biggr)^2 \rd s \lesssim \frac{\rho_{t\wedge\tau}(z_{t\wedge\tau})}{N\eta_{t\wedge\tau}},
			\end{split}
		\end{equation}
		where we used the imaginary part of \eqref{eq:m_flow} to obtain $\rho_{s}(z_s) \sim \rho_{t\wedge\tau}(z_{t\wedge\tau})$.
		Therefore, using the path-wise Burkholder-Davis-Gundy inequality, we deduce the very-high-probability bound
		\begin{equation} \label{eq:iso_mart_bound}
			\max_{\tin \le s \le t}\biggl\lvert \int_{\tin}^{s\wedge \tau} \frac{1}{\sqrt{N}} \sum_{ab } \partial_{ab} \bigl (G_s \bigr)_{\bm u\bm v} \rd \bigl(\Brwn_{s}\bigr)_{ab}   \biggr\rvert  \le N^\arb\sqrt{\frac{\rho_{t\wedge\tau}(z_{t\wedge\tau})}{N\eta_{t\wedge\tau}}} .
		\end{equation}
		Moreover, using the Ward identity within the Schwarz estimate 
		$|\bigl(G_s^2\bigr)_{\bm u \bm v}| \le \sqrt{(|G_s|^2)_{\bm u \bm u}(|G_s|^2)_{\bm v \bm v}}$, together with \eqref{eq:imM}, \eqref{eq:tau_1_def}, and the relations $\tar + K\arb \ll\scl$, we deduce that 
		\begin{equation}
			\bigl\lvert \bigl(G_s^2\bigr)_{\bm u \bm v} \bigr\rvert \lesssim \rho_s(z_s)\biggl(1  + \frac{N^{\tar + (\ell+\tfrac{1}{2})\arb}}{\sqrt{\rho_s(z_s)N\eta_s}}\biggr) \lesssim \rho_s(z_s).
		\end{equation}
		Therefore, from \eqref{eq:abvD_t} and the bound \eqref{eq:stop_time_av_conclusion}, we conclude that
		\begin{equation} \label{eq:iso_forcing_bound}
			\begin{split}
				\biggl\lvert \int_{\tin}^{t\wedge\tau} \bigl\langle G_s - M_s \bigr\rangle \bigl(G_s^2\bigr)_{\bm u \bm v}  \rd s \biggr\rvert 
				&\lesssim \int_{\tin}^{t\wedge\tau} \frac{N^{3\tar+3(\ell + \tfrac{1}{2})\arb}}{N\eta_s} \frac{\rho_s(z_s)}{\eta_s}\rd s \le \frac{N^{3\tar + 3(\ell+1)\arb}}{N^{\scl/2}}\sqrt{\frac{\rho_{t\wedge\tau} (z_{t\wedge\tau} )}{N\eta_{t\wedge\tau} }}.
			\end{split}
		\end{equation}
		Integrating \eqref{eq:Giso_evol}, and combining the assumption of Proposition~\ref{prop:zig} at time $t = t_{k-1} = \tin$,  \eqref{eq:iso_mart_bound} and \eqref{eq:iso_forcing_bound} yields
		\begin{equation} \label{eq:stop_time_iso_conclusion}
			\bigl\lvert \bigl( G_{t\wedge\tau} - M_{t\wedge\tau}  \bigr)_{\bm u \bm v} \bigr\rvert \le N^{\tar+(\ell+\tfrac{1}{2})\arb} \sqrt{\frac{\rho_{t\wedge\tau}(z)}{N\eta_{t\wedge\tau}}},
		\end{equation}
		uniformly in $\tin \le t \le \tfin$, with very high probability, for all $\bm u, \bm v \in \mathcal{V}$. Note that the term $\tfrac{1}{2}(G_t-M_t)_{\bm u \bm v}$ on the right-hand side of \eqref{eq:Giso_evol} and can be removed by differentiating $\ee^{-t/2}(G_t-M_t)_{\bm u \bm v}$ with the harmless prefactor $\ee^{-t/2} = 1 +\mathcal{O}(T)$.
		
		Hence, using \eqref{eq:stop_time_av_conclusion} and \eqref{eq:stop_time_iso_conclusion}, we conclude that $\tau = \tfin$ with very high probability, therefore establishing the isotropic local law and averaged local law in \eqref{eq:ll_template} for $B:= I$ with data $(\abvD_{\tfin}, \tar + (\ell+\tfrac{1}{2})\arb)$.
		
		For a general observable $B \in \mathbb{C}^{N\times N}$, we use the bound \eqref{eq:stop_time_av_conclusion} as input to obtain the very-high-probability estimate
		\begin{equation} \label{eq:B_forcing}
			\biggl\lvert	\bigl\langle G_s - M_s \bigr\rangle  \bigl\langle G_s^2 B \bigr\rangle \biggr\rvert \le \frac{N^{3\tar + 3(\ell+\tfrac{1}{4})\arb}}{N\eta_s} \frac{\bigl\langle \Im G_s \bigr\rangle^{1/2} \bigl\langle \Im G_s BB^* \bigr\rangle^{1/2}}{\eta_s} \lesssim \frac{N^{3\tar + 3(\ell+\tfrac{1}{4})\arb}}{N\eta_s} \frac{\rho_s(z_s)}{\eta_s} \norm{B}_{\mathrm{hs}}.
		\end{equation}
		uniformly in $\tin \le s \le \tfin$. Here, in the last step we used the isotropic bound \eqref{eq:stop_time_iso_conclusion} for the eigenvectors $\bm v_j$ of $BB^*$, corresponding to the eigenvalues $|\sigma_j|^2$, to conclude that, with very high probability,
		\begin{equation} \label{eq:zig_isotrick}
			\langle \Im G BB^* \rangle = \frac{1}{N}\sum_j |\sigma_j|^2  (\Im G)_{\bm v_j \bm v_j} \lesssim \rho_s(z_s) \Vert B \Vert_{\rm hs}^2. 
		\end{equation}
		Similarly, using \eqref{eq:zig_isotrick}, we estimate the quadratic variation of the corresponding martingale term in \eqref{eq:Gav_evol} for general $B$, 
		\begin{equation} \label{eq:B_mart}
			\biggl[\int_{\tin}^\cdot \frac{1}{\sqrt{N}} \sum_{ab } \partial_{ab} \bigl\langle G_s B \bigr\rangle \rd \bigl(\Brwn_{s}\bigr)_{ab} \biggr]_{t\wedge\tau} \lesssim \frac{N^\arb}{N^2\eta_{t\wedge\tau}^2}\norm{B}_{\mathrm{hs}}.
		\end{equation}
		Combining \eqref{eq:Gav_evol}, \eqref{eq:B_forcing}, and \eqref{eq:B_mart}, we conclude that the resolvent $G_{\tfin}$ satisfies the averaged local law in \eqref{eq:ll_template}  with data $(\abvD_{\tfin}, \tar + (\ell+1)\arb)$ for any $B \in \mathbb{C}^{N\times N}$. This concludes the proof of Proposition~\ref{prop:zig}. \qed

		\section{Green function comparison: Proof of Proposition~\ref{prop:zag}} \label{sec:zag}
		The goal of this section is to prove Proposition~\ref{prop:zag} and thereby conclude the argument for the \emph{zag} step of our proof. {For simplicity, we will carry out the proof only in the real symmetric case; the complex Hermitian case can be dealt with minor modifications and is thus omitted.} Moreover, since throughout the argument the time $t_k$ defined in \eqref{eq:t_steps} remains fixed, for the remainder of this section, we drop the superscript $t_k$ from $\abvD_{t_{k}}, \rho_{t_{k}}$, and $M_{t_k}$. To further condense the notation, we abbreviate $\mathcal{D}:= \abvD_{t_{k}}$ and $s_{\mathrm{final}} := s(\dift_k)$. %
		
		The proof will be conducted iteratively along vertical truncations of the domain $\mathcal{D}$, defined as
		\begin{equation}
			\mathcal{D}_{\gamma} \equiv \abvD_{t_{k}, \gamma} := \bigl\{ z := E+\ii\eta \in\mathcal{D} \equiv \abvD_{t_{k}}: \eta \ge N^{-1+\gamma} \}, \quad 0 < \gamma  \le 1.
		\end{equation} 
		This is formalized in the following proposition, which we prove in Section~\ref{subsec:gronwall}. 
		\begin{proposition} [Zag Bootstrap] \label{prop:zag_bootstrap}
			Fix a constant $0 < \gamma_0 \le 1$ and assume that the very-high-probability bounds
			on the matrix elements of the resolvent \eqref{eq:G^s} 
			\begin{equation} \label{eq:bootstrap_init}
				\bigl\lvert \bigl(G^s(z)\bigr)_{\bm u \bm v} \bigr\rvert \lesssim  1, \quad \bigl\lvert \bigl(\Im G^s(z)\bigr)_{\bm u \bm u} \bigr\rvert \lesssim  \rho(z),
			\end{equation}
			hold uniformly in $z \in \mathcal{D}_{\gamma_0}$ and $ s \in [0, s_{\mathrm{final}}]$, for any deterministic $\bm u, \bm v \in \mathbb{C}^N$ with $\norm{\bm u }=\norm{\bm v } =1$.
			
			Fix $\gamma_1 \ge \gamma_0 - \step$   with $\step < \mu$ satisfying $\step \ll \tar$,    %
			and assume that for some ${\arb} > 0$ and ${\ell} \in \mathbb{N}$, the resolvent $G^s$ satisfies the local laws \eqref{eq:ll_template} with data $(\mathcal{D}_{\gamma_1}, \tar+{\ell}{\arb})$ at time $s = s_{\rm final}$. Then the resolvent $G^s$ satisfies the local laws \eqref{eq:ll_template} with data $(\mathcal{D}_{\gamma_1}, \tar+({\ell}+1){\arb})$ uniformly in $s \in  [0, s_{\rm final}]$.
		\end{proposition}
		Armed with Proposition~\ref{prop:zag_bootstrap}, we can easily conclude Proposition~\ref{prop:zag}.
		\begin{proof}[Proof of Proposition~\ref{prop:zag}] The proof goes
			via  induction in $\gamma(k) := 1 - k\step$ 
			by iteratively applying Proposition~\ref{prop:zag_bootstrap}. %
			As the base case, clearly, the estimates \eqref{eq:bootstrap_init} hold for $\gamma_0 = \gamma(0) = 1$ as a direct consequence of the bounds $\norm{G^s(E+\ii\eta)} \le \eta^{-1}$ and $\rho(E+\ii\eta) \sim 1$ for $\eta \sim 1$. Moreover, the resolvent $G^s$ satisfies the local laws \eqref{eq:ll_template} with data $(\mathcal{D}_{\gamma_1}, \tar+{\ell}{\arb})$ {with $\gamma_1 = \gamma(1)$} at time $s = s_{\rm final}$ by assumption. 
			Hence, the resolvent $G^s$ satisfies the local laws \eqref{eq:ll_template} with data $(\mathcal{D}_{\gamma(1)}, \tar + ({\ell} + 1) {\arb})$ uniformly in $s \in [0, s_{\rm final}]$ by Proposition~\ref{prop:zag_bootstrap}. As a consequence, since $\tar + (\ell +1) \arb \ll \scl$, we have that the resolvent $G^s$ satisfies the bounds \eqref{eq:bootstrap_init} uniformly in $z \in \mathcal{D}_{\gamma(1)}$ and $s \in [0, s_{\rm final}]$

			As the induction step, assume now that for an integer $k\ge 1$ the resolvent $G^s$ satisfies the bounds \eqref{eq:bootstrap_init} uniformly in $z \in \mathcal{D}_{\gamma(k)}$ and $s \in [0, s_{\rm final}]$. (Recall that, as above, $G^s$ satisfies the local laws \eqref{eq:ll_template} with data $(\mathcal{D}_{\gamma(k)}, \tar+{\ell}{\arb})$ at time $s = s_{\rm final}$ by assumption.) Therefore, the resolvent $G^s$ satisfies the local laws \eqref{eq:ll_template} with data $(\mathcal{D}_{\gamma(k+1)}, \tar + ({\ell} + 1) {\arb})$ uniformly in $s \in [0, s_{\rm final}]$ by Proposition~\ref{prop:zag_bootstrap}. Note that after $K':=  \lceil (1+ \scl)/\step \rceil  \sim 1$ steps, $\mathcal{D}_{\gamma(K')} = \mathcal{D}$ and we have hence proven Proposition~\ref{prop:zag}. %
		\end{proof}
		
		It thus remains to prove Proposition~\ref{prop:zag_bootstrap}. We begin by collecting several preliminaries in Section~\ref{subsec:zag_prelim}. Afterwards, in Section~\ref{subsec:gronwall} we give the proof of Proposition~\ref{prop:zag_bootstrap} based on average and isotropic \emph{Gronwall estimates}. These bounds are proven in Sections~\ref{subsec:zag_av} and~\ref{subsec:zag_iso}, respectively.

		\subsection{Preliminaries} \label{subsec:zag_prelim}
		In order to perform the GFT, i.e.,~compare initial and final $W$'s, given by $W^t = H^t- A$
		with $H^t$ being the solution to \eqref{eq:zag_flow},   we employ Itô's formula: For a $C^2$-function $f(W^t)$, it holds that%
		\begin{equation} \label{eq:cumexstart}
			\frac{\rd }{\rd t} \E f(W^t) = - \frac{1}{2}\E \sum_\alpha w_\alpha(t) (\partial_\alpha f)(W^t) + \frac{1}{2N} \sum_{\alpha, \beta}\kappa_t(\alpha, \beta) \E (\partial_\alpha \partial_\beta f)(W^t) \,,
		\end{equation}
		where $\kappa_t(\alpha, \beta)$ denotes the (normalized, recall \eqref{eq:cumulants}) second order cumulant of $w_\alpha(t)$ and $w_\beta(t)$, the matrix entries of $W^t$.    
		The first summand on the rhs.~of \eqref{eq:cumexstart} can now be further treated by cumulant expansion, which is first 
		key ingredient for our proof. 
		
		\begin{proposition}[Multivariate cumulant expansion; cf.~Proposition~3.2 in \cite{slowcorr} and Lemma 3.1 in \cite{HeKnowles}] \label{prop:cumex}
			Let  %
			$f : \R^{N \times N} \to \C$ be a $L$ times differentiable function with bounded derivatives. Let $W$ be a random matrix, whose normalized cumulants satisfy Assumption~\ref{ass:cumulants}. Then, for any index $\alpha_0 \in [N]^2$ it holds that %
			\begin{equation} \label{eq:cumulantexpansion}
				\E w_{\alpha_0} f(W) = \sum_{k=0}^{L-1} \sum_{\bm \alpha \in \mathcal{N}(\alpha_0)^k} \frac{\kappa(\alpha_0 , \bm \alpha) }{N^{(k+1)/2} k!} \E (\partial_{\bm \alpha} f)(W)+ \Omega_L(f, \alpha_0),
			\end{equation}
			where $\bm \alpha = (\alpha_1, ... , \alpha_k)$ and $\partial_{\bm \alpha} = \partial_{w_{\alpha_1}} ... \partial_{w_{\alpha_k}}$ for $k \ge 1$, and for $k=0$ 
			is %
			considered as the function $f$ itself. %
			Moreover, the error term in \eqref{eq:cumulantexpansion} satisfies
			\begin{equation} \label{eq:errorbound}
				\big|\Omega_L(f, \alpha_0)\big| \lesssim \frac{C_L}{N^{(L+1)/2}} \sum_{\bm \alpha \in \mathcal{N}(\alpha_0)^L} \sup_{ \lambda \in [0,1]} \left( \E \big| (\partial_{\bm \alpha}f)(\lambda W \vert_{\mathcal{N}(\alpha_0)} + W\vert_{[N]^2 \setminus \mathcal{N}(\alpha_0)}) \big|^2\right)^{1/2},
			\end{equation}
			for some constant $C_L > 0$ depending only on $L$. The notation $W\vert_{\mathcal{N}}$ for $\mathcal{N} \subset [N]^2$ in \eqref{eq:errorbound} refers to the matrix which equals $W$ at all entries $\alpha \in \mathcal{N}$ and is zero otherwise. 
		\end{proposition}
		Note that the $k=1$  term   %
		in the expansion of the first summand on the rhs.~of \eqref{eq:cumexstart}   exactly cancels the second summand on the rhs.~of \eqref{eq:cumexstart}. For Proposition~\ref{prop:cumex} being practically applicable we need to control (i) every order of the expansion, and (ii) the truncation term $\Omega$. These will be guaranteed by Assumption~\ref{ass:cumulants} above.

		The second key input required for the GFT argument %
		is the following monotonicitiy estimate on resolvents. 
		\begin{lemma}[Monotonicity estimate] \label{lem:monotone}
			Fix a constant $0 < \gamma_0 \le 1$ and assume that the very-high-probability bounds \eqref{eq:bootstrap_init} hold uniformly in $z \in \mathcal{D}_{\gamma_0}$ and $ s \in [0, s_{\mathrm{final}}]$, for any deterministic $\bm u, \bm v \in \mathbb{C}^N$ with $\norm{\bm u }=\norm{\bm v } =1$.
			
			Fix $\gamma_1 \ge \gamma_0 - \step$. Then, we have 
			\begin{equation} \label{eq:monotone}
				|G^s(E + \ii \eta_1)_{\bm u \bm v}| \lesssim \frac{\eta_0}{\eta_1}\,, \qquad |\Im G^s(E + \ii \eta_1)_{\bm u \bm u}| \lesssim \rho(E + \ii \eta_0) \frac{\eta_0}{\eta_1}\,, 
			\end{equation}
			with very high probability, 
			uniformly in $z:= E+ \ii \eta_1 \in \mathcal{D}_{\gamma_1}$ 	for any  $\eta_0 \ge N^{-1+\gamma_0} \vee \eta_1 $, time $ s \in [0, s_{\mathrm{final}}]$, and for any deterministic vectors $\bm u, \bm v \in \C^N$ with $\norm{\bm u} = \norm{\bm v} = 1$.  

		\end{lemma}
		We defer the proof of Lemma~\ref{lem:monotone} to Appendix~\ref{app:tech}.
		
		\subsection{Gronwall estimates: Proof of Proposition~\ref{prop:zag_bootstrap}} \label{subsec:gronwall}
		In this section, we provide the proof of Proposition~\ref{prop:zag_bootstrap} based on two Gronwall estimates, formulated in Propositions~\ref{prop:gronwalliso}--\ref{prop:gronwallav} below
		that will be proven in the next subsection.  The isotropic part of Proposition 
		\ref{prop:zag_bootstrap} will be concluded in a self contained way, based entirely on the \emph{isotropic Gronwall estimate} in Proposition~\ref{prop:gronwalliso}. %
		Its conclusion in \eqref{eq:avzaginput} then serves as an input for the \emph{average Gronwall estimate} in Proposition~\ref{prop:gronwallav}. 
		\begin{proposition}[Isotropic Gronwall estimate] \label{prop:gronwalliso} Assume the conditions of Proposition~\ref{prop:zag_bootstrap}. 
			Fix $\bm x, \bm y \in \C^N$ of bounded norm, $z := E + \ii \eta_1 \in \mathcal{D}_{\gamma_1}$ and $\eta_0 \ge N^{-1+\gamma_0} \vee \eta_1$ such that $\eta_0/\eta_1 \le N^\step$. For $s \in [0,s_{\rm final}]$, define
			\begin{equation}
				S_s := \big(G^s(E+ \ii \eta_1) - M(E+ \ii \eta_1)\big)_{\bm x \bm y} \,.
			\end{equation}
			Then, for any (large) even $p \in \N$, %
			it holds that 
			\begin{equation} \label{eq:Gronwalliso}
				\left| \frac{\dif }{\dif s} \E |S_s|^{p} \right| \lesssim \left(1 + N^{10 \step} \sqrt{\frac{\rho(E+\ii\eta_0)}{\eta_0}} \, \right) \,  \big[ \E |S_s|^{p} + \left(\Psi(\eta_1)\right)^{p} \big], 
			\end{equation}
			uniformly in  $ s \in [0, s_{\rm final}]$, bounded $\bm x, \bm y \in \C^N$, and $z \in \mathcal{D}_{\gamma_1}$. Here, for $\eta \in [\eta_0, \eta_1]$, we denoted 
			\begin{equation} \label{eq:Psiabvdef}
				\Psi(\eta) := \sqrt{\frac{\rho(E+ \ii \eta)}{N \eta}} \,. 
			\end{equation}
		\end{proposition}
		By Gronwall's lemma, uniformly in $s \in [0, s_{\rm final}]$,  from \eqref{eq:Gronwalliso} we find that 
		\begin{equation} \label{eq:isoint}
			\begin{split}
				\E |S_s|^p &\lesssim \exp\left(\left(1 + N^{10 \step} \sqrt{\frac{\rho(E+\ii\eta_0)}{\eta_0}} \, \right) \, (s_{\rm final} - s)\right) \, \big[\E |S_{s_{\rm final}}|^{p} + \left(\Psi(\eta_1)\right)^{p}\big] \\
				&\lesssim \exp(N^{-\tar/10}) \, \big[\E |S_{s_{\rm final}}|^{p} + \left(\Psi(\eta_1)\right)^{p}\big] \lesssim \E |S_{s_{\rm final}}|^{p} + \left(\Psi(\eta_1)\right)^{p}.
			\end{split}
		\end{equation}
		Here we used that $\rho(E + \ii \eta_0)/\eta_0 \lesssim N^{k \step}/T$ by \eqref{eq:abvD_t}, $s_{\rm final} \lesssim N^{-(k-1) \step}T$ by \eqref{eq:s(t)_est}, $T \sim N^{-\tar/4}$ from \eqref{eq:term_time}, and $\step \ll \tar$ by \eqref{eq:exponents}. 
		
		To estimate $ \E |S_{s_{\rm final}}|^{p} $, recall that the resolvent $G^s$ satisfies the isotropic local law in \eqref{eq:ll_template} with data $(\mathcal{D}_{\gamma_1}, \tar + {\ell} {\arb})$ at $s = s_{\rm final}$. Therefore, since $p$ in \eqref{eq:isoint} was arbitrary, we find that %
		\begin{equation} \label{eq:avzaginput}
			\left|\big(G^s(z) - M(z)\big)_{\bm x\bm y}\right| \le N^{\tar + ({\ell} + 1) {\arb}} \sqrt{\frac{\rho(z)}{N \eta_1}} ,
		\end{equation}
		uniformly in $z := E+ \ii \eta_1 \in \mathcal{D}_{\gamma_1}$, $s \in [0, s_{\rm final}]$, and bounded $\bm x, \bm y \in \C^N$, with very high probability. 
		
		This proves the isotropic part of Proposition~\ref{prop:zag_bootstrap} and we are left with the average part. 
		
		\begin{proposition}[Average Gronwall estimate] \label{prop:gronwallav} Fix $B \in \C^{N \times N}$ of bounded Hilbert--Schmidt norm, $\Vert B \Vert_{\rm hs} \le 1$, $z := E + \ii \eta \in \mathcal{D}_{\gamma_1}$, and $\eta_0 \ge N^{-1+\gamma_0} \vee \eta_1$ such that $\eta_0/\eta_1 \le N^\step$. For $s \in [0, s_{\rm final}]$, define 
			\begin{equation}
				R_s := \langle (G^s(E+ \ii \eta_1) - M(E+ \ii \eta_1))B \rangle \,. 
			\end{equation}
			Moreover, suppose that \eqref{eq:avzaginput} holds uniformly in $z := E+ \ii \eta_1 \in \mathcal{D}_{\gamma_1}$, $s \in [0, s_{\rm final}]$, and bounded $\bm x, \bm y \in \C^N$. %
			Then, for any (large) even $p \in \N$ it holds that 
			\begin{equation} \label{eq:Gronwallav}
				\left| \frac{\dif }{\dif s} \E |R_s|^{p} \right| \lesssim \left(1 + N^{-2\step } \frac{\rho(E + \ii \eta_0)}{\eta_0}\right) \, \left[ \E |R_s|^{p} + \left(\frac{N^{3\tar}}{N \eta_1}\right)^{p} \right] , 
			\end{equation}
			uniformly in  $ s \in [0, s_{\rm final}]$, bounded $B \in \C^{N\times N}$, and $z \in \mathcal{D}_{\gamma_1}$.
		\end{proposition}
		Analogously to \eqref{eq:isoint}, by Gronwall's lemma, uniformly in $s \in [0, s_{\rm final}]$, we find that 
		\begin{equation} \label{eq:avint}
			\begin{split}
				\E |R_s|^p &\lesssim \exp\left( \left(1 + N^{-2\step } \frac{\rho(E + \ii \eta_0)}{\eta_0}\right)  \, (s_{\rm final} - s)\right) \, \big[\E |R_{s_{\rm final}}|^{p} + \left(\frac{N^{3 \tar} }{N \eta_1}\right)^{p}\big] \\
				&\lesssim \exp(N^{-\step}) \, \big[\E |R_{s_{\rm final}}|^{p} + \left(\Psi(\eta_1)\right)^{p}\big] \lesssim \E |R_{s_{\rm final}}|^{p} + \left(\frac{N^{3 \tar} }{N \eta_1}\right)^{p} .
			\end{split}
		\end{equation}
		Here we used that $\rho(E + \ii \eta_0)/\eta_0 \lesssim N^{k \step}/T$ by \eqref{eq:abvD_t}, $s_{\rm final} \lesssim N^{-(k-1) \step}T$ by \eqref{eq:s(t)_est}, $T \sim N^{-\tar/4}$ by \eqref{eq:term_time}, and $\step \ll \tar$ by \eqref{eq:exponents}. 
		Note that the small prefactor $N^{-2\delta}$ in \eqref{eq:Gronwallav} is absolutely essential, unlike in the
		isotropic case \eqref{eq:isoint}, where a large prefactor $N^{10\delta}$ is affordable thanks to the square root. The linear appearance of $\rho/\eta$ in \eqref{eq:Gronwallav} is only due to fact that we estimate $B$ in terms of its Hilbert--Schmidt norm $\Vert B \Vert_{\rm hs}$; cf.~the estimate in \eqref{eq:k3_1chain}. For observables with $\Vert B \Vert \sim \Vert B \Vert_{\rm hs}$, such as the identity matrix $B = \mathbf{1}$,  the linear dependence on $\rho/\eta$ can be improved to a $\sqrt{\rho/\eta}$. We exploit this fact in \eqref{eq:avimprove} below.  

		Recall that the resolvent $G^s$ satisfies the average local law in \eqref{eq:ll_template} with data $(\mathcal{D}_{\gamma_1}, \tar + {\ell} {\arb})$ at $s = s_{\rm final}$. Therefore, since $p$ in \eqref{eq:avint} was arbitrary, we find that %
		\begin{equation*} 
			\left|\big\langle \big(G^s(z) - M(z)\big)B \big\rangle\right| \le \frac{N^{3(\tar + ({\ell} + 1) {\arb})}}{N \eta_1} ,
		\end{equation*}
		uniformly in $z := E+ \ii \eta_1 \in \mathcal{D}_{\gamma_1}$, $s \in [0, s_{\rm final}]$, and $ B \in \C^{N \times N}$ with $\Vert B \Vert_{\rm hs} \le 1$, with very high probability. 
		
		This concludes the proof of Proposition~\ref{prop:zag_bootstrap}. \qed

		\subsection{Cumulant expansion: Proofs of Propositions~\ref{prop:gronwallav} and~\ref{prop:gronwalliso}} The proofs of Propositions~\ref{prop:gronwalliso}--\ref{prop:gronwallav} are based on the multivariate cumulant expansion from Proposition~\ref{prop:cumex} and the monotonicity estimate from Lemma~\ref{lem:monotone}. 
		We begin by proving the average Gronwall estimate in Proposition~\ref{prop:gronwallav}. Moreover, we will henceforth omit the superscript $s$ from the resolvent $G^s$.

		\subsubsection{Average case: Proof of Proposition~\ref{prop:gronwallav}} \label{subsec:zag_av}
		Throughout the proof, we will assume that $\Vert B \Vert_{\rm hs} \lesssim 1$. 	By \eqref{eq:cumexstart} for $R_s$ we have 
		\begin{equation} \label{eq:GFTstart}
			\frac{\dif}{\dif s} \E |R_s|^{p} = - \frac{1}{2} \E \sum_{\alpha_1} w_{\alpha_1}(s) (\partial_{\alpha_1} |R_s|^{p}) + \frac{1}{2} \sum_{\alpha_1, \alpha_2} \kappa_s(\alpha_1, \alpha_2) \E \bigl[\partial_{\alpha_1} \partial_{\alpha_2} |R_s|^{p}\bigr],
		\end{equation}
		where $w_{\alpha_i}(s)$ is the $\alpha_i$-th entry of $W_s$, $\kappa_s(\alpha_1, \alpha_2, ...)$ is a joint normalized cumulant of $w_{\alpha_1}(s), w_{\alpha_2}(s), ...$ and $\partial_{\alpha_i} = \partial_{w_{\alpha_i}(s)}$ denotes the partial derivative in the direction of $w_{\alpha_i}(s)$.  
		
		The first term on the rhs.~of \eqref{eq:GFTstart} can now be expanded by means of Proposition~\ref{prop:cumex}:
		\begin{equation} \label{eq:cumexres}
			\E  \bigl[ w_{\alpha_1}(s) (\partial_{\alpha_1} |R_s|^{p}) \bigr]  = \sum_{k=0}^{L-1} \sum_{\bm \alpha \in \mathcal{N}(\alpha_1)^k}\frac{\kappa_s(\alpha_1, \bm \alpha)}{N^{(k+1)/2} \, k!} \E \bigl[\partial_{\alpha_1} \partial_{\bm \alpha} |R_s|^{p}\bigr] + \Omega_L \,. 
		\end{equation}
		Since $L$ derivatives of $|R_s|^{p}$ create $L$ additional resolvent matrix elements (where each of them is bounded with the aid of Lemma~\ref{lem:monotone}) and using that $|\mathcal{N}(\alpha_1)| \lesssim N^{1/2-\mu}$ by Assumption~\ref{ass:cumulants}~(ii), the error term $\Omega_L$ can be estimated as\footnote{To be precise, note some of the
			$p+L$ resolvents in the error term $\Omega_L$ are actually resolvents of the random matrix $W^{(\lambda)} := \lambda W \vert_{\mathcal{N}(\alpha_0)} + W\vert_{[N]^2 \setminus \mathcal{N}(\alpha_0)}$ (recall \eqref{eq:errorbound})
			and we need to guarantee their boundedness as well, uniformly in $\lambda\in[0,1]$. We  
			perform a resolvent expansion of $G^{(\lambda)} := (A + W^{(\lambda)} - z)^{-1}$ up to some order $\tilde{m}\in \N$
			around  $G^{(1)}$
			whose boundedness is known.  
			For each $G^{(\lambda)}$, the $m^{\rm th}$ order term in this expansion can be bounded by $N^{\step} N^{m(\step - \mu + \arb)}$ with the aid of Lemma~\ref{lem:monotone} (to bound $G^{(1)}$ isotropically) and using the norm estimate $\Vert W\vert_{\mathcal{N}(\alpha_0)} \Vert \le N^{-\mu + \arb}$, w.v.h.p.~for any $\arb > 0$, which is a consequence of Assumption~\ref{ass:cumulants}~(ii). By a simple norm bound $\Vert G^{(\lambda)} \Vert \le \eta^{-1}$, the last   truncation term in the resolvent expansion admits the bound $N^{\step} N^{\tilde{m}(\step - \mu + \arb)} \eta^{-1}$.   Therefore, since $\eta$ depends at most polynomially on $N$ and $\mu > \step + \arb$ for some $\arb> 0$ small enough, the resolvent expansion can be truncated at finite order, leaving us with the bound $N^{\step}$ for every matrix element of  $G^{(\lambda)}$ employed in \eqref{eq:Omegaest}, uniformly in $\lambda \in [0,1]$. %
		}
		\begin{equation} \label{eq:Omegaest}
			|\Omega_L| \lesssim N^{-\frac{L+1}{2}} N^{L(1/2 - \mu)} N^{(p+L)\step} \lesssim N^{2p\step + L(\step- \mu)}. 
		\end{equation} 
		Using the relation $\mu > \step$ from \eqref{eq:exponents}
		and $L := \lceil((1+\step)p+2)/(\mu - \step) \rceil$, we see that
		$|\Omega_L| \le N^{-2}(N \eta_1)^{-p}$ (the factor $N^{-2}$ is needed to bound the summation over $ \alpha_1$ in \eqref{eq:GFTstart}).  
		With this choice of $L$, the error term $\Omega_L$ will henceforth be ignored. 
		
		Plugging \eqref{eq:cumexres} into \eqref{eq:GFTstart} and using that
		the $k=0$ term is zero by  $\kappa_s(\alpha_1) = \E w_{\alpha_1}(s)= 0$,  and that   the $k=1$ term in \eqref{eq:cumexres} cancels the second term on the rhs.~of \eqref{eq:GFTstart}, we obtain
		\begin{equation} \label{eq:GFTcumulants}
			\left| \frac{\dif}{\dif s} \E |R_s|^{p}  \right| \lesssim \left|  \sum_{k=2}^{L-1} \sum_{\alpha_1} \sum_{\bm \alpha \in \mathcal{N}(\alpha_1)^k}\frac{\kappa_s(\alpha_1, \bm \alpha)}{N^{(k+1)/2} \, k!} \E (\partial_{\alpha_1} \partial_{\bm \alpha} |R_s|^{p})  \right| + \left(\frac{1}{N \eta_1}\right)^{p} \,. 
		\end{equation}
		
		We will now first estimate the third order cumulant terms (i.e.~those with $k=2$ in \eqref{eq:GFTcumulants}), as these are the most delicate, and afterwards turn to the higher order ones that can be handled by simple power counting
		with a little twist due to the Hilbert--Schmidt norm of the observable $B$. 
		Moreover, we drop the time dependence of $R_s$ and $\kappa_s$ whenever it does not lead to confusion. We point out that Assumption~\ref{ass:cumulants} also holds for $W^s$ from \eqref{eq:zag_flow}, uniformly in %
		$s \in [0, \infty)$.   Indeed, adding an independent Gaussian random matrix to $W_0$ has no effect on cumulants of order $k \ge3$ (by Gaussianity) and leaves the first two joint moments as well as the independence property of Assumption~\ref{ass:cumulants}~(ii) invariant (the covariance tensor $\Sigma$ is trivial beyond the
		range $\mathcal{N}(\alpha_1)$)   by construction \eqref{eq:zag_flow}. In particular, we can freely extend the summation over $\bm \alpha \in \mathcal{N}(\alpha_1)^k$ in \eqref{eq:GFTcumulants} to $\bm \alpha \in ([N]^2)^k$ and combine the latter two summations in \eqref{eq:GFTcumulants} into $\sum_{\alpha_1, \bm \alpha}$.

		Now, for the third order cumulant terms, we aim to control 
		\begin{equation*}
			\left| N^{-3/2}\sum_{\alpha_1, \alpha_2, \alpha_3} \kappa(\alpha_1, \alpha_2, \alpha_3) \E (\partial_{\alpha_1} \partial_{\alpha_2} \partial_{\alpha_3} |R|^p) \right|,
		\end{equation*}
		which, after employing the Leibniz rule, can be broken up into terms of the form $(\partial_\alpha^3 R) |R|^{p-1}$, $(\partial_\alpha R) (\partial_\alpha^2 R) |R|^{p-2}$, and $(\partial_\alpha R)^3 |R|^{p-3}$. To further ease the notation, here and in the following, we neglect the difference between $R$ and $\overline{R}$, as these will be estimated in a completely analogous way. %
		
		We begin with the terms of the form $(\partial_\alpha^3 R)|R|^{p-1}$, which requires the bound  \eqref{eq:kappa_3_av_norm} in Assumption~\ref{ass:cumulants}~(i). Writing $\langle G B\rangle = N^{-1} \sum_j (GB)_{jj}$ and identifying $\alpha_i \equiv (a_i, b_i) \in [N]^2$, we aim to estimate (ignoring the $|R|^{p-1}$-factor)  
		\begin{equation*}	
			N^{-5/2} \left| \sum_{j, \alpha_1, \alpha_2, \alpha_3}  \kappa(\alpha_1, \alpha_2, \alpha_3)G_{ja_1} G_{b_1a_2} G_{b_2a_3} (GB)_{b_3j} \right|  = N^{-5/2} \left| \sum_{\alpha_1, \alpha_2, \alpha_3}  \kappa(\alpha_1, \alpha_2, \alpha_3)
			G_{b_1a_2} G_{b_2a_3} (GBG)_{b_3a_1} \right| \,.   
		\end{equation*}
		For both $G_{b_1a_2}$ and $G_{b_2a_3}$ we write $G_{ba} = M_{ba} + (G-M)_{ba}$ and use $\Vert M \Vert \lesssim 1$ for the $M$-term and the bound \eqref{eq:avzaginput} for the $(G-M)$-term. 
		In particular (recalling the notation \eqref{eq:Psiabvdef}),
		\begin{equation} \label{eq:3cumk=1}
			\begin{split}
				&N^{-5/2} \left| \sum_{\alpha_1, \alpha_2, \alpha_3}  \kappa(\alpha_1, \alpha_2, \alpha_3) M_{b_1a_2} (G-M)_{b_2a_3}  (GBG)_{b_3a_1}   \right| \\
				&\quad\lesssim   N^{-5/2} \, N^{\tar + ({\ell} + 1) \arb} \, \Psi(\eta_1)  \sum_{ \alpha_1, \alpha_2, \alpha_3}  |\kappa(\alpha_1, \alpha_2, \alpha_3)| \,   |(GBG)_{b_3 a_1}| \\
				&\quad\lesssim N^{-5/2} \,N^{\tar + ({\ell} + 1) {\arb}} \,  \Psi(\eta_1) \sum_{ \alpha_1, \alpha_2, \alpha_3}  |\kappa(\alpha_1, \alpha_2, \alpha_3)| \, |(GBB^*G^*)_{b_3b_3}|^{1/2} |(GG^*)_{a_1 a_1}|^{1/2}  \\
				&\quad\lesssim  N^{\tar + ({\ell} + 1) {\arb}} \, \Psi(\eta_1)^2 \left\Vert \sum_{\alpha_2} |\kappa(*, \alpha_2, *)|\right\Vert \, \langle GG^* BB^* \rangle^{1/2} \lesssim  N^{-\step} \sqrt{\frac{\rho(E + \ii \eta_0)}{\eta_0}} \, \frac{N^{3 \tar}}{N \eta_1}\,,
			\end{split}
		\end{equation}
		with very high probability.
		In the second step, we used the Schwarz inequality. In the ultimate step, similarly to \eqref{eq:zig_isotrick}, we used \eqref{eq:summcum}, the Ward identity, the spectral decomposition of $BB^*$, and \eqref{eq:avzaginput} together with $(\Im M)_{\bm v \bm v}\lesssim
		\rho(E + \ii \eta_1)$ by \eqref{eq:imM}, to obtain
		\begin{equation} \label{eq:isotrick}
			\langle \Im G BB^* \rangle = \frac{1}{N}\sum_j |\sigma_j|^2  (\Im G)_{\bm v_j \bm v_j} \lesssim \rho(E + \ii \eta_1) \Vert B \Vert_{\rm hs}^2 \,,
		\end{equation}
		and used $\step \ll \tar$ by \eqref{eq:exponents},  %
		and the fact that ${\arb} > 0$ is arbitrarily small. Note that the small factor $N^{-\step}$ in the last line of \eqref{eq:3cumk=1} is balanced by an additional $N^\tar$.   
		The terms with $(G-M)_{b_1 a_2} G_{b_2 a_3}$ and  $(G-M)_{b_1 a_2} (G-M)_{b_2 a_3}$ are treated analogously and we are thus left with the $M_{b_1 a_2} M_{b_2 a_3}$-term. Here, using the $\vertiii{\kappa}_3^{\rm av}$ norm from \eqref{eq:kappa_3_av_norm}, we estimate
		\begin{equation} \label{eq:k3_1chain}
			\begin{split}
				N^{-5/2} &\left| \sum_{ \alpha_1, \alpha_2, \alpha_3}  \kappa(\alpha_1, \alpha_2, \alpha_3) M_{b_1a_2} M_{b_2a_3}  (GBG)_{b_3a_1}   \right| \\
				&\le N^{-1} \vertiii{\kappa}_3^{\rm av}\Vert M \Vert^2 \Vert GBG \Vert_{\rm hs} \lesssim \eta_1^{-1/2} \frac{1}{N \eta_1} \langle \Im G BB^* \rangle^{1/2} \lesssim N^{-\step} \sqrt{\frac{\rho(E + \ii \eta_0)}{\eta_0}} \,  \frac{N^{3 \tar}}{N \eta_1} \,,
			\end{split}
		\end{equation}
		with very high probability. In the penultimate step we used the definition of $\Vert \cdot \Vert_{\rm hs}$ together with a Ward identity and the trivial bound $\Vert G \Vert \le \eta_1^{-1}$; in the last step we employed \eqref{eq:isotrick} and $\eta_0/\eta_1 \le N^\step$  together with monotonicity of $\eta \mapsto \eta \rho(E+ \ii \eta)$ and $\step \ll \tar$. Hence, by two Young inequalities , we thus find  %
		\begin{equation} \label{eq:3cumk=1final}
			\left| N^{-3/2} \sum_{ \alpha_1, \alpha_2, \alpha_3}  \kappa(\alpha_1, \alpha_2, \alpha_3)\E \bigl[ (\partial_{\alpha_1}\partial_{\alpha_2} \partial_{\alpha_3} R) |R|^{p-1} \bigr] \right| \lesssim \left(1 + N^{-2\step} \frac{\rho(E + \ii \eta_0)}{\eta_0} \right)\,  \left[ \E |R|^p +  \left(\frac{N^{3 \tar}}{N \eta_1}\right)^p \right] \,,
		\end{equation}
		where we overestimated $N^{-\step} \sqrt{\rho/\eta_0} \lesssim 1 + N^{-2\step} \rho/\eta_0$. 
		
		Next, we turn to terms of the form $(\partial_\alpha R) (\partial_\alpha^2 R) |R|^{p-2}$. Similarly to \eqref{eq:3cumk=1}, using \eqref{eq:summcum} for $k=3$, we find
		\begin{equation} \label{eq:3cumk=2}
			\begin{split}
				&N^{-7/2} \left| \sum_{j,k, \alpha_1, \alpha_2, \alpha_3}  \kappa(\alpha_1, \alpha_2, \alpha_3)G_{ja_1} G_{b_1a_2} (GB)_{b_2j} G_{k a_3}(GB)_{b_3k} \right| \\
				&\quad \lesssim  N^{-7/2}  \sum_{ \alpha_1, \alpha_2, \alpha_3} |\kappa(\alpha_1, \alpha_2, \alpha_3)| \, |(GBG)_{b_3a_3}| \, |(GBG)_{b_2a_1}|\\
				&\quad\lesssim  N^{-7/2} \sqrt{\frac{\rho(E+ \ii \eta_1)}{\eta_1}} \sum_{ \alpha_1, \alpha_2, \alpha_3} |\kappa(\alpha_1, \alpha_2, \alpha_3)| \, |(GBG)_{b_3a_3}| \, \sqrt{(GBB^*G^*)_{b_2b_2}} \\
				&\quad\lesssim  N^{-7/2} \sqrt{ \frac{\rho(E + \ii \eta_1)}{\eta_1}} \,  \vertiii{\kappa}_3 \sqrt{\sum_{b_3, a_3} |(GBG)_{b_3a_3}|^2} \sqrt{\sum_{b_2 a_2} (GBB^*G^*)_{b_2 b_2}} \\
				&\quad\lesssim  \frac{\rho(E + \ii \eta_1)^{1/2}}{N^2\eta_1^{2}} \langle \Im GB\Im G B^* \rangle^{1/2} \langle \Im GBB^* \rangle^{1/2}			\lesssim  N^{- \step}\sqrt{\frac{\rho(E + \ii \eta_0)}{\eta_0}} \left(\frac{N^{3 \tar}}{N \eta_1}\right)^2 \,, 
			\end{split}
		\end{equation}
		with very high probability. %
		To go to the third line, we used a Schwarz inequality and the estimate $(GG^*)_{a_1 a_1} \lesssim \rho/\eta_1$ w.v.h.p.~(as follows by a Ward identity and \eqref{eq:avzaginput}). In the penultimate step, we again used several Ward identities. In the last step we used $\langle \Im GB\Im G B^* \rangle \le \langle \Im GBB^* \rangle/\eta_1$ and \eqref{eq:isotrick} together with $\eta_0/\eta_1 \le N^\step$, monotonicity of $\eta \mapsto \eta \rho(E+ \ii \eta)$, and $\step \ll \tar$ by \eqref{eq:exponents}.  Hence, again by Young's inequality and overestimating $N^{-\step} \sqrt{\rho/\eta_0} \lesssim 1 + N^{-2\step} \rho/\eta_0$,   we find  %
		\begin{equation} \label{eq:3cumk=2final}
			\left| N^{-3/2} \hspace{-2mm}\sum_{ \alpha_1, \alpha_2, \alpha_3} \hspace{-2mm} \kappa(\alpha_1, \alpha_2, \alpha_3)\E \bigl[ (\partial_{\alpha_1}\partial_{\alpha_2}  R) (\partial_{\alpha_3} R)|R|^{p-2}\bigr] \right| \lesssim \left(1 + N^{-2\step} \frac{\rho(E + \ii \eta_0)}{\eta_0} \right) \left[ \E |R|^p + N^\tar \left(\frac{N^\step }{N \eta_1}\right)^p\right]\,. 
		\end{equation}
		
		Finally, we estimate terms of the form $(\partial_\alpha R)^3 |R|^{p-3}$, which are the most critical ones, since they necessarily contribute the $N^{-2 \step} \rho/\eta$ factor as we estimate $B$ by its Hilbert--Schmidt norm $\Vert B \Vert_{\rm hs}$ . %
		For terms of the form $(\partial_\alpha R)^3 |R|^{p-3}$, similarly to \eqref{eq:3cumk=1} and \eqref{eq:3cumk=2}, we find
		\begin{equation} \label{eq:3cumk=3}
			\begin{split}
				&N^{-9/2} \left| \sum_{j,k,\ell, \alpha_1, \alpha_2, \alpha_3}  \kappa(\alpha_1, \alpha_2, \alpha_3)G_{ja_1} (GB)_{b_1j} G_{ka_2} (GB)_{b_2k} G_{\ell a_3}(GB)_{b_3\ell} \right| \\
				&\quad\lesssim N^{-9/2} \frac{\rho(E+ \ii \eta_1)}{\eta_1} \Vert B \Vert  \sum_{ \alpha_1, \alpha_2, \alpha_3}  |\kappa(\alpha_1, \alpha_2, \alpha_3)| \,   \big|(GBG)_{b_2a_2} \big| \,  \big|(GBG)_{b_3a_3} \big| \\
				&\quad\lesssim N^{-9/2} \frac{\rho(E+ \ii \eta_1)}{\eta_1} \Vert B \Vert   \vertiii{\kappa}_3 \sum_{a,b}\big|(GBG)_{ab} \big|^2 
				\lesssim \, N^{-7/2} \frac{\rho(E+ \ii \eta_1)^2}{\eta_1^4} \Vert B \Vert  \Vert B \Vert_{\rm hs}^2 \\
				&\quad\lesssim N^{- 2\step}{\frac{\rho(E + \ii \eta_0)}{\eta_0}}\left(\frac{N^{3 \tar}}{N \eta_1}\right)^3 \,. 
			\end{split}
		\end{equation}
		To go to the second line, we used that 
		\begin{equation} \label{eq:entrywise}
			|(GBG)_{ab}| \le \Vert B \Vert \sqrt{(GG^*)_{aa} (GG^*)_{bb}} \lesssim \Vert B \Vert \frac{\rho(E+ \ii \eta_1)}{\eta_1}\,,
		\end{equation}
		by a Schwarz inequality, a Ward identity and \eqref{eq:avzaginput}. In the third line we estimated 
		\begin{equation} \label{eq:absumm}
			\sum_{a,b}\big|(GBG)_{ab} \big|^2 = \frac{N}{\eta_1^2} \langle \Im G B \Im G B^* \rangle \lesssim \frac{N \rho(E+ \ii \eta_1)}{\eta_1^3} \Vert B \Vert^2_{\rm hs}\,,
		\end{equation}
		with very high probability,  by means of Ward identities and \eqref{eq:isotrick}. To go to the fourth line, we used $\Vert B \Vert \le \sqrt{N} \Vert B \Vert_{\rm hs}$ and the fact that $\step \ll \tar$ by \eqref{eq:exponents}, together with $\eta_0/\eta_1 \le N^\step$ and monotonicity of $\eta \mapsto \eta \rho(E+ \ii \eta)$. 
		Hence, \eqref{eq:3cumk=3} together with Young's inequality implies that %
		\begin{equation} \label{eq:3cumk=3final}
			N^{-3/2}	\left| \sum_{ \alpha_1, \alpha_2, \alpha_3} \E \bigl[ (\partial_{\alpha_1}  R)(\partial_{\alpha_2} R) (\partial_{\alpha_3} R)|R|^{p-3}\bigr] \right| \lesssim \left(1 + N^{-2\step} \frac{\rho(E + \ii \eta_0)}{\eta_0} \right) \left[ \E |R|^p +  \left(\frac{N^{3\tar}}{N \eta_1}\right)^p\right]\,. 
		\end{equation}

		For the higher order terms in \eqref{eq:GFTcumulants} with $n = k+1 \ge 4$ %
		we aim to estimate  %
		\begin{equation*}
			\left| N^{-n/2}\sum_{\alpha_1, ..., \alpha_n} \kappa(\alpha_1, ..., \alpha_n) \E \bigl[\partial_{\alpha_1} ... \partial_{\alpha_n} |R|^p\bigr] \right| \,. 
		\end{equation*}
		In case that the $n$ derivatives are distributed on $k \in  [n]$ factors of $R$, we find that, for $n_\ell \in \N$ with $\sum_{\ell = 1}^k n_\ell = n$ and identifying $(\alpha_i)_{i \in [n]} \equiv \big((a_{\ell_i}, b_{\ell_i})\big)_{i \in [n_\ell], \ell \in [k]}$, 
		\begin{equation} \label{eq:higherorderav}
			\begin{split}
				&\left| N^{-n/2} N^{-k} \sum_{j_1, ... ,j_k} \sum_{\alpha_1, ... , \alpha_n} \kappa(\alpha_1, ... , \alpha_n) \prod_{\ell = 1}^{k} \big( G_{j_\ell a_{\ell_1}} G_{b_{\ell_1}a_{\ell_2}}  ... G_{b_{\ell_{n_\ell-1}} a_{\ell_{n_\ell}}} (GB)_{b_{\ell_{n_\ell}}j_\ell}  \big) \right|  \\
				&\quad\lesssim  N^{-n/2} N^{-k} \sum_{\alpha_1, ... , \alpha_n} |\kappa(\alpha_1, ... , \alpha_n)|  \prod_{\ell = 1}^{k} \big| (GBG)_{b_{\ell_{n_\ell}}a_{\ell_1}} \big| \\
				&\quad\lesssim  N^{-n/2} N^{-k} \left(\frac{\rho(E+ \ii \eta_1)}{\eta_1}\right)^{k-2} \,\Vert B \Vert^{k-2} \sum_{\alpha_1, ..., \alpha_n} |\kappa(\alpha_1, ... , \alpha_n)| \, \big| (GBG)_{\tilde{b}_1 \tilde{a}_1} \big| \, \big| (GBG)_{\tilde{b}_2 \tilde{a}_2} \big|%
			\end{split}
		\end{equation}
		To go to the second line, we performed all the $j$ summations and estimated all the other resolvents without a $j$ index by \eqref{eq:avzaginput}; to go to the third line, we used \eqref{eq:entrywise} for $k-2$ of the $k$ factors and used a simplified notation for the indices $\tilde{a}, \tilde{b}$, which agree with some $a_{\ell_i}, b_{\ell_j}$. The two factors of $GBG$ are kept separately, since we aim for an estimate in terms of Hilbert--Schmidt norm $\Vert B \Vert_{\rm hs}$ of the observable $B$; otherwise the whole argument for the higher order terms would be a simple power counting. 
		However, now we distinguish two cases: (i) $k \le n-2$, and (ii) $k \in \{n-1,n\}$. 
		
		In the less critical case (i),   we use a Schwarz inequality to estimate $\big| (GBG)_{\tilde{b} \tilde{a}} \big| \lesssim \sqrt{(GBB^* G^*)_{\tilde{b} \tilde{b}}} \, \sqrt{\rho/\eta_1}$, similarly to \eqref{eq:entrywise}. Then, we continue to estimate \eqref{eq:higherorderav} as
		\begin{equation} \label{eq:avcase2}
			\begin{split}
				&N^{-n/2} N^{-k} \left(\frac{\rho(E+ \ii \eta_1)}{\eta_1}\right)^{k-1} \,\Vert B \Vert^{k-2} \sum_{\alpha_1, ..., \alpha_n} |\kappa(\alpha_1, ... , \alpha_n)| \,  \sqrt{(GBB^*G^*)_{\tilde{b}_1 \tilde{b}_1}}  \, \sqrt{(GBB^*G^*)_{\tilde{b}_2 \tilde{b}_2}} \\[2mm]
				&\quad\le  N^{-n/2} N^{-k} \left(\frac{\rho(E+ \ii \eta_1)}{\eta_1}\right)^{k-1} \,\Vert B \Vert^{k-2} \vertiii{\kappa}_n \sum_{ab} (GBB^*G^*)_{a a} \\
				&\quad\lesssim N^{2-n/2}  \left(\frac{\rho(E+ \ii \eta_1)}{N\eta_1}\right)^{k} \,\Vert B \Vert^{k-2} \Vert B \Vert_{\rm hs}^2 \lesssim \left(\frac{\rho(E+ \ii \eta_1)}{N\eta_1}\right)^{k} \,. 
			\end{split}
		\end{equation}
		While in the second step, we used \eqref{eq:isotrick}, the final step follows from $\Vert B \Vert \le \sqrt{N}\Vert B \Vert_{\rm hs} \lesssim \sqrt{N}$ and $k\le n-2$. 
		
		For case (ii), we first note that necessarily $(\tilde{a}_1, \tilde{b}_1) = (a_1, b_1) = \alpha_1$, and similarly for index $2$, up to permutation of the arguments of $\kappa$ in \eqref{eq:higherorderav}. This simply follows, since $n \ge 4$ derivatives hitting each of $k \in \{n-1, n\}$ factors at least once, means that at least two of them are hit exactly once. Therefore, we can continue estimating \eqref{eq:higherorderav} as 
		\begin{equation} \label{eq:avcase1}
			\begin{split}
				&N^{-n/2} N^{-k} \left(\frac{\rho(E+ \ii \eta_1)}{\eta_1}\right)^{k-2} \,\Vert B \Vert^{k-2} \sum_{\alpha_1, ..., \alpha_n} |\kappa(\alpha_1, ... , \alpha_n)| \, \big| (GBG)_{{b}_1 {a}_1} \big| \, \big| (GBG)_{{b}_2 {a}_2} \big| \\
				&\quad\le  N^{-n/2} N^{-k} \left(\frac{\rho(E + \ii \eta_1)}{\eta_1}\right)^{k-2} \,\Vert B \Vert^{k-2} \vertiii{\kappa}_n \sum_{ab} \big|(GBG)_{ba}\big|^2 \\
				&\quad\lesssim N^{(k-n)/2}\left(\frac{1}{N \eta_1}\right)^k \frac{\rho(E+ \ii \eta_1)}{\eta_1} \lesssim N^{-2 \delta} \frac{\rho(E + \ii \eta_0)}{\eta_0} \left(\frac{N^{3 \tar}}{N \eta_1}\right)^k \,. 
			\end{split}
		\end{equation}
		Note that for $k=n$ this estimate truly contributes the critical $N^{-2 \step} \rho(E + \ii \eta_0)/\eta_0$ factor.  
		Here, in the second step, we used \eqref{eq:absumm} together with $\Vert B \Vert \le \sqrt{N}\Vert B \Vert_{\rm hs} \lesssim \sqrt{N}$; the final step follows from $\eta_0/\eta_1 \le N^\step$ together with monotonicity of $\eta \mapsto \eta \rho(E+ \ii \eta)$ and $\step \ll \tar$ by \eqref{eq:exponents}.

		Hence, by Young's inequality we deduce
		\begin{equation} \label{eq:4cumfinal}
			\left| N^{-n/2}\sum_{\alpha_1, ..., \alpha_n} \kappa(\alpha_1, ..., \alpha_n) \E (\partial_{\alpha_1} ... \partial_{\alpha_n} |R|^p) \right| \lesssim  \left(1 + N^{-2\step} \frac{\rho(E + \ii \eta_0)}{\eta_0} \right)\,  \left[ \E |R|^p + \left(\frac{N^{3\tar}}{N \eta_1}\right)^p\right]\,. 
		\end{equation}
		
		Therefore, combining \eqref{eq:GFTcumulants} with \eqref{eq:3cumk=1final}, \eqref{eq:3cumk=2final}, \eqref{eq:3cumk=3final}, and \eqref{eq:4cumfinal}, we obtain \eqref{eq:Gronwallav}. This finishes the proof of Proposition~\ref{prop:gronwallav}. \qed

		\subsubsection{Isotropic case: Proof of Proposition~\ref{prop:gronwalliso}} \label{subsec:zag_iso}
		Similarly to the proof of Proposition~\ref{prop:gronwallav}, after applying Itô's Lemma and a cumulant expansion, we find
		\begin{equation} \label{eq:GFTcumulantsiso}
			\left| \frac{\dif}{\dif s} \E |S_s|^{p}  \right| \lesssim \left|  \sum_{k=2}^{L-1} \sum_{\alpha_1} \sum_{\bm \alpha \in \mathcal{N}(\alpha_1)^k}\frac{\kappa_s(\alpha_1, \bm \alpha)}{N^{(k+1)/2} \, k!} \E \bigl[\partial_{\alpha_1} \partial_{\bm \alpha} |S_s|^{p}\bigr] \right| + \Psi(\eta_1)^{p} \,. 
		\end{equation}
		for some large enough $L$. 
		
		Employing the same notational simplifications as explained below \eqref{eq:GFTcumulants}, we again first estimate the third order cumulant terms, given by
		\begin{equation*}
			\left| N^{-3/2}\sum_{\alpha_1, \alpha_2, \alpha_3} \kappa(\alpha_1, \alpha_2, \alpha_3) \E \bigl[\partial_{\alpha_1} \partial_{\alpha_2} \partial_{\alpha_3} |S|^p \bigr] \right| \,. 
		\end{equation*}
		Distributing the derivatives according to the Leibniz rule, we need to estimate various terms of the forms $(\partial_\alpha^3 S) |S|^{p-1}$, $(\partial_\alpha S) (\partial_\alpha^2 S) |S|^{p-2}$, and $(\partial_\alpha S)^3 |S|^{p-3}$.  In contrast to the average case treated in the proof of Proposition~\ref{prop:gronwallav}, there is no term in the cumulant expansion producing the most critical $N^{-2 \step} \rho/\eta$ factor; instead we get $N^{8 \step} \sqrt{\rho/\eta} = N^{1/2 + 8 \step} \Psi$.  
		
		We start with estimating the first type of terms.   %
		In this case, identifying $\alpha_i \equiv (a_i, b_i) \in [N]^2$ and using Lemma~\ref{lem:monotone} together with a Ward identity and Assumption~\ref{ass:cumulants}~(i), we find
		\begin{equation} \label{eq:3cumk=1iso}
			\begin{split}
				&N^{-3/2} \left| \sum_{\alpha_1, \alpha_2, \alpha_3}  \kappa(\alpha_1, \alpha_2, \alpha_3)G_{\bm x a_1} G_{b_1a_2} G_{b_2a_3} G_{b_3\bm y} \right| \\
				&\quad\lesssim N^{-3/2} N^{2 \step}\sum_{ \alpha_1, \alpha_2, \alpha_3}  |\kappa(\alpha_1, \alpha_2, \alpha_3)| \,  |G_{\bm x a_1}| \,  |G_{b_3\bm y}| \\
				&\quad\lesssim N^{-3/2} N^{2\step} \, \vertiii{\kappa}_3 \, \bigg(\sum_{ a_1, b_1} |G_{\bm x a_1}|^2\bigg)^{1/2} \, \bigg(\sum_{ a_3, b_3} |G_{b_3\bm y}|^2\bigg)^{1/2} 		\lesssim \, N^{1/2+3 \step} \frac{\rho(E + \ii \eta_0)}{N\eta_1}  
			\end{split}
		\end{equation}
		with very high probability. 
		Completely analogously we obtain
		\begin{equation} \label{eq:3cumk=2iso}
			N^{-3/2} \left| \sum_{\alpha_1, \alpha_2, \alpha_3}  \kappa(\alpha_1, \alpha_2, \alpha_3)G_{\bm x a_1} G_{b_1\bm y} G_{\bm x a_2} G_{b_2a_3} G_{b_3\bm y} \right| \lesssim N^{1/2 + 4 \step} \left(\frac{\rho(E+ \ii \eta_0)}{N \eta_1}\right)^{3/2}
		\end{equation}
		and 
		\begin{equation} \label{eq:3cumk=3iso}
			N^{-3/2} \left| \sum_{\alpha_1, \alpha_2, \alpha_3}  \kappa(\alpha_1, \alpha_2, \alpha_3)G_{\bm x a_1} G_{b_1\bm y} G_{\bm x a_2} G_{b_2\bm y} G_{\bm x a_3} G_{b_3\bm y} \right| \lesssim N^{1/2 + 4 \step} \left(\frac{\rho(E + \ii \eta_0)}{N \eta_1}\right)^{2} \,, 
		\end{equation}
		again 	with very high probability. 
		Hence, combining \eqref{eq:3cumk=1iso}, \eqref{eq:3cumk=2iso}, and \eqref{eq:3cumk=3iso} with Young's inequality and additionally using that $\eta \mapsto \rho(E+ \ii \eta) /\eta$ is monotonically decreasing, we infer 
		\begin{equation*}
			\left| N^{-3/2}\sum_{\alpha_1, \alpha_2, \alpha_3} \kappa(\alpha_1, \alpha_2, \alpha_3) \E \bigl[\partial_{\alpha_1} \partial_{\alpha_2} \partial_{\alpha_3} |S|^p \bigr] \right| \lesssim N^{1/2 + 8 \step} \Psi(\eta_0) \big[ \E |S|^p  + \Psi(\eta_1)^p\big] \qquad \text{w.v.h.p.}
		\end{equation*}
		
		Next, we turn to the higher order terms, where we aim to estimate 
		\begin{equation} \label{eq:GFThigherorderiso}
			\left| N^{-n/2}\sum_{\alpha_1, ..., \alpha_n} \kappa(\alpha_1, ..., \alpha_n) \E \bigl[\partial_{\alpha_1} ... \partial_{\alpha_n} |S|^p \bigr] \right| \,. 
		\end{equation}
		Distributing the $n$ derivatives on $k \in  [n]$ factors of $S$, we find that, for $n_\ell \in \N$ with $\sum_{\ell = 1}^k n_\ell = n$ and (w.l.o.g.) $n_1 \le n_2 \le ... \le n_k$, and identifying $(\alpha_i)_{i \in [n]} \equiv \big((a_{\ell_i}, b_{\ell_i})\big)_{i \in [n_\ell], \ell \in [k]}$, \eqref{eq:GFThigherorderiso} can be rewritten as (ignoring the factor $|S|^{p-k}$)
		\begin{equation} \label{eq:GFThighbound}
			\left| N^{-n/2}  \sum_{\alpha_1, ... , \alpha_n} \kappa(\alpha_1, ... , \alpha_n) \prod_{\ell = 1}^{k} \big( G_{\bm x a_{\ell_1}} G_{b_{\ell_1}a_{\ell_2}}  ... G_{b_{\ell_{n_\ell-1}} a_{\ell_{n_\ell}}} G_{b_{\ell_{n_\ell}}\bm y}  \big) \right| \,. 
		\end{equation}
		If $n_2 = 1$, since there are now at least two factors of $S$ hit by a single derivative, we find that (similarly to \eqref{eq:avcase1} in the proof of Proposition~\ref{prop:gronwallav}, cf.~also \cite[Eqs.~(8.82)--(8.85)]{edgeiid}) 
		\begin{equation*}
			\begin{split}
				\eqref{eq:GFThighbound} &\lesssim N^{-n/2} N^{(n+k-4)\step} \vertiii{\kappa}_n \sum_{a,b} |G_{\bm x a} G_{b \bm y}|^2 \\
				& \lesssim N^{2-n/2} N^{(n+k-2)\step} \Psi(\eta_1)^4 \le \big[N^{1/2 + 8 \step} \Psi(\eta_0)\big] \Psi(\eta_1)^k,
			\end{split}
		\end{equation*}
		with very high probability. 
		If $n_2 > 1$, we find, analogously to \eqref{eq:avcase2} in the proof of Proposition~\ref{prop:gronwallav} (cf.~also \cite[Eqs.~(8.86)--(8.87)]{edgeiid})   %
		\begin{equation*}
			\begin{split}
				\eqref{eq:GFThighbound} &\lesssim N^{-n/2} N^{(n+k-2)\step} \vertiii{\kappa}_n \sum_{a,b} |G_{\bm x a}|^2 \lesssim N^{2-n/2} N^{(n+k)\step} \Psi(\eta_1)^2 \\
				&\lesssim N^{1/2} N^{(n+k)\step - (n-4)\tar/2} \Psi(\eta_0)\Psi(\eta_1)^{n-2} \lesssim \big[N^{1/2 + 8 \step} \Psi(\eta_0)\big] \Psi(\eta_1)^{k} ,
			\end{split}
		\end{equation*}
		with very high probability. 
		Here, to go to the second line, we used that $N^{-1/2 + \tar /2} \le \Psi(\eta_0) \le \Psi(\eta_1)$. In the ultimate step, we used $\Psi(\eta_1) \le 1$ and that, since $n_2 > 1$ and $n_1 \le n_2 \le ... \le n_k$, we have $n \ge k+2$. 
		Therefore, using Young's inequality, we infer %
		\begin{equation} \label{eq:4cumfinaliso} 
			\left| N^{-n/2}\sum_{\alpha_1, ..., \alpha_n} \kappa(\alpha_1, ..., \alpha_n) \E \bigl[\partial_{\alpha_1} ... \partial_{\alpha_n} |S|^p \bigr] \right| \lesssim N^{1/2 + 8 \step} \Psi(\eta_0) \big[\E |S|^p + \Psi(\eta_1)^{p}\big] ,
		\end{equation} 
		with very high probability, 
		and thus, combining \eqref{eq:3cumk=1iso}, \eqref{eq:3cumk=2iso}, and \eqref{eq:3cumk=3iso} with \eqref{eq:4cumfinaliso}, and including the $\Psi(\eta_1)^p$ term from \eqref{eq:GFTcumulantsiso}, we obtain \eqref{eq:Gronwalliso}. This finishes the proof of Proposition~\ref{prop:gronwalliso}. \qed

		\section{Local law outside the support of the scDOS} \label{sec:outside}		
		In this section, we prove  Theorem~\ref{thm:noeig}, that is, the absence of spectrum inside the gaps in the support of $\rho_T$ of size $\Delta_T \ge N^{-3/4 + 5\scl}$, where $\scl > 0$ is the exponent from \eqref{eq:abvD}.  
		Recall our choice of the terminal time $T\sim N^{-\tar/4}$ from \eqref{eq:term_time}.
		
		The characteristic flow was used to exclude outliers near a regular square-root edge for Dyson Brownian motion with general $\beta$ and potential in \cite[Section 4]{adhikari20}. In \cite[Section 8.1]{edgeiid}, the approach was used at the edge of non-Hermitian i.i.d. matrices, which corresponds to a cusp-like singularity of the hermitization. We present a modified version of the proof that allows us to avoid moment-matching arguments, used in \cite{edgeiid} to remove the order one Gaussian component.
		
		\subsection{Time-Evolution of the Gaps}
		
		First, we analyze the dynamics of the gaps in the support the scDOS corresponding to the time-dependent MDE \eqref{eq:MDEt}.
		For all $t\in [0,T]$, define the density $\rho_t : \mathbb{R} \to \mathbb{R}_+$ via the Stieltjes inversion formula, $\rho_t(x) := \pi^{-1}\lim\limits_{\eta \to +0}\langle\Im M_t(x+\ii\eta) \rangle$.  
		\begin{definition}[Endpoints of a Gap]
			For a continuous probability density function $\rho$ on $\mathbb{R}$, we say that $\mathfrak{e}^-, \mathfrak{e}^+$ 
			are left and right \textit{end-points of a gap in the support of $\rho$} if and only if $\mathfrak{e}^-_{}, \mathfrak{e}^+ \in \partial \{x \in \mathbb{R} : \rho(x) > 0\} $ and $\rho(x) = 0$ for all $x \in [\mathfrak{e}^-_{}, \mathfrak{e}^+]$.
		\end{definition}
		
		Once Theorem~\ref{th:Neta_local_laws} is established, the proof of Theorems~\ref{thm:main} and~\ref{thm:noeig} reduces to considering gaps in the support of $\rho_T$ with at least one end point satisfying $\dist(\mathfrak{e}_T, \mathcal{I}) \le c_M/4$, where $\mathfrak{e}_T \in \{\mathfrak{e}^-_T, \mathfrak{e}^+_T\}$, $\mathcal{I}$ is the set of admissible energies defined in \eqref{eq:admE}, and $c_M > 0$ is the constant from Assumption~\ref{ass:Mbdd}. We then distinguish between two relevant cases:
		\begin{enumerate}
			\item[(i)] The final gap size $\Delta_T := \mathfrak{e}^+_T - \mathfrak{e}^-_T \le  c_M/4$,
			\item[(ii)] $\Delta_T  >  c_M/4$.
		\end{enumerate}
		We focus on the more challenging case  (i), which, in particular, includes all cusp-like singularities in the set of admissible energies. In this case, by Lemma~\ref{lemma:Mt}, the solution $M_t(z)$ remains bounded in and around the gap for all times $0 \le t \le T$. 
		
		In the simpler case (ii), it is straightforward to verify that the singularity at the endpoint $\mathfrak{e}_t := \varphi_{t,T}(\mathfrak{e}_T)$ is a regular edge-point for all $0 \le t \le T$, where $\varphi_{t,T}$ is the flow map defined in \eqref{eq:phi_map}.   Consequently, there is no need to track the precise behavior of the opposite endpoint of the gap, and the analysis in Section~\ref{sec:outside} holds with $\Delta_t$ replaced by $1$. The definition of the sub-scale domain $\subD_t$ (see \eqref{eq:below_Dt} below) must be adjusted by the condition $\kapd_t(z) := \dist(\mathfrak{e}_t, z ) \le c_M/8 + C'(T-t)$, where $C' \sim 1$ is an appropriate constant (e.g., from Lemma~\ref{lemma:abvD_t}). The rest of the proof then follows verbatim. Therefore, for the remainder of this section, we assume that $\Delta_T \le c_M/4$.

		For any $t \in [0,T]$ and any $z := E+\ii\eta$ with  $E$ lying inside the gap $[\mathfrak{e}^-_{t}, \mathfrak{e}^+_{t}]$ in the support of $\rho_t$, the scDOS $\rho_t(z)$ satisfies (see Remark 7.3 in \cite{AEK2020})
		\begin{equation} \label{eq:rho_comp}
			\rho_t(z) \sim \frac{\eta}{(\kapd_t(z)+\eta)^{1/2}(\Delta_t+\kapd_t(z)+\eta)^{1/6}}, \quad 
			\kapd_t(z) := \dist(E, \mathfrak{e}^\pm_t).
		\end{equation}
		
		In the following lemma, we collect the necessary properties of the quantities $\mathfrak{e}^\pm_t$, $\Delta_t$, $\kapd_t(z_t)$ along the flow \eqref{eq:z_flow}, that we later use in the proof of Proposition~\ref{prop:zig_below}. Recall that  the terminal time is small,  $T\sim N^{-\xi/4}\ll 1$ by \eqref{eq:term_time}, and the final gap is also sufficiently small $\Delta_T\le c_M/4$.
		\begin{lemma}[Characteristic Flow near Small Gaps] \label{lemma:rho_t_props}
			For any time $0 \le t \le T$, let $\mathfrak{e}^-_{t}, \mathfrak{e}^+_{t}$ be the left and right end-points of a gap in the support of $\rho_t$ with size $0 < \Delta_t\lesssim 1$, then for any $0 \le s\le t$, there exist a gap in the support of $\rho_s$ with endpoints $ \mathfrak{e}^-_{s}, \mathfrak{e}^+_{s}$ and width $\Delta_s : = \mathfrak{e}^+_{s} -\mathfrak{e}^-_{s} $, 
			that satisfy
			\begin{equation} \label{eq:Delta_s_comp}
				\Delta_s \sim \Delta_t + (t-s)^{3/2},
			\end{equation}
			\begin{equation} \label{eq:end_s_flow}
				\rd \mathfrak{e}^\pm_{s}  = -\frac{1}{2} \mathfrak{e}^\pm_{s} \rd s - \langle M_s(\mathfrak{e}^\pm_{s}) \rangle \rd s.
			\end{equation}

			Pick an $E_t \in (\mathfrak{e}^-_{t}, \mathfrak{e}^+_{t})$ and $\eta_t \lesssim N^{-\arb}\Delta_t\,$ for some $\arb > 0$. Let $z_s   = E_s+\ii\eta_s   := \varphi_{s,t}(E_t + \ii \eta_t)$, as defined in \eqref{eq:phi_map}, then
			\begin{equation} \label{eq:eta_ll_Delta}
				\eta_s \lesssim N^{-\arb/2}\Delta_s, \quad E_s \in (\mathfrak{e}^-_{s}, \mathfrak{e}^+_s), \quad 0 \le s \le t.
			\end{equation}
			Moreover, for any $0 \le s \le t$, recall $\kapd_s(z) := \dist(\Re z, \mathfrak{e}^\pm_s)$,   
			and assume that $\kapd_t(z_t) \gtrsim N^{\arb}\eta_t\,$, then
			\begin{equation} \label{eq:eta_ll_d_preserved}
				\eta_s^{-1}\kapd_s(z_s)  \gtrsim \eta_t^{-1}\kapd_t(z_t), \quad 0 \le s \le t.
			\end{equation} 
			Finally, there exists a constant $\mathfrak{c}  > 0$, such that for any $0 \le t \le T$, if $E_t \in (\mathfrak{e}^-_t, \mathfrak{e}^+_t)$ and $\eta_t \lesssim N^{-\arb}\kapd_t$, then $z_s := \varphi_{s,t}(E_t + \ii\eta_t)$ satisfies
			\begin{equation} \label{eq:kapd_s_bound}
				\sqrt{\kapd_s(z_s)} \ge \sqrt{\kapd_t(z_t)} + \mathfrak{c}  (t-s) \Delta_s^{-1/6}.
			\end{equation} 
		\end{lemma}
		We defer the proof of Lemma~\ref{lemma:rho_t_props} to Appendix~\ref{app:tech}.
		
		\subsection{Absence of Spectrum inside Small Gaps. Proof of Theorem~\ref{thm:noeig}}
		In the sequel, we always assume that the final gap satisfies
		$\Delta_T \ge N^{-3/4 + 5\scl}$.
		Recall the constant  $\scl$ from \eqref{eq:abvD}, and define the function $f \equiv f_{\scl}$ by
		\begin{equation} \label{eq:front_func}
			f(t) \equiv f_{\scl}(t) := \biggl[  \frac{N^{-1+\scl} + \mathfrak{r} (T-t)}{2\Delta_t^{1/6} } \vee N^{\scl}\sqrt{\eta_{\mathfrak{f},t}} \biggr]^2, \quad \eta_{\mathfrak{f},t} := N^{-2/3}\Delta_t^{1/9}, \quad t\in [0,T],
		\end{equation}
		where we chose the constant  $\mathfrak{r} $ satisfying  $1 \lesssim \mathfrak{r} \le \mathfrak{c}$ (where $\mathfrak{c}$ is the constant from \eqref{eq:kapd_s_bound}) to be sufficiently small such that $f(t) \le \tfrac{1}{4}\Delta_t$. This is indeed possible, since it follows from \eqref{eq:Delta_s_comp} that $\Delta_t^{2/3} \gtrsim \Delta_T^{2/3} + (T-t)$, and $\Delta_T^{2/3} \gg N^{-1/2}$ by assumption on the final gap size.

		Fix a tolerance exponent $0 < \sscl < \tfrac{1}{100}\tar$, where $\tar$ is the exponent from \eqref{eq:term_time}, and define the time-dependent sub-scale domain $\subD_t$ (see Figure~\ref{fig:sub_fig}) by
		\begin{equation} \label{eq:below_Dt}
			\subD_t \eqcirc \subD_t(\scl, \sscl) := \bigl\{ z:=E+\ii\eta \in \mathbb{H} : \kapd_t(z) \ge f(t), \, N^{-\sscl/2} \le \rho_t(z)N\eta \le N^{\scl}  \bigr\},
		\end{equation}
		where we recall $\kapd_t(z) = \dist(\Re z, \mathfrak{e}^\pm_t)$.   In the sequel, we omit the arguments $\scl, \zeta$ of the domain $\subD_t$ from the notation.  
		
		\begin{figure}
			\centering
			\includegraphics[width=.33\textwidth]{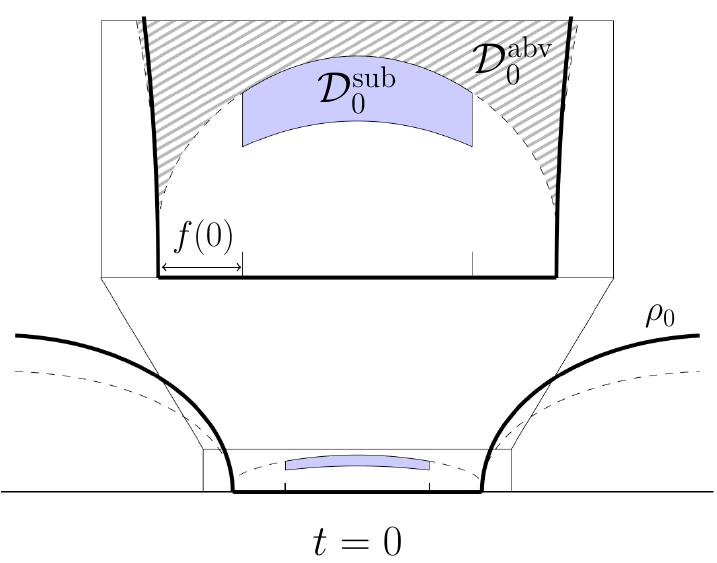}\hfill
			\includegraphics[width=.33\textwidth]{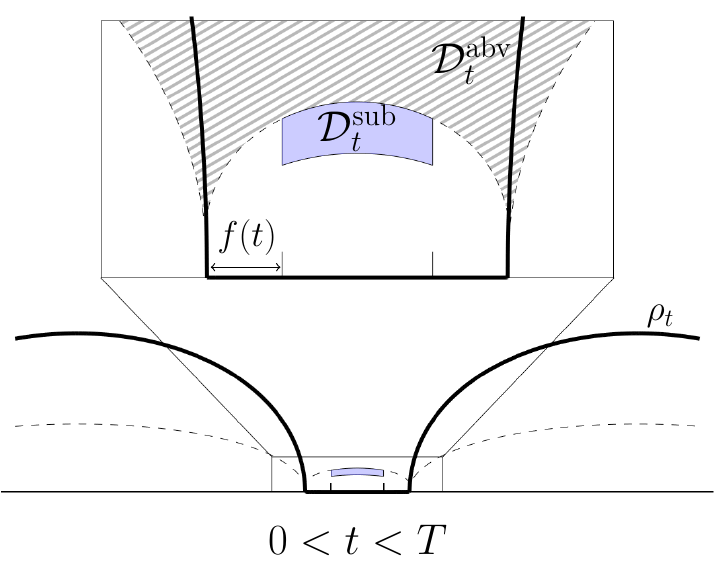}\hfill
			\includegraphics[width=.33\textwidth]{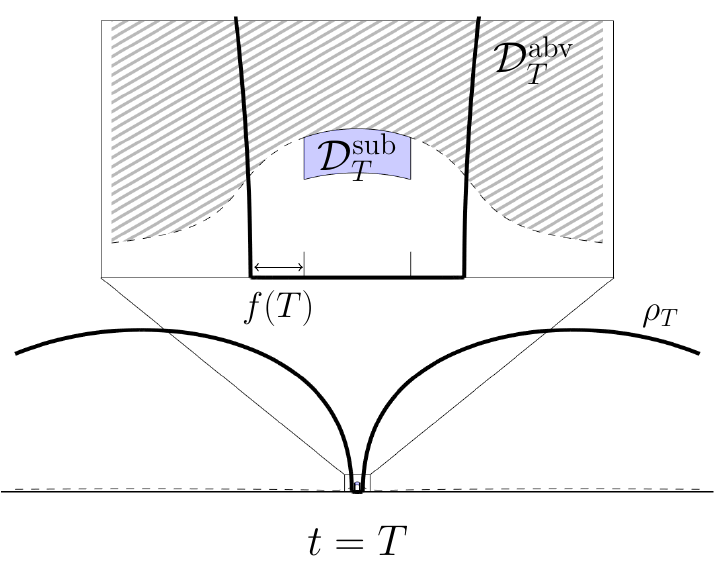}\hfill
			\caption{Shaded in blue is the illustration of the time-dependent domain $\subD_{t}$, defined in \eqref{eq:below_Dt}, at three distinct times: the initial time $t = 0$ (left), an intermediate time $0 < t < T$ (center), and the terminal time $t=T$ (right). The domain $\abvD_{t}$ at the corresponding time $t$ is indicated with crosshatching in the zoomed-in insert, with its boundary indicated by a dashed line in the main plot. The zoomed-in insert also depicts the distance $f(t)$, defined in \eqref{eq:front_func}, between the edge of the support of $\rho_t$ and the corresponding horizontal cut-off of the domain $\subD_{t}$. The graph of the scDOS $\rho_t$ is superimposed in black on each panel (not to scale). }
			\label{fig:sub_fig}
		\end{figure}
		
		\begin{definition}[Exclusion Estimate]
			Let $H_u$ be a random matrix depending on some parameter\footnote{
				As in Definition~\ref{def:ll_template}, the parameter $u$ will typically be  time and the set $ \mathcal{U}$ will be a 
				bounded subinterval of $\R$.} 
			$u \in \mathcal{U}$, and let $M_u$ be the solution to the MDE \eqref{eq:MDE} with the data pair $(\E H_u, \mathcal{S}_{u})$, where $\mathcal{S}_{u}$ is the self-energy operator corresponding to $H_u$ via \eqref{eq:self_energy_def}.
			For all $u \in \mathcal{U}$, let $\mathcal{D}_u$ be a subset of $\mathbb{C}$, and let $\sscl > 0$.
			We say that the resolvent $G_u(z) := (H_u-z)^{-1}$ satisfies the   \emph{exclusion estimate}, with data $(\mathcal{D}_u, \sscl, \Omega)$ uniformly in $u \in \mathcal{U}$, if and only if the bound
			\begin{equation} \label{eq:excl_template}
				\biggl\lvert \bigl\langle G_u(z) - M_{u}(z)\bigr\rangle  \biggr\rvert \le \frac{N^{-\sscl}}{N |\Im z|},
			\end{equation}
			holds uniformly in $z \in \mathcal{D}_u$ and in $u \in \mathcal{U}$, on the event $\Omega$.
		\end{definition}
		
		The goal of the present subsection is to deduce the following claim.
		\begin{claim} \label{claim:sub_goal}
			If a random matrix $H$  satisfies
			the assumptions of Theorem~\ref{thm:main}, then for any $0<\sscl<\tfrac{1}{100}\tar$, 
			the resolvent $G(z) := (H-z)^{-1}$ satisfies the exclusion estimate \eqref{eq:excl_template} with data $(\subD_T, 2\sscl, \Omega)$ for some very-high-probability event $\Omega$. 
		\end{claim}
		
		Then, using Claim~\ref{claim:sub_goal} as an input, we conclude \eqref{eq:noeig} using the following lemma.
		\begin{lemma}[Eigenvalue Exclusion] \label{lemma:exclusion}
			Fix a time  $t\in [0, T]$, with the terminal time $T$ as in \eqref{eq:term_time}, 
			and let $H$ be a random matrix satisfying $\E H = A_t$ and $\mathcal{S}_{H} = \mathcal{S}_t$, where $\mathcal{S}_{H}$ is the self-energy corresponding to $H$ via \eqref{eq:self_energy_def}.
			Assume that for some tolerance exponent $\arb > 0$ and $\ell\in\mathbb{N}$ with $\ell\arb \ll \sscl$, the resolvent $G(z) := (H-z)^{-1}$ satisfies the exclusion estimate \eqref{eq:excl_template} with data $(\subD_{t} ,\sscl -\ell\arb,\Omega)$, then 
			\begin{equation} \label{eq:no-spec}
				\spec(H)\, \cap \, [\mathfrak{e}^-_t + f(t), \mathfrak{e}^+_t - f(t)] = \emptyset \quad \text{on }\Omega.
		\end{equation}
	\end{lemma}
	We defer the proof of Lemma~\ref{lemma:exclusion} to Appendix~\ref{app:tech}.
	
	\begin{proof} [Proof of Theorem~\ref{thm:noeig}]
		Choose $\scl := \tfrac{1}{5}\theta_0$, $ \tar  := \tfrac{1}{10}\scl$ and $\sscl < \tfrac{1}{100}\tar$. It follows from Claim~\ref{claim:sub_goal} that $G(z):= (H-z)^{-1}$ satisfies the exclusion estimate \eqref{eq:excl_template} with data $(\subD_T, 2\sscl, \Omega)$ for some very-high-probability event $\Omega$, where $\subD_T := \subD_{T}(\scl, \tar)$ is defined in \eqref{eq:below_Dt}. Hence, \eqref{eq:noeig} follows immediately from \eqref{eq:no-spec} of Lemma~\ref{lemma:exclusion}, since $f(T) := f_\scl(T) \ge N^{2\scl}\etaf(e_0)$ by definition \eqref{eq:front_func}.  This concludes the proof of Theorem~\ref{thm:noeig}.
	\end{proof}

	To prove Claim~\ref{claim:sub_goal}, we augment the Zigzag induction of Section~\ref{sec:zigzag} with the following propositions. Recall that the relations \eqref{eq:exponents} between the fixed tolerance exponents $\sscl, \tar, \scl$ from \eqref{eq:abvD}, \eqref{eq:term_time} and \eqref{eq:below_Dt}, respectively.
	\begin{proposition} [Zig Step below the Scale] \label{prop:zig_below}
		Fix $k \in \{1, \dots, K\}$, and recall the definition of $t_k$ from \eqref{eq:t_steps}. Let $G_{t}(z)$ be the time-dependent resolvent defined in \eqref{eq:G_t}.  
		Assume that for some $\arb > 0$ and $\ell \in \mathbb{N}$ with $\ell\arb \ll \zeta$, 
		the resolvent $G_{t}$ satisfies the exclusion estimate \eqref{eq:excl_template} with data $(\subD_{t}, \sscl-\ell\arb, \Omega)$ at time $t = t_{k-1}$, for some very-high-probability event $\Omega$. Then the resolvent $G_t$ satisfies the exclusion estimate \eqref{eq:excl_template} with data  
		$(\subD_{t}, \sscl-(\ell+1)\arb, \Omega')$ uniformly in $t \in [t_{k-1}, t_k]$, for some very-high-probability event $\Omega'\subset \Omega$.
	\end{proposition}
	
	\begin{proposition} [Zag Step below the Scale] \label{prop:zag_below}
		Fix $k \in \{1, \dots, K\}$, and let $G^s(z)$ be the time-dependent resolvent defined in \eqref{eq:G^s}, and let $s_k := s(\dift_k)$ be as defined in \eqref{eq:init_cond_reverse}.
		Assume that for some $\arb > 0$ and $\ell \in \mathbb{N}$ with $\ell \arb \ll \sscl$, the resolvent $G^s(z)$ satisfies the exclusion estimate \eqref{eq:excl_template} with data $(\subD_{t_k}, \sscl-\ell\arb, \Omega)$ at time $s = s_k$, for some very-high-probability event $\Omega$, and  the isotropic local law in \eqref{eq:ll_template} with data  $(\abvD_{t_{k}}, \tar+\ell\arb)$ uniformly in time $s\in [0, s_k]$.
		Then the bound $G^s(z)$ satisfies the exclusion estimate \eqref{eq:excl_template} with data $(\subD_{t_k}, \sscl-(\ell+1)\arb, \Omega')$ uniformly in time $s\in [0, s_k]$, for some very-high-probability event $\Omega'\subset \Omega$.
	\end{proposition}

	\begin{proof} [Proof of  Claim~\ref{claim:sub_goal}]
		Claim~\ref{claim:sub_goal} follows by induction in $k$ as in Section~\ref{sec:zigzag} using the tandem of Propositions~\ref{prop:zig_below} and~\ref{prop:zag_below}, and using the global law of Proposition~\ref{prop:global} for $H_0$ as the initial estimate at step $k=0$. This is indeed sufficient, since for all $z := E+\ii\eta \in \subD_0$,   $N\eta \gtrsim N^{1/4 + \varepsilon - \zeta/4}$,   hence
		the right-hand side of \eqref{eq:global_av} satisfies
		\begin{equation}
			N^{3\tar}\Psi(z)\sqrt{\frac{\langle z \rangle}{N\eta}} \lesssim \frac{N^{3\tar}}{N\eta}\sqrt{\frac{1+\rho_0(z)N\eta}{N\eta}} \lesssim \frac{N^{-1/8+3\tar+\zeta/8}}{N\eta} \le \frac{N^{-\zeta}}{N\eta}.
		\end{equation}
		This concludes the proof of Claim~\ref{claim:sub_goal}.
	\end{proof}

	\begin{proof} [Proof of Proposition~\ref{prop:zig_below}]
		The proof is essentially analogous to that in Sections~\ref{sec:zig}, hence we only outline the key differences.
		
		It follows from \eqref{eq:z_flow} that $z_s := \varphi_{s,t}(z_t) \in \subD_s$ for all $z_t \in \subD_t$ and all $0 \le s \le t$.
		Moreover, using \eqref{eq:rho_comp}, we conclude that 
		\begin{equation}
			\frac{\Im z}{\kapd_s(z)} \lesssim \biggl(\frac{\eta_{\mathfrak{f},s}}{\kapd_s(z)} \biggr)^{3/4}\sqrt{\rho_s(z)N\Im z} \lesssim N^{-\scl}, \quad z \in \subD_s.
		\end{equation}
		Therefore,  it follows from \eqref{eq:kapd_s_bound} that for all $z_t \in \subD_t$, the trajectory $z_s := \varphi_{s,t}(z_t)$ satisfies 
		\begin{equation} \label{eq:kapd-f}
			\kapd_s(z_s) - f(s) = \biggl(\sqrt{\kapd_s(z_s)} + \sqrt{f(s)}\biggr)\biggl(\sqrt{\kapd_s(z_s)} - \sqrt{f(s)}\biggr) \gtrsim \sqrt{\kapd_s(z_s)}\, \frac{t-s}{\Delta_s^{1/6}}, \quad 0 \le s \le t,
		\end{equation}
		where, in the second step, we used \eqref{eq:kapd_s_bound} and \eqref{eq:front_func} to estimate $\sqrt{\kapd_s(z_s)} - \sqrt{f(s)}$.
		
		Let $\tin := t_{k-1}$ and $\tfin:=t_{k}$. Define the stopping time $\tau$ by
		\begin{equation} \label{eq:tau_outside}
			\tau := \inf\biggl\{ \tin < t \le \tfin : \sup_{\tin \le s \le t} \sup\limits_{z \in \subD_t} \bigl \lvert N\eta_s \langle G_s(z) - M_s(z) \rangle \rvert \le N^{
				-\zeta + (\ell+1)\arb} \biggr\}.
		\end{equation}
		Statement \eqref{eq:no-spec} of Lemma~\ref{lemma:exclusion} then implies that on the event $\Omega := \{t \le \tau\}$, the resolvent $G_t$ satisfies the norm bound
		\begin{equation}
			\norm{G_t(z)} \le \frac{\Im z}{\bigl(\kapd_t(z) - f(t)\bigr)^2 + (\Im z)^2}, \quad z \in \subD_t.
		\end{equation}
		Therefore, computing the quadratic variation of the   martingale    term in \eqref{eq:Gav_evol} with $B=\bm 1$ similarly to \eqref{eq:av_mart_compute} yields 
		\begin{equation} \label{eq:av_mart_below}
			\begin{split}
				\biggl[\int_{\tin}^\cdot \frac{1}{\sqrt{N}} \sum_{ab } \partial_{ab} \bigl\langle G_s \bigr\rangle \rd \bigl(\Brwn_{s}\bigr)_{ab} \biggr]_{t\wedge\tau} 
				&\le \int_{\tin}^{t\wedge\tau} \frac{\bigl\langle (\Im G_s)^2\bigr\rangle}{N^2\eta_s^2} \rd s \le \int_{\tin}^{t\wedge\tau} \frac{\bigl\langle \Im G_s\bigr\rangle}{N^2\eta_s^2}\frac{\eta_s}{\bigl(\kapd_s - f(s)\bigr)^2 + \eta_s^2} \rd s  \\ 
				&\lesssim  \int_{\tin}^{t\wedge\tau} \frac{1}{N^2 \kapd_s^{3/2}\Delta_s^{-1/6}\bigl( (t\wedge\tau-s)^2 + \kapd_s^{-1}\Delta_s^{1/3}\eta_s^2\bigr)} \rd s \\
				&\lesssim \frac{1}{N^2\eta_{t\wedge\tau}^2} \frac{\eta_{t\wedge\tau}}{\kapd_{t\wedge\tau}}  \lesssim \frac{N^{-\scl} }{N^2\eta_{t\wedge\tau}^2},
			\end{split}
		\end{equation}
		abbreviating $G_s := G_s(z_s)$, $\eta_s:= \Im z_s$, and $\kapd_s := \kapd_s(z_s)$. In \eqref{eq:av_mart_below}, in the second step we used \eqref{eq:rho_comp}, \eqref{eq:kapd-f} and \eqref{eq:tau_outside}, while in the last line we used the fact that $\kapd_s \gtrsim \kapd_t$, $\Delta_s \gtrsim \Delta_t$ and $\kapd_s^{1/2}\Delta_s^{-1/6} \gtrsim \kapd_t^{1/2}\Delta_t^{-1/6}$ for all $ s \le t$, that follows from \eqref{eq:rho_comp}, \eqref{eq:Delta_s_comp}, \eqref{eq:eta_ll_d_preserved} and \eqref{eq:kapd_s_bound}.

		The remainder of the proof follows analogously to Section~\ref{sec:zig}.
	\end{proof}

	\begin{proof} [Proof of Proposition~\ref{prop:zag_below}]
		
		Note that by choosing the constant $c' \sim 1 $ in \eqref{eq:abvD_t} small enough, we can guarantee that for any $t \in [0,T]$ and any $z := E + \ii\eta \in \outD_t$, the point $E + \ii \eta(E) $ lies in $\abvD_t$, where $\eta(E)$ is defined implicitly via $\eta(E) \rho_t(E+\ii\eta(E)) = N^{-1+\scl}$. Indeed, we only need to check that $\rho_t(E+\ii\eta(E))^{-1}\eta(E) \ge c'(N^{-1+\scl}+T-t)$. 
		However, it follows from \eqref{eq:rho_comp} and the definition of $f(t)$ in \eqref{eq:front_func} that $\rho_t(E+\ii\eta(E))^{-1}\eta(E) \gtrsim N^{\scl}\eta_{\mathfrak{f},t}^{1/2}\Delta_t^{1/6} + T-t$. Together with 
		$\Delta_t \gtrsim   \Delta_T \gtrsim  N^{-3/4+5\scl}$, this  immediately implies that the inclusion  $E + \ii \eta(E) \in \abvD$ for sufficiently small $c'\sim1$.

		Since throughout the proof the time $t_k$ remains fixed, for the remainder of this section, we drop the superscript $t_k$ from $\abvD_{t_{k}}, \subD_{t_{k}}, \rho_{t_{k}}, \kapd_{t_k}, \Delta_{t_k}$, and $M_{t_k}$. 

		First, using a monotonicity estimate analogous to Lemma~\ref{lem:monotone} (see \eqref{eq:monot_im_below} and \eqref{eq:monot_iso} in Remark~\ref{rem:monot_below}), we conclude from the isotropic local law in \eqref{eq:ll_template} for $G^{s}(z)$ that, uniformly in $z \in \subD$, in $a,b \in [N]$ and in $s\in [0, s_k]$, 
		\begin{equation} \label{eq:G_bounds_below}
			\bigl\lvert (\Im G^s)_{aa} \bigr\rvert \lesssim \frac{N^{\scl}}{N\eta}, \quad  \bigl\lvert (G^s - M)_{ab} \bigr\rvert \lesssim \frac{N^{\scl}}{N\eta}, \quad \bigl\lvert (G^s)_{ab} \bigr\rvert \lesssim 1, \quad \text{w.v.h.p.}
		\end{equation}
		Moreover, note that for all $z := E+\ii\eta \in \subD$, we have the estimates
		\begin{equation}
			\kapd(z) \Delta^{1/3} \gtrsim N^{-1+4\scl},\quad N\eta \sim N^{1/2}\kapd(z)^{1/4}\Delta^{1/12} \sqrt{\rho(z) N \eta} \gtrsim N^{1/4+\scl - \zeta/4}.
		\end{equation}
		
		As in Section~\ref{sec:zag}, we conduct the proof along the vertical truncations of the domain $\subD$, defined as 
		\begin{equation} 
			\subD_\gamma \equiv \subD_{t_k, \gamma} := \bigl\{ z \in \subD \eqcirc\subD_{t_k} \, : \, \Im z \ge N^{-1+ \gamma} \bigr\}, \quad 0 < \gamma \le 1.
		\end{equation}
		In particular, we assert that if for some constant $\gamma_0 > 0$, the resolvent $G^s$ satisfies the estimate 
		\begin{equation} \label{eq:TrImG_below}
			\bigl\langle \Im G^s(z) \bigr\rangle \lesssim \rho(z),
		\end{equation}
		with very   high    probability uniformly in $z \in \subD_{\gamma_0} \cup \abvD$ and in time $s \in [0, s_k]$, then the estimate \eqref{eq:excl_template} holds uniformly in $z \in \subD_{\gamma_1}$ for any fixed $\gamma_1 \le \gamma_0 - (\zeta\wedge \tfrac{1}{2}\mu)$, and uniformly in time $s \in [0, s_k]$ with very high probability. 
		
		To this end, we show that the quantity $ R_s(z) := \langle G^s(z) - M(z) \rangle$ satisfies
		\begin{equation} \label{eq:zag_below_R}
			\biggl\lvert \frac{\rd }{\rd s} \E |R_s(z)|^p \biggr\rvert \lesssim \biggl(1+\frac{N^{3\zeta}}{\sqrt{\dift_k}}\biggr)\biggl[\E |R_s(z)|^p + \biggl(\frac{N^{-\zeta}}{N|\Im z|}\biggr)^p\,\biggr], \quad z \in \subD_{\gamma_1},
		\end{equation}
		where $\dift_k := t_k - t_{k-1}$ and $t_k$ are defined in \eqref{eq:t_steps}. Note that $N^{3\zeta}\sqrt{\dift_k} \le N^{3\zeta}T^{1/2} \lesssim N^{-\zeta}$, using that $T\sim N^{-\tar/4}$ from \eqref{eq:term_time}.
		The proof of \eqref{eq:zag_below_R} is analogous to that of Proposition~\ref{prop:gronwallav}. The main difference is that for the   most critical   term \eqref{eq:k3_1chain}, we use the bound 
		\begin{equation} \label{eq:avimprove}
			\begin{split}
				N^{-5/2} &\left| \sum_{\alpha_1, \alpha_2, \alpha_3}  \kappa_s(\alpha_1, \alpha_2, \alpha_3)M_{b_1a_2} M_{b_2a_3} (G^sG^s)_{b_3a_1} \right| 
				\le N^{-1} \vertiii{\kappa}_3^{\rm av}\Vert M \Vert^2 \Vert G^sG^s \Vert_{\rm hs}\\
				&\lesssim \frac{\langle \Im G^s\rangle^{1/2}}{N\eta^{3/2}} \lesssim \frac{N^{-\zeta}}{N\eta} \frac{N^{2\zeta}}{\kapd^{1/4}\Delta^{1/12}} \lesssim \frac{N^{-\zeta}}{N\eta} \frac{N^{2\zeta}}{\sqrt{T - t_k}} \lesssim \frac{N^{-\zeta}}{N\eta} \frac{N^{\tfrac{5}{2}\zeta}}{\sqrt{\dift_k}},
			\end{split}
		\end{equation}
		where we used \eqref{eq:TrImG_below} together with the monotonicity of the map $\eta \mapsto \eta \langle\Im G^s(E+\ii\eta) \rangle$ for any fixed $E \in \mathbb{R}$ to assert that $\langle\Im G^s(z) \rangle \lesssim N^{2\zeta}\rho(z)$ with very high probability, uniformly in $z \in \subD_{\gamma_1}$.
		
		The remainder of the proof follows analogously to Section~\ref{sec:zag} using the estimates \eqref{eq:G_bounds_below} instead of the respective bounds in \eqref{eq:bootstrap_init} and \eqref{eq:avzaginput}.
	\end{proof}

	\subsection{Improved Local Laws away from the Spectrum. Proof of Theorem~\ref{thm:main}} \label{sec:main_proof}
	Let $\scl := \min\{\tfrac{1}{5}\scl_0, \tfrac{1}{2}\tar_0\}$ and $\tar := \tfrac{1}{10}\scl$.
	Let $z \in \mathbb{C}$ be a spectral parameter satisfying $ N^{\scl_0}\etaf(E) \le \dist(z, \supp\rho) \le N^D$. Without loss of generality, we assume that $\norm{\bm x} = \norm{\bm y} = \norm{B}_{\mathrm{hs}} = 1$, and that $ z: = E + \ii\eta$  with $\eta \ge 0$.
	
	First, consider the case $\dist(z, \supp\rho) \le 2\eta$, then it is straightforward to check using the universal shape of the density $\rho$ (see, e.g., Remark 7.3 in \cite{AEK2020})
	that $\rho(z)N\eta \gtrsim N^{\scl}$. Therefore, in this regime, Theorem~\ref{thm:main} follows from 
	Theorem~\ref{th:Neta_local_laws} and Proposition~\ref{prop:global}.  
	
	It remains to consider the regime $\dist(z, \supp\rho) \ge 2\eta$. Clearly, $E$ lies outside of the support of $\rho$.
	Let $\mathfrak{e}^-$ and $\mathfrak{e}^+$  be the left and right end-points of the gap that contains $E$. The assumption $\dist(z, \supp\rho) \ge 2\eta$ implies that $\kapd := \dist(E, \mathfrak{e}^\pm) \gtrsim \eta$, hence $\Delta:= \mathfrak{e}^+ - \mathfrak{e}^- \ge \kapd \gtrsim N^{\scl_0}\eta_{\mathfrak{f}}(E) = N^{-2/3+\scl_0}\Delta^{1/9}$, 
	and thus $\Delta \ge N^{-3/4+9\scl_0/8}$.
	
	Define a local domain $\outD\equiv \outD(E)$ as 
	\begin{equation}
		\outD \equiv \outD(E) := \{z' \in \mathbb{C}\, :\, |\Re z' - E| \le \frac{1}{2}\kapd, \, |\Im z'| \le \kapd\}, \quad \kapd := \dist(E, \mathfrak{e}^\pm),
	\end{equation}
	and observe that $z \in \outD$. Moreover, by Theorem~\ref{thm:noeig} with $\excl_0 := \tfrac{1}{2}\scl_0$, there exists a very-high-probability event $\Omega$, such that $\spec(H)\cap \outD = \emptyset$ on $\Omega$.
	
	Therefore, on the very-high-probability event $\Omega$, the matrix-valued map $z' \mapsto G(z')-M(z')$ is analytic in the interior of $\outD$. Using the Cauchy formula, we obtain the contour integral representation
	\begin{equation} \label{eq:contour_int_Gamma}
		G(z) - M(z) = \frac{1}{2\pi\ii}\oint_\Gamma \frac{G(z')-M(z')}{z-z'}\rd z',
	\end{equation}
	where $\Gamma \subset \outD$ is the contour tracing the boundary of a rectangle centered at $z$ with width $\tfrac{1}{4}\kapd$ and height $\tfrac{3}{4} \kapd$. Note that $|z'-z| \gtrsim \kapd$ for all $z' \in \Gamma$. Using a 
	monotonicity estimate
	analogous to Lemma~\ref{lem:monotone} (see \eqref{eq:monot_iso}, \eqref{eq:monot_hs} in Remark~\ref{rem:monot_below}), we conclude from Proposition~\ref{prop:global} and Theorem~\ref{th:Neta_local_laws} that on a very-high-probability event $\Omega' \subset \Omega$, the resolvent $G(z')$ satisfies
	\begin{equation} \label{eq:bounds_on_Gamma}
		\biggl\lvert \bigl\langle \bigl(G(z')-M(z')\bigr) B \bigr\rangle \biggr\rvert \lesssim \frac{N^{\scl}}{N|\Im z'|} \wedge \frac{1}{\kapd}, \quad \biggl\lvert \bigl( G(z')-M(z') \bigr)_{\bm x \bm y} \biggr\rvert \lesssim N^{\scl}\sqrt{\frac{\rho(z')}{N|\Im z'|}} + \frac{N^{\scl}}{N|\Im z'|}\wedge \frac{1}{\kapd},
	\end{equation}
	uniformly in $z' \in \Gamma$, where the alternative $\kapd^{-1}$ bound follows from the norm-bound on $\lVert G(z')\rVert$ and \eqref{eq:noeig}.
	
	Plugging the bounds \eqref{eq:bounds_on_Gamma} into the representation \eqref{eq:contour_int_Gamma} and using the comparison relation \eqref{eq:rho_comp}, we obtain \eqref{eq:AVLL} and \eqref{eq:ISOLL} at the point $z$.
	Here we used \eqref{eq:rho_comp} and $\kapd \ge N^{\scl_0}\etaf(E)$ to assert that
	\begin{equation}
		\sqrt{\frac{\rho(z)}{N\eta}} \sim \sqrt{\frac{1}{N\kapd^{1/2}\Delta^{1/6}}} \gtrsim \frac{1}{N\kapd}.
	\end{equation}
	This concludes the proof of Theorem~\ref{thm:main}. \qed
	
	\begin{remark} [Faraway Regime] \label{rem:far_away}
		Similarly to the away from-the-spectrum part of the proof of Theorem~\ref{thm:main} in Section~\ref{sec:main_proof} above, the global law \eqref{eq:global_av}, together with the contour integration, can be used  to obtain the faraway laws
		\begin{equation} \label{eq:farawaylaws}
			\biggl\lvert \bigl\langle \bigl(G(z) - M(z)\bigr) B \bigr\rangle \biggr\rvert \lesssim \frac{N^{\tar_0} }{N\langle z \rangle^2}\norm{B}_{\mathrm{hs}}, \quad \biggl\lvert \bigl( G(z) - M(z) \bigr)_{\bm x \bm y} \biggr\rvert \lesssim  \frac{N^{\tar_0}}{\sqrt{N}\langle z \rangle^2}\norm{\bm x} \norm{\bm y},
		\end{equation}
		in the regime $\dist(z, \supp \rho) \in [C, N^D]$ for some sufficiently large positive $C \sim 1$.
		The proof requires only the global laws of Proposition~\ref{prop:global} as an input, and is conducted without the use of the Zigzag dynamics. 
	\end{remark}

	\section{Global Laws: Proof of Proposition~\ref{prop:global}} \label{sec:stable}
	We prove Proposition~\ref{prop:global} in two steps. First, in Section ~\ref{sec:glob_iso}, we prove the isotropic local law \eqref{eq:global_iso}. Then, in Section~\ref{sec:glob_av}, we conclude the proof of Proposition~\ref{prop:global} by proving the averaged law \eqref{eq:global_av}, using the isotropic law \eqref{eq:global_iso} as an input.
	Before proceeding with the proof, we collect some preliminary bounds on the stability operator and define the appropriate norm for proving the isotropic local law.
	\subsection{Preliminaries for the Global Law}

	First, for any $z \in \mathbb{C}$, the \textit{stability operator} $\mathcal{B}(z) : \mathbb{C}^{N\times N} \to \mathbb{C}^{N\times N}$ is defined by its action on $X \in \mathbb{C}^{N\times N}$,
	\begin{equation} \label{eq:stab_def}
		\mathcal{B}(z)[X] := X - M(z)\mathcal{S}[X]M(z).
	\end{equation}
	We control the inverse of the stability operator $\mathcal{B}$ using the following lemma.
	\begin{lemma}(Proposition 4.4 in \cite{edgelocallaw}) \label{lemma:stab_op}
		Let $M(z)$ be the solution to the MDE \eqref{eq:MDE}, and let $\mathcal{I}$ be the set of admissible energies defined in \eqref{eq:admE}. Then the stability operator $\mathcal{B}(z)$, defined in \eqref{eq:stab_def} satisfies, for all $z\in \mathbb{C}$ with $\dist(\Re z, \mathcal{I}) \le \tfrac{3}{4}c_M$,
		\begin{equation} \label{eq:stability_factor}
			\norm{\mathcal{B}^{-1}(z)}_{\mathrm{hs}\to \mathrm{hs}} + \norm{\mathcal{B}^{-1}(z)}_{\norm{\cdot}\to \norm{\cdot}} \lesssim 1 + \beta(z)^{-1}, \quad \beta(z) := \rho(z)^2 + \rho(z) |\sigma(z)| + \rho(z)^{-1}|\Im z|,
		\end{equation}
		where the function\footnote{ Roughly speaking, 
			the quantity $|\sigma(z)|$ measures how close $z$ is to a possible almost cusp, in
			particular, if $x$ is an exact cusp of the density $\rho(x)$, then $\sigma(x)=0$.}  $\sigma(z)$ is defined as 
		\begin{equation} \label{eq:sigma}
			\sigma(z) := \biggl\langle \sign\bigl(\Re  U (z)\bigr) \bigl(\rho(z)^{-1}\Im U (z)\bigr)^3 \biggr\rangle, \quad \quad  U  := \frac{(\Im M)^{-1/2}(\Re M)(\Im M)^{-1/2} + \ii}{\bigl\lvert (\Im M)^{-1/2}(\Re M)(\Im M)^{-1/2} + \ii \bigr\rvert}
			,\quad z \in \mathbb{H}.
		\end{equation}
	\end{lemma}
	
	Note that by definition of $\,\globD$ in \eqref{eq:globD_def}, the stability \emph{factor satisfies} $\beta(z) \ge N^{-\tar/4}$ for all $z \in \globD$.
	\begin{remark} [Local Laws in the Stable Domain]  
		In Section~\ref{sec:stable} we only use the bound $\beta(z) \ge \rho(z)^{-1}|\Im z|$. However, by Remark 10.4 in \cite{AEK2020}, there exists a function $\other{\beta}(z)$ satisfying $\beta(z) \lesssim \other{\beta}(z) \le \beta(z)$, such that the map $\eta \mapsto \other{\beta}(E+\ii\eta)$ is non-decreasing in $\eta > 0$ for any fixed $E$. Therefore, the global domain, defined in \eqref{eq:globD_def}, can be replaced by the \emph{stable domain}, defined as 
		\begin{equation}
			\mathcal{D}^{\mathrm{stab}}  := \bigl\{ z := E+\ii\eta\in \mathbb{H}\,:\, |E| \le N^D,\, N^{-1+\scl} \le \eta \le N^D,\, \other{\beta}(z) \ge N^{-\tar/4} \bigr\},
		\end{equation} 
		with our proof of Proposition~\ref{prop:global} naturally extending to the larger \emph{stable domain}.
		In particular, the stable domain extends down to the level $\eta \ge N^{-1+\scl}$ in the bulk of spectrum, where $\rho(E) \gtrsim 1$. Therefore, we provide an independent proof of the local laws in Theorems 2.1 and 2.2 of \cite{slowcorr} under the Assumptions~\ref{ass:boundedexp}--\ref{ass:Mbdd} without the complicated graphical expansion machinery.
	\end{remark}
	
	Next, for a fixed spectral parameter $z \in \globD(\tar,D)$, and a fixed pair of vectors $\bm x$, $\bm y \in \mathbb{C}^N$, define a family of sets of vectors,
	\begin{equation}
		\begin{split}
			\mathcal{V}_{0} \equiv \mathcal{V}_{0}(z) &:= \bigl\{ \bm e_a\bigr\}_{a=1}^N \cup \{\bm x, \bm y\},\\ 
			\mathcal{V}_{j} \equiv \mathcal{V}_{j}(z) &:= \mathcal{V}_{j-1} \cup \bigl\{ M\bm u,   \kappa_{\mathrm{c}}\bigl((M\bm u)a, \cdot b\bigr), \kappa_{\mathrm{d}}\bigl((M\bm u)a, b\cdot \bigr) : \bm u \in \mathcal{V}_{j-1},\, a,b \in [N] \bigr\}, \quad j \in \{1,\dots, J\},
		\end{split}
	\end{equation}
	where $M := M(z)$, and $J$ is an integer satisfying $J  \ge 2/\tar$.
	We use the corresponding isotropic norm (Section 5.1 in \cite{slowcorr})
	\begin{equation} \label{eq:star_norm}
		\lVert X\rVert_* \equiv \lVert X\rVert_*^{\bm x, \bm y, J, z} := \sum_{j=0}^J N^{-\frac{j}{2J}} \lVert X \rVert_{(j)} + N^{-1/2} \max_{\bm v \in \mathcal{V}_J} \frac{\norm{X_{\cdot\bm v}}}{\norm{\bm v}},
		\quad \lVert X \rVert_{(j)} := \max_{\bm u, \bm v \in \mathcal{V}_j} \frac{\bigl\lvert X_{\bm u \bm v}\bigr\rvert }{\lVert \bm u \rVert\lVert \bm v\rVert}.
	\end{equation}
	Note that the cardinality of the sets $\mathcal{V}_j$ is bounded by $N^{CJ}$,   hence we can take the maximum of very-high-probability bounds over these sets.  
	
	Finally, recall that for all $z$ with $\Re z$ in the set  of admissible energies $\mathcal{I}$ from Assumption~\ref{ass:Mbdd}, $M(z)$ satisfies the bound
	\begin{equation} \label{eq:M_bound}
		\norm{M(z)} \lesssim \langle z \rangle^{-1}.
	\end{equation}
	
	\subsection{Proof of the Isotropic Bound in Proposition~\ref{prop:global}} \label{sec:glob_iso}
	\begin{proof}[Proof of the isotropic law in \eqref{eq:global_iso}]
		Recall the definition of the domain $\globD$ from \eqref{eq:globD_def}.   We conduct the proof iteratively along vertical truncations $\globD_{\gamma}$ of the domain $\globD$, defined as
		\begin{equation}
			\globD_{\gamma} := \bigl\{ z := E+\ii\eta \in\globD%
			\, :  \,\eta \ge N^{-1+\gamma} \}, \quad \gamma > 0.
		\end{equation}
		Once the local law \eqref{eq:global_iso} is established in the domain $\globD_{\gamma_0}$ for some $\gamma_0 \ge 0$, a simple monotonicity argument analogous to Lemma~\ref{lem:monotone} (see the proof of Lemma~\ref{lem:monotone} in Appendix~\ref{app:tech}) implies that the following bounds on the resolvent $G(z)$,
		\begin{equation} \label{eq:G_bounds}
			\bigl\lvert G(z)_{\bm u \bm v} \bigr\rvert \lesssim N^\step\langle z\rangle^{-1}, \quad 
			\bigl\lvert \bigl(\Im G(z)\bigr)_{\bm u \bm u}\bigr\rvert \lesssim N^{\tar+\step}\biggl(\rho(z) + \frac{1}{N\eta}
			\biggr) , \quad 
			\text{w.v.h.p.,}
		\end{equation}
		hold uniformly in $z \in \globD_{\gamma_1}$ for any $\gamma_1 \ge \gamma_0 - \step$ with $\step \le \tfrac{1}{20}\tar$, and for any deterministic $\bm u, \bm v$ with $\norm{\bm u }=\norm{\bm v } =1$. 
		Therefore, the key step in the iteration is going from estimates on the resolvent $G(z)$ to a bound on $(G(z)-M(z))_{\bm x \bm y}$, that is, using the bounds \eqref{eq:G_bounds} as an input to prove the isotropic local law \eqref{eq:global_iso} in the domain $\globD_{\gamma_1}$.
		This crucial step is based on  the following gap in the possible values of $\norm{G-M}_*$. 
		\begin{lemma}[Gap in the Values of $G-M$] \label{lemma:Gap}
			Fix a spectral parameter $z \in \globD_{\gamma_1}$, with some $\gamma_1 > 0$ such that \eqref{eq:G_bounds} holds on $\globD_{\gamma_1}$,  then
			\begin{equation} \label{eq:gap_global_iso}
				\norm{G(z) - M(z)}_* \lesssim N^{-\tar}\,\,\text{w.v.h.p.}\quad \Longrightarrow \quad	\norm{G(z) - M(z)}_* \lesssim N^{\tar}\Psi(z) \,\,\text{w.v.h.p.}
			\end{equation}
		\end{lemma}
		
		We initialize the iteration in the domain $\globD_{2+\step}$. Indeed, owing to the very high probability bound $|H_{\bm u \bm v}| \lesssim N^{1/2+\arb}$ for any $\arb > 0$, we have, for any deterministic $\bm u, \bm v$ with $\norm{\bm u} = \norm{\bm v} = 1$,
		\begin{equation}
			\norm{G(z)} \lesssim \langle z \rangle^{-1}, \quad \bigl\lvert \bigl(\Im G(z)\bigr)_{\bm u \bm u}\bigr\rvert \lesssim \frac{\eta}{\langle z \rangle^2} \sim \rho(z), \quad z \in \globD_{2+\step}, \quad \text{w.v.h.p.}
		\end{equation}
		Note that the bound $\norm{G(z) - M(z)}_* \le N^{-\tar} $ holds trivially for all $z$ with $\Im z \ge N^\tar$.
		After Lemma~\ref{lemma:Gap} is established, the proof of \eqref{eq:global_iso} follows the standard continuity argument on a fine grid (see Section 5.4 in \cite{slowcorr}).	This concludes the proof of the isotropic law in \eqref{eq:global_iso}.
	\end{proof}
	
	The remainder of this subsection is devoted to the proof of  Lemma~\ref{lemma:Gap}. A local law for random matrices with slow correlation decay away from the cusps was already proved in \cite{slowcorr} and \cite{edgelocallaw}. We present an independent proof under the Assumptions~\ref{ass:boundedexp}--\ref{ass:Mbdd}. We utilize the \textit{minimalistic cumulant expansion}, that was used previously in \cite{lee18} and \cite{multiG}. This allows us to avoid the complicated graphical expansions.  
	\begin{proof}[Proof of Lemma~\ref{lemma:Gap}]
		Since $z := E+\ii \eta$ is fixed, we omit the argument of $G, M, \Psi, %
		\rho, \beta$, and $\mathcal{B}$.
		Assume the very-high-probability bound
		\begin{equation} \label{eq:glob_assume}
			\norm{G - M}_* \lesssim N^{-\tar}.
		\end{equation}
		It suffices to show that $\norm{G - M}_* \le N^{\tar} \Psi$   with very high probability.
		Assume that for a deterministic control parameter $\psi$, the quantity $\Psi^{-1}\norm{G - M}_*$ satisfies
		\begin{equation} \label{eq:Psi_assume}
			\Psi^{-1}\norm{G - M}_* \lesssim \psi, \quad\text{w.v.h.p}.
		\end{equation}
		
		By definition of the resolvent $G := (H-z)^{-1}$ and the MDE \eqref{eq:MDE}, we difference $G-M$ satisfies
		\begin{equation} \label{eq:G-M_equation1}
			G-M = -M\underline{WG} + M\mathcal{S}[G-M]G,
		\end{equation}
		where the matrix\footnote{  The underline $\underline{WG}$ is a renormalization of $WG$; for renormalization of general products $f(W) W g(W)$, see Section 4 in \cite{ETHpaper}.} $\underline{WG}$ is defined as
		\begin{equation}
			\underline{WG} := WG + \mathcal{S}[G]G.
		\end{equation}
		Therefore, subtracting $M\mathcal{S}[G-M]M$ from both sides and the inverse of the stability operator $\mathcal{B}$, defined in \eqref{eq:stab_def}, yields the equation  
		\begin{equation} \label{eq:1G_expand}
			G-M = - \mathcal{B}^{-1}\bigl[M \underline{WG}\bigr] + \mathcal{B}^{-1}\bigl[M \mathcal{S}[G-M] (G-M) \bigr],
		\end{equation}
		
		Observe, that for any $X\in \mathbb{C}^{N\times N}$, (Eq. (5.4c) in \cite{slowcorr})
		\begin{equation} \label{eq:stab*_est}
			\begin{split}
				\norm{\mathcal{B}^{-1}[X]}_{(j)} &\le \norm{X}_{(j)}  + \biggl(\norm{M}^2 \vertiii{\mathcal{S}} + \norm{M}^4 \vertiii{\mathcal{S}}^2 \norm{\mathcal{B}^{-1}}_{\mathrm{hs}\to\mathrm{hs}}\biggr)\norm{X}_{\max}\\
				&\lesssim \norm{X}_{(j)}  +   \bigl(1+\beta^{-1}\bigr)   \norm{X}_{(0)},
			\end{split}
		\end{equation}
		where in the last step we used \eqref{eq:stability_factor}. Here we denote
		\begin{equation}
			\vertiii{\mathcal{S}}  :=  \norm{\mathcal{S}}_{\max \to \lVert\cdot\rVert} \vee \norm{\mathcal{S}}_{\mathrm{hs} \to \lVert\cdot\rVert}.
		\end{equation}
		To control the norm $\lVert G-M \rVert_*$, we first bound the $\norm{\cdot}_{(j)}$ individually, and then estimate the contribution coming from the last summand in \eqref{eq:star_norm} later. 
		Fix an index $j \in \{0,\dots, J\}$ and fix a pair of vectors $\bm u, \bm v \in \mathcal{V}_{j}$. We compute the $p$-th (for even $p$) moment of
		\begin{equation} \label{eq:globS_def}
			S_j \equiv S_j^{\bm u \bm v} := N^{\tfrac{-j}{2J}}(G-M)_{\bm u\bm v},
		\end{equation} 
		using the equation \eqref{eq:1G_expand} for a single factor,
		\begin{equation} \label{eq:glob_iso_moment}
			\E\bigl[ |S_j|^{p}\bigr] 
			\le \E\bigl[ N^{\tfrac{-j}{2J}}\bigl(\mathcal{B}^{-1}[M\underline{WG}]\bigr)_{\bm u \bm v} \overline{S_j}|S_j|^{p-2}\,\bigr]
			+ \E\bigl[ N^{\tfrac{-j}{2J}}\bigl(\mathcal{B}^{-1}[M \mathcal{S}[G-M] (G-M)]\bigr)_{\bm u \bm v} \overline{S_j}|S_j|^{p-2}\,\bigr].
		\end{equation}
		
		First, we estimate the size of the second term on the right-hand side of \eqref{eq:glob_iso_moment}. We observe that ( Eq. (5.5a), (5.5b) in \cite{slowcorr})
		\begin{equation} \label{eq:MSXX_est}
			\begin{split}
				\norm{M\mathcal{S}[X]X}_{(j)} \lesssim \vertiii{\kappa}_2^{\mathrm{iso}}\norm{M}
				\min\biggl\{&\norm{X}_{(j+1)},\sqrt{N}\norm{X}_{(0)} \biggr\}\norm{X}_{*}.
			\end{split}
		\end{equation}
		We only use the second mode of the $\min$ bound when $j=J$.
		Combining \eqref{eq:glob_assume}, \eqref{eq:stab*_est} and \eqref{eq:MSXX_est}, we deduce that $Q_j := N^{\tfrac{-j}{2J}}\bigl(\mathcal{B}^{-1}[M \mathcal{S}[G-M] (G-M)]\bigr)_{\bm u \bm v}$ satisfies
		\begin{equation} \label{eq:Q_j_norm}
			\begin{split}
				\norm{Q_j}_{(j)} 
				\lesssim&~ \frac{\norm{G-M}_*}{\langle z \rangle N^{\tfrac{j}{2J}}}\biggl( \norm{G-M}_{(j+1)} \mathbf{1}_{j < J}+  \sqrt{N}\norm{G-M}_{(0)} \mathbf{1}_{j = J} +   \bigl(1+ \beta^{-1}\bigr)   \norm{G-M}_{(1)}\biggr)\\
				\lesssim&~ N^{\tfrac{1}{2J}-\tar}  \langle z \rangle^{-1}  \bigl(1+{\beta}^{-1}\bigr) \psi \Psi, \quad\text{w.v.h.p.}, 
			\end{split}
		\end{equation}
		where in the last step we used the estimate \eqref{eq:stability_factor}, the definition of $\,\globD$ in \eqref{eq:globD_def},  assumptions (\ref{eq:glob_assume}--\ref{eq:Psi_assume}),  and the bound $ \norm{X}_{(j)} \le N^{\tfrac{j}{2J}}\norm{X}_*$ that follows from the definition of $\norm{\cdot}_*$ in \eqref{eq:star_norm}.
		
		Next, we estimate the first term in \eqref{eq:glob_iso_moment}. For any $j \in \{0, \dots, J\}$ and any $\bm u, \bm v \in \mathcal{V}_j$, using the multivariate cumulant expansion formula from Proposition~\ref{prop:cumex}, we obtain 
		\begin{equation} \label{eq:glob_iso_cumulant}
			\begin{split}
				\biggl\lvert\,\E\bigl[(M\underline{WG})_{\bm u\bm v}\overline{S_j}|S_j|^{p-2}\,\bigr]\,\biggl\lvert \le&~ \biggl\lvert\,\E\biggl[\frac{1}{N}\sum_{a b} \sum_{\alpha_1} M_{\bm u a}G_{b \bm v}\kappa(ab, \alpha_1)\partial_{\alpha_1} \bigl\{\overline{S_j}|S_j|^{p-2}\bigr\}\,\biggr]\,\biggr\rvert\\
				&+  \sum_{k=2}^{L-1}\biggl\lvert\E\biggl[ \sum_{a b} \sum_{\bm \alpha \in \mathcal{N}(ab)^k} M_{\bm u a}\frac{\kappa(ab, \bm \alpha)}{N^{(k+1)/2} k!}\partial_{\bm \alpha} \bigl\{G_{b \bm v} \overline{S_j}|S_j|^{p-2}\bigr\}\,\biggr]\biggr\rvert\\
				&+ N^{\tfrac{j}{2J}}\bigl\lvert\Omega_{j,L}^{\bm u \bm v}\bigr\rvert.
			\end{split}
		\end{equation}
		Similarly to \eqref{eq:Omegaest}, we can choose $L$ large enough such that $|\Omega_{j,L}| \lesssim \bigl(\Psi \norm{\bm u}\norm{\bm v} \bigr)^{p}$. We note that the $N^{\tfrac{-j}{2J}}$ factors in \eqref{eq:globS_def} are only relevant for the quadratic term $Q_j$ estimated above, therefore, we do not follow it in the sequel. Moreover, we drop the norms $\lVert \bm u \rVert$ and $\lVert \bm v\rVert$ for brevity.
		
		First, we estimate the term involving second-order cumulants on the right-hand side of \eqref{eq:glob_iso_cumulant}. Here we estimate the contribution coming from the cross part of the second cumulants $\kappa_{\mathrm{c}}$, the estimate for the direct part $\kappa_\mathrm{d}$ is completely analogous. Ignoring the difference between $S_j$ and $\overline{S_j}$, and dropping the overall $|S_j|^{p-2}$ factor, we obtain the bound 
		\begin{equation} \label{eq:glob_iso_k2}
			\begin{split}
				\biggl\lvert \frac{1}{N}\sum_{a b} \sum_{\alpha_1} \kappa_{\mathrm{c}}(ab, \alpha_1) M_{\bm u a}G_{b \bm v}\partial_{\alpha_1}S_j \biggr\rvert 
				\lesssim&~  \frac{1}{N}\sum_{b b_1}  \biggl\lvert \sum_{a_1}\kappa_{\mathrm{c}}\bigl( (M \bm u)\,b, a_1b_1\bigr) G_{\bm u a_1}\biggr\rvert 	\bigl\lvert G_{b \bm v} G_{b_1 \bm v}\bigr\rvert  \\
				\lesssim &~ N^{-1+\step}{\langle z \rangle}^{-1} \bigl\lVert \norm{\kappa_{\mathrm{c}} ( (M \bm u)\ast, \cdot \ast )} \bigr\rVert \norm{G_{\cdot \bm v}}^2 \\ %
				\lesssim &~  \vertiii{\kappa}_{2}^\mathrm{iso}N^{\tar+2\step}\Psi^2%
				, \quad\text{w.v.h.p.}
			\end{split}
		\end{equation}
		In the ultimate step, we used \eqref{eq:M_bound} and \eqref{eq:G_bounds} to assert that, with very high probability,
		\begin{equation} \label{eq:G_psi_bound}
			\frac{1}{\langle z \rangle \sqrt{N}}\norm{G_{\cdot\bm v}} = \sqrt{\frac{(\Im G)_{\bm v \bm v}}{\langle z \rangle^2 N\eta}} \lesssim N^{\tfrac{\tar+\step}{2}}\Psi. %
		\end{equation}
		
		Next, we bound term involving third and higher order cumulants   in \eqref{eq:glob_iso_cumulant}.  
		Consider, for example,
		\begin{equation} \label{eq:global_iso_k3_example}
			\begin{split}
				\biggl\lvert \sum_{a b} \sum_{\alpha_1,\alpha_2} &\frac{\kappa(ab, \alpha_1, \alpha_2)}{N^{3/2}} M_{\bm u a} G_{b \bm v} (\partial_{\alpha_1}S_j) (\partial_{\alpha_2}S_j) \biggr\rvert \\
				\lesssim&~  N^{-3/2}\biggl\lvert \sum_{a b} \sum_{a_1 b_1 a_2 b_2}  \kappa(ab, a_1b_1, a_2b_2) M_{\bm u a} G_{b \bm v} G_{\bm u  a_1} G_{b_1\bm v}G_{\bm u  a_2} G_{b_2\bm v}    \biggr\rvert\\
				\lesssim &~  N^{3\tar/2 + 7\step/2} \Psi^{3}  \vertiii{\kappa}_3%
				, \quad\text{w.v.h.p}. 
			\end{split}
		\end{equation}
		Note that the structure of the term \eqref{eq:global_iso_k3_example} is identical to that of \eqref{eq:3cumk=3iso}. Indeed, the only difference is that the resolvent $G_{\bm{x}a}$ is replaced by the deterministic approximation $M_{\bm u a}$ ($\bm u$ and $\bm v$ in \eqref{eq:global_iso_k3_example} play the role of $\bm x$ and $\bm y$ in \eqref{eq:3cumk=3iso}). Consequently, the summation over $a$ is bounded using
		\begin{equation} \label{eq:rhoeta_savings}
			\biggl(\sum_a |M_{\bm u a}|^2\biggr)^{1/2} \le \norm{M}%
			\lesssim \frac{1}{\langle z \rangle}%
			\quad \text{instead of} \quad \biggl(\sum_a |G_{\bm u a}|^2\biggr)^{1/2} \lesssim N^{\frac{\tar+\step}{2}}\sqrt{\frac{\rho+
					\frac{1}{N\eta}
				}{\eta}}%
			,
		\end{equation}
		yielding a saving %
		of a  $\sqrt{\rho/\eta}$ factor
		in terms of the $(\rho/\eta)$-power on the right-hand side of \eqref{eq:global_iso_k3_example} compared to the bound in \eqref{eq:3cumk=3iso}. All other terms in \eqref{eq:glob_iso_cumulant} with cumulant of order three and higher   are bounded analogously to their counterparts in the proof of Proposition~\ref{prop:gronwalliso}, with the additional saving of $\sqrt{\rho/\eta}$ coming from \eqref{eq:rhoeta_savings}.
		
		Therefore, using a weighted Young inequality to handle the separated $|S_j|^{p-k}$ terms, 	we deduce that for all $j \in \{0,\dots, J\}$,
		\begin{equation} \label{eq:iso_underline_uv}
			\E\bigl[ N^{\tfrac{-j}{2J}}\bigl(\mathcal{B}^{-1}[M\underline{WG}]\bigr)_{\bm u \bm v} \overline{S_j}|S_j|^{p-2}\,\bigr] \le \bigl(N^{\tar/2 + 4\step }  (1+{\beta}^{-1})  \Psi%
			\bigr)^{p} + N^{-p\step}\E\bigl[ |S_j|^{p}\bigr].
		\end{equation}
		It follows from \eqref{eq:glob_iso_moment}, \eqref{eq:Q_j_norm},
		\eqref{eq:glob_iso_cumulant},  and \eqref{eq:iso_underline_uv} that
		\begin{equation}
			\E\bigl[ |S_j|^{p}\bigr] \lesssim \bigl(\Psi %
			\bigr)^p
			\biggl(N^{\tar/2+4\step} (1+{\beta}^{-1})  + N^{-\step}\psi +  N^{\tfrac{1}{2J}-\tar+\step}\langle z\rangle^{-1}   \bigl(1+{\beta}^{-1}\bigr) \psi\biggr)^{p}.
		\end{equation}
		Since $J \ge 2/\tar$,  and $\step \le \tar/20$,	 we have $ (1+{\beta}^{-1})  N^{\tfrac{1}{2J}-\tar + \step} \le N^{-\step}$, and we conclude that
		\begin{equation} \label{eq:S_j_prec}
			|S_j| \lesssim N^\arb\Psi%
			\bigl(N^{3\tar/4 + 4 \step } (1+{\beta}^{-1})  + N^{-\step}\psi \bigr) , \quad\text{w.v.h.p.}
		\end{equation}
		
		Next, we estimate the contribution of the last summand in \eqref{eq:star_norm} to $\lVert G-M\rVert_*$. We fix a vector $\bm v \in \mathcal{V}_J$ and compute a the $p$-th (for even $p$) moment of 
		\begin{equation}
			S \equiv S^{\bm v} :=  N^{-1}\norm{(G-M)_{\cdot\bm v}}^2 = N^{-1}\bigl((G-M)^*(G-M)\bigr)_{\bm v\bm v}.
		\end{equation} 
		Using the equation \eqref{eq:G-M_equation1} for a single $S$ factor, we obtain
		\begin{equation}
			\begin{split} \label{eq:glob_iso_moment_v}
				\E\bigl[ |S|^p \bigr] \le&~ N^{-1}\bigl\lvert\,\E\bigl[ \bigl((G-M)^* M \underline{WG}\bigr)_{\bm v\bm v} \overline{S}|S|^{p-2} \bigr]\,\bigr\rvert \\
				&+ N^{-1}\bigl\lvert\,\E\bigl[ \bigl((G-M)^*  M \mathcal{S}[G-M] G\bigr)_{\bm v\bm v} \overline{S}|S|^{p-2} \bigr]\,\bigr\rvert.
			\end{split}
		\end{equation}
		To estimate the term in the second line of \eqref{eq:glob_iso_moment_v}, we note the following bound,
		\begin{equation} \label{eq:quadS_est}
			\bigl\lvert \bigl(X^*M\mathcal{S}[X]Y\bigr)_{\bm v \bm v} \bigr\rvert \le \norm{X_{\cdot\bm v}}\norm{M}\norm{\mathcal{S}}_{\max \to \norm{\cdot}} \norm{X}_{(0)}\norm{Y_{\cdot\bm v}}.
		\end{equation}
		Therefore, using \eqref{eq:star_norm}, \eqref{eq:glob_assume}, \eqref{eq:Psi_assume}, \eqref{eq:G_psi_bound},  and \eqref{eq:quadS_est}, we obtain the very-high-probability bound
		\begin{equation} \label{eq:v_quad_term_bound}
			\begin{split}
				\frac{1}{N}\bigl\lvert \bigl((G-M)^* M\mathcal{S}[G-M]G\bigr)_{\bm v \bm v}\bigr\rvert &\lesssim \frac{1%
				}{\langle z \rangle\sqrt{N}} \norm{G-M}_{*}^2\norm{G_{\cdot\bm v}} \lesssim N^{-\tar+\step}\Psi^2\psi.%
			\end{split}
		\end{equation}
		Next, we turn to estimating the first term on the right-hand side of \eqref{eq:glob_iso_moment_v} using the multivariate cumulant expansion formula,
		\begin{equation} \label{eq:v_cumulant_expand}
			\begin{split}
				\frac{1}{N}\bigl\lvert\,\E\bigl[ \bigl((G-M)^* M \underline{WG}\bigr)_{\bm v\bm v}& \overline{S}|S|^{p-2} \bigr]\,\bigr\rvert 
				\lesssim \frac{1}{N^2}\biggl\lvert \sum_{a b c}\sum_{\alpha_1}  \kappa(ab,\alpha_1)G_{b\bm v} M_{ca} \partial_{\alpha_1}\bigl\{(G-M)^*_{\bm v c}\overline{S}|S|^{p-2}\bigr\} \biggr\rvert\\
				&+ \frac{1}{N}\sum_{k=2}^L \biggl\lvert \sum_{a b c}\sum_{\bm\alpha \in \mathcal{N}(ab)^k}\frac{\kappa(ab,\bm\alpha)}{N^{(k+1)/2}k!} M_{ca} \partial_{\bm\alpha}\bigl\{G_{b\bm v}(G-M)^*_{\bm v c}\overline{S}|S|^{p-2}\bigr\} \biggr\rvert\\
				&+\Omega_{L}^{\bm{v}},
			\end{split}
		\end{equation}
		where for sufficiently large integer $L$, the error term $\Omega_{L}^{\bm{v}}$ admits the bound $\Omega_{L}^{\bm{v}} \lesssim \Psi^{2p}$%
		, and is therefore negligible.
		
		We bound the term involving the second cumulants in \eqref{eq:v_cumulant_expand}. First, for the term containing $\partial_{\alpha_1} (G-M)^*_{\bm v c}$, completely analogously to \eqref{eq:glob_iso_k2}, we obtain
		\begin{equation}
			\frac{1}{N^2}\sum_{c}\biggl\lvert \sum_{b a_1 b_1} \kappa\bigl((M\bm{e}_c) b,a_1b_1\bigr)G_{b\bm v} G^*_{\bm v a_1}G^*_{b_1 c}\biggr\rvert \lesssim \norm{\kappa}_2^\mathrm{iso} N^{\tar+2\step}\Psi^2%
			, \quad\text{w.v.h.p.},
		\end{equation}
		where the additional summation over the index $c$ is compensated by the $N^{-1}$ prefactor. Next, we estimate the terms arising from $\partial_{\alpha_1}S$. We focus on the term containing $((G-M)^*\partial_{\alpha_1}G)_{\bm v\bm v}$, other terms are estimated similarly. For the cross part	 $\kappa_\mathrm{c}$, we obtain (ignoring the factor $|S|^{p-2}$ temporarily)
		\begin{equation}
			\begin{split}
				\frac{1}{N^3} \biggl\lvert &\sum_{c b} \sum_{\alpha_1} \kappa_{\mathrm{c}}\bigl((M \bm{e}_c) b,\alpha_1\bigr)G_{b\bm v} (G-M)^*_{ \bm v c} \bigl((G-M)^*\partial_{\alpha_1}G\bigr)_{\bm v \bm v} \biggr\rvert\\
				&\lesssim \frac{1}{N^2} \sum_{c d} \bigl\lvert (G-M)^*_{ \bm v c } (G-M)^*_{\bm v d} \bigr\rvert \frac{1}{N}\sum_{b  b_1} \biggl\lvert\sum_{a_1}\kappa_{\mathrm{c}}\bigl((M \bm{e}_c) b,a_1b_1\bigr)G_{d a_1}\biggr\rvert \bigl\lvert G_{b\bm v}G_{b_1\bm v} \bigr\rvert \\
				&\lesssim \vertiii{\kappa}_{2}^\mathrm{iso} N^{\tar+2\step}\Psi^2 %
				\frac{1}{N^2} \sum_{c d}\bigl\lvert (G-M)^*_{ \bm v c } (G-M)^*_{\bm v d} \bigr\rvert\\
				&\lesssim \vertiii{\kappa}_{2}^\mathrm{iso} N^{\tar+2\step} \Psi^2   \norm{G-M}_*^2 %
				\lesssim \norm{\kappa}_2^\mathrm{iso}  N^{\tar+2\step} \Psi^4 %
				\psi^2, \quad\text{w.v.h.p.},
			\end{split}
		\end{equation}
		where in the second step we used the bound analogous to \eqref{eq:glob_iso_k2} for each $c,d$, and in the last step we used \eqref{eq:Psi_assume}.
		
		Similar estimates hold for terms involving higher order cumulants in \eqref{eq:v_cumulant_expand}.  
		For example, identifying $\alpha_i := (a_i, b_i)$, 
		\begin{equation}
			\begin{split}
				N^{-7/2}\biggl\lvert \sum_{abc}&\sum_{\alpha_1,\alpha_2}\kappa(ab,\alpha_1,\alpha_2) M_{ca}(G-M)^*_{\bm v c} \bigl(\partial_{\alpha_1}G\bigr)_{b\bm v} \bigl((G-M)^*\partial_{\alpha_2}G\bigr)_{\bm v\bm v}  \biggr\rvert \\
				&\lesssim N^{-7/2}\sum_{cd}\bigl\lvert (G-M)^*_{\bm v d} (G-M)^*_{\bm v c}\bigr\rvert \biggl\lvert \sum_{a b}\sum_{\alpha_1, \alpha_2}\kappa(ab,\alpha_1,\alpha_2) M_{ca} G_{b a_1} G_{b_1\bm v} G_{d a_2} G_{b_2\bm v}  \biggr\rvert\\
				&\lesssim \vertiii{\kappa}_3\norm{G-M}_*^2  N^{\tar+3\step} \Psi^2%
				\lesssim   \vertiii{\kappa}_3 N^{\tar+3\step} \langle z \rangle^{-1}\Psi^4%
				\psi^2, \quad\text{w.v.h.p}.
			\end{split}
		\end{equation}
		Therefore, we obtain, using the very-high-probability bound $S \lesssim \psi\Psi$ by \eqref{eq:Psi_assume}, 
		\begin{equation}
			\frac{1}{N}\bigl\lvert\,\E\bigl[ \bigl((G-M)^* M \underline{WG}\bigr)_{\bm v\bm v} \overline{S}|S|^{p-2} \bigr]\,\bigr\rvert \lesssim \bigl(\Psi%
			\bigr)^{2p} \biggl(N^{\tar}N^{8\step} + N^{-\step}\psi^2 \biggr)^p,
		\end{equation}
		hence, using \eqref{eq:glob_iso_moment_v} and \eqref{eq:v_quad_term_bound}, we deduce that with very high probability,
		\begin{equation} \label{eq:S_prec}
			\sqrt{S} \lesssim N^\arb\Psi%
			\bigl(N^{\tar/2+4\step} + N^{-\step/2}\psi\bigr).
		\end{equation}
		It follows from \eqref{eq:star_norm}, \eqref{eq:iso_psi_def},  \eqref{eq:S_j_prec} and \eqref{eq:S_prec}, that 
		\begin{equation}
			\Psi^{-1}\norm{G - M}_* \lesssim \psi \,\,\text{w.v.h.p.} \Longrightarrow \Psi \lesssim N^{\tar/2+4\step+\arb} (1+{\beta}^{-1})  + N^{-\step/2+\arb}\psi \,\,\text{w.v.h.p}.
		\end{equation}
		By iteration, this implies that  $\Psi^{-1}\norm{G - M}_* \lesssim N^{\tar/2+4\step+\arb} (1+{\beta}^{-1})  \lesssim N^{3\tar/4 + 4\step +\arb}$ with very high probability, since $\beta \ge N^{-\tar/4}$ in $\globD$. This concludes the proof of Lemma~\ref{lemma:Gap}.	
	\end{proof}
	
	\subsection{Proof of the Averaged Bound in Proposition~\ref{prop:global}} \label{sec:glob_av}
	We conclude this section by proving the averaged law in Proposition~\ref{prop:global} using the isotropic law
	\eqref{eq:global_iso}, proved in Section~\ref{sec:glob_iso} above, as an input.

	\begin{proof}[Proof of the averaged law in \eqref{eq:global_av}]
		Fix a deterministic matrix $B$ and a spectral parameter $z \in \globD$, and let $R := \langle (G-M)B \rangle$. Using the equation \eqref{eq:1G_expand}, we compute the $p$-th (for even $p$) moment of $R$,
		\begin{equation} \label{eq:global_av_moment}
			\E\bigl[|R|^{p}\bigr] 
			\le \bigl\lvert\E\bigl[\bigl\langle M\mathcal{S}[G-M](G-M) \other{B} \bigr\rangle\overline{R}|R|^{p-2}\bigr]\bigr\rvert 
			+ \bigl\lvert\E\bigl[\bigl\langle  M\underline{WG}\other{B}\bigr\rangle\overline{R}|R|^{p-2}\bigr]\bigr\rvert,
		\end{equation}
		where we denote $\other{B} := \bigl((\mathcal{B}^{-1})^*[B^*]\bigr)^*$. By \eqref{eq:stability_factor} and Lemma~\ref{lemma:stab_op}, the observable $\other{B}$ satisfies
		\begin{equation} \label{eq:obs_hs_norm}
			\bigl\lVert\other{B}\bigr\rVert_{\mathrm{hs}} \lesssim  \bigl(1+{\beta}^{-1}\bigr) \norm{B}_{\mathrm{hs}}.
		\end{equation}

		To bound the first term on the right-hand side of \eqref{eq:global_av_moment}, we employ the polar decomposition 
		$\other{B} = \sum_j \sigma_j \bm v_j \bm u_j^*$, where $\sigma_j := \sigma_j(\other{B})$ and $\bm u_j := \bm u_j(\other{B}), \bm v_j := \bm v_j(\other{B})$ are the singular values and corresponding left and right, respectively, singular vectors of $\other{B}$. It follows from  \eqref{eq:global_iso}, \eqref{eq:MSXX_est}, and \eqref{eq:obs_hs_norm}, that with very high probability,
		\begin{equation} \label{eq:global_av_quad_term}
			\bigl\lvert \bigl\langle M\mathcal{S}[G-M](G-M) \other{B} \bigr\rangle \bigr\rvert \le \frac{1}{N}\sum_{j} |\sigma_j| \bigl\langle \bigl(M\mathcal{S}[G-M](G-M)\bigr)_{\bm u_j \bm v_j} \bigr\rangle \bigr\rvert \lesssim N^{2\tar} \bigl(1+{\beta}^{-1}\bigr) \Psi^2\norm{B}_{\mathrm{hs}},
		\end{equation}
		where $\Psi %
		:=\Psi(z)$ is
		defined in \eqref{eq:iso_psi_def}.

		Next, we bound the second term on the right-hand side of \eqref{eq:global_av_moment} using the multivariate cumulant expansion formula from Proposition~\ref{prop:cumex}, 
		\begin{equation} \label{eq:global_av_cumul_expand}
			\begin{split}
				\biggl\lvert\E\bigl[\bigl\langle  M\underline{WG}\other{B}\bigr\rangle\overline{R}|R|^{p-2}\bigr]\biggr\rvert 
				\le &~ \biggl\lvert\E\biggl[ \frac{1}{N^2}\sum_{ab}\sum_{\alpha_1} \kappa(ab, \alpha_1) \bigl(G\other{B}M\bigr)_{ba} \partial_{\alpha_1}\bigl\{\overline{R}|R|^{p-2}\bigr\}\biggr]\biggr\rvert\\
				&+ \sum_{k = 2}^L\biggl\lvert\E\biggl[ \frac{1}{N}\sum_{ab}\sum_{\bm\alpha \in \mathcal{N}(ab)^k} \frac{\kappa(ab, \bm\alpha)}{N^{(k+1)/2}k!}  \partial_{\bm\alpha}\bigl\{\bigl(G\other{B}M\bigr)_{ba}\overline{R}|R|^{p-2}\bigr\}\biggr]\biggr\rvert\\
				&+\bigl\lvert\Omega^B_{L}\bigr\rvert.
			\end{split}
		\end{equation}
		Here, once again $\Omega^B_{L}$ is an error term satisfying $|\Omega^B_{L}| \lesssim (\sqrt{\langle z \rangle/(N\eta)}\Psi\norm{B}_{\mathrm{hs}})^p$ for large enough $L$, controlled similarly to \eqref{eq:Omegaest}.
		The terms involving second order cumulants admit the bound    (ignoring the common $|R|^{p-2}$ factor)  
		\begin{equation} \label{1}
			\begin{split}
				\biggl\lvert \frac{1}{N^2}\sum_{ab}\sum_{\alpha_1} \kappa(ab, \alpha_1) \bigl(G\other{B}M\bigr)_{ba} \partial_{\alpha_1}R \biggr\rvert &\le \biggl\lvert \frac{1}{N^3}\sum_{ab}\sum_{a_1b_1} \kappa(ab, a_1b_1) \bigl(G\other{B}M\bigr)_{ba} \bigl(G B G\bigr)_{b_1a_1} \biggr\rvert \\ 
				&\le \frac{1}{\langle z \rangle N^2\eta^2} \bigl\lVert|\kappa(*,*)|\bigr\rVert \bigl\langle \other{B}\other{B}^*\Im G\bigr\rangle^{1/2} \bigl\langle BB^*\Im G\bigr\rangle^{1/2}\\ %
				&\lesssim N^\tar \bigl(1+{\beta}^{-1} \bigr) \vertiii{\kappa}_2 \frac{\langle z \rangle}{N\eta}\Psi^2\norm{B}_{\mathrm{hs}}^2, \quad\text{w.v.h.p.}, %
			\end{split}
		\end{equation}
		where in the second step we used the norm bound \eqref{eq:M_bound}. Here, in the last step, we used the established isotropic law \eqref{eq:global_iso}, the spectral decomposition of $\other{B}\other{B}^*$ and \eqref{eq:obs_hs_norm} to assert that, with very high probability,
		\begin{equation} \label{eq:TrBBG_HS_bound}
			\frac{\bigl\langle \other{B}\other{B}^*\Im G\bigr\rangle}{N\eta} = \frac{1}{N^2\eta}\sum_{j} |\sigma_j|^2 \bigl(\Im G\bigr)_{\bm u_j\bm u_j} \lesssim \frac{N^\tar}{N^2\eta}\sum_j |\sigma_j|^2 \biggl(\rho+\sqrt{\frac{\rho}{N\eta}} +
			\frac{1}{N\eta}\biggr)  
			\lesssim N^\tar   \bigl(1+{\beta}^{-1}\bigr)^2  \langle z \rangle^2 \Psi^2  \lVert B\rVert_{\mathrm{hs}}^2,
		\end{equation}
		where $\sigma_j$ and $\bm u_j$ are the singular values and left singular vectors of $\other{B}$. Similar bound without the factor $  (1+{\beta}^{-1})^2 $ holds for $B$ instead of $\other{B}$.
		Note that, unlike for the isotropic law \eqref{eq:global_iso}, for the current proof of the average law there is no need to split the second order cumulant into direct and cross terms,
		the simpler bound \eqref{eq:summcum} suffices.

		Next, we estimate the terms in \eqref{eq:global_av_cumul_expand} involving third order cumulants.
		Consider the term containing a single $(\partial R)$. Dropping $|R|^{p-2}$, we obtain
		\begin{equation} \label{2}
			\begin{split}
				\biggl\lvert N^{-5/2}\sum_{ab}\sum_{\alpha_1, \alpha_2}  &\kappa(ab, \alpha_1, \alpha_2) \bigl(\partial_{\alpha_1}G\other{B}M\bigr)_{ba} (\partial_{\alpha_2}R) \biggr\rvert\\
				&\lesssim N^{-7/2} \norm{M} \max_{cd}|G_{cd}|\sum_{ab}\sum_{\alpha_1\alpha_2}  \bigl\lvert \kappa(ab, \alpha_1,\alpha_2) \bigr\rvert   \bigl\lvert\bigl(G\other{B}\other{B}^*G^*\bigr)_{b_1b_1}\bigr\rvert^{1/2} \bigl(GBG\bigr)_{\alpha_2} \\
				&\lesssim \langle z \rangle^{-2} \biggl\lVert \sum_{ab}|\kappa(ab, *,*)|\biggr\rVert  \sqrt{\frac{\langle \other{B}\other{B}^*\Im G \rangle}{N\eta} } \sqrt{\frac{\langle  BB^* \Im G  \rangle }{N^3\eta^3}}  %
				\lesssim %
				N^\tar  \bigl(1+{\beta}^{-1}\bigr)  \frac{\vertiii{\kappa}_3}{N\eta}\Psi^2\norm{B}_{\mathrm{hs}}^2 %
				,
			\end{split}
		\end{equation}	
		with very high probability, where we used \eqref{eq:global_iso}, \eqref{eq:TrBBG_HS_bound}, and the bounds
		\begin{equation}
			\bigl\lvert (G\other{B}M)_{ab}\bigr\rvert \le \norm{M}\bigl\lvert\bigl(G\other{B}\other{B}^*G^*\bigr)_{aa}\bigr\rvert^{1/2}, \quad 
			\frac{1}{N}\sum_{ab} |(GBG)_{ab}|^2 %
			\le \frac{1}{\eta^3}\langle BB^* \Im G \rangle.  %
		\end{equation}
		The term containing $(\partial^2R )$ admits a completely analogous estimate.  
		
		For the term containing $(\partial R)^2$, we obtain, dropping $|R|^{p-3}$,
		\begin{equation} \label{3}
			\begin{split}
				\biggl\lvert N^{-5/2}\sum_{ab}\sum_{\alpha_1, \alpha_2}  &\kappa(ab, \alpha_1, \alpha_2) \bigl(G\other{B}M\bigr)_{ba} (\partial_{\alpha_1}R)(\partial_{\alpha_2}R) \biggr\rvert\\
				&\lesssim N^{-9/2} \max_{\alpha}\bigl\lvert (GBG)_{\alpha}\bigr\rvert \sum_{ab,\alpha_2}\sum_{\alpha_1}  \bigl\lvert \kappa(ab, \alpha_1,\alpha_2)\bigr\rvert \bigl\lvert\bigl(G\other{B}M\bigr)_{ba}  \bigl(GBG\bigr)_{\alpha_2}\bigr\rvert \\ 
				&\lesssim N^{\tar}\langle z \rangle^2 \Psi^2\norm{B}_{\mathrm{hs}}  \biggl\lVert \sum_{\alpha_1} |\kappa(*, \alpha_1,*)|\biggr\rVert \sqrt{\frac{\langle \other{B}MM^*\other{B}^* \Im G\rangle }{N\eta}} \sqrt{\frac{\langle B B^* \Im G \rangle}{N^3\eta^3}} \\   %
				&\lesssim %
				N^{2\tar}  \bigl(1+{\beta}^{-1}\bigr) \frac{\vertiii{\kappa}_3}{N\eta }    \langle z \rangle^3 \Psi^4 \lVert B\rVert_{\mathrm{hs}}^3, \quad \text{w.v.h.p.}, %
			\end{split}
		\end{equation}
		where we used the local law \eqref{eq:global_iso} to assert that, with very high probability,
		\begin{equation}
			\frac{1}{N^{3/2} }\bigl\lvert (GBG)_{ab}\bigr\rvert \lesssim \frac{\norm{B}}{\sqrt{N} } \frac{\sqrt{\bigl(\Im G\bigr)_{aa}\bigl(\Im G\bigr)_{bb}}}{N\eta}\lesssim N^{\tar}\langle z \rangle^2 \Psi^2\norm{B}_{\mathrm{hs}}.
		\end{equation}
		Note that in estimating $\max_\alpha\lvert (GBG)_{\alpha}\rvert$,
		we need to use the operator norm $\norm{B}$ since no summation on indices is available.
		We convert it into $\norm{B}_{\mathrm{hs}}$ at a costs of an extra $\sqrt{N}$ factor, as $\norm{B} \le \sqrt{N}\norm{B}_{\mathrm{hs}}$, but this is affordable since we collected sufficiently many powers of $N^{-1/2}$ in the third cumulant term.
		
		Finally, we estimate the term with no $(\partial R)$, namely, dropping $|R|^{p-1}$
		\begin{equation}
			\biggl\lvert N^{-5/2}\sum_{ab}\sum_{\alpha_1, \alpha_2}  \kappa(ab, \alpha_1, \alpha_2) G_{ba_1} G_{b_1a_2}\bigl(G\other{B}M\bigr)_{b_2a} \biggr\rvert.
		\end{equation}
		For both $G_{ba_1}$ and $G_{b_1a_2}$, we write $G_{ab} = M_{ab} + (G-M)_{ab}$ and use the bound $|M_{ab}|\lesssim \langle z \rangle^{-1}$, $| (G-M)_{ab}| \lesssim N^{\tar}\Psi$, w.v.h.p., that follow  from \eqref{eq:M_bound} and \eqref{eq:global_iso}, respectively,  to estimate the contributions coming from the deterministic and the fluctuating part separately. In particular, we obtain the very-high-probability bound,
		\begin{equation} \label{4}
			\begin{split}
				\biggl\lvert N^{-5/2}\sum_{ab}\sum_{\alpha_1, \alpha_2}  &\kappa(ab, \alpha_1, \alpha_2) (G-M)_{ba_1} M_{b_1a_2}\bigl(G\other{B}M\bigr)_{b_2a} \biggr\rvert\\
				&\lesssim N^{-1/2+\tar} \Psi \langle z \rangle^{-2}  \biggl\lVert\sum_{\alpha_1}  |\kappa(*, \alpha_1, *)|\biggr\rVert  \bigl\langle G\other{B}\other{B}^*G^* \bigr\rangle^{1/2}\\
				&\lesssim \langle z \rangle^{-1} N^{2\tar} \bigl(1+{\beta}^{-1}\bigr) \vertiii{\kappa}_3 \Psi^2\norm{B}_{\mathrm{hs}}.
			\end{split}
		\end{equation}
		The contributions coming from $M_{ba_1} (G-M)_{b_1a_2}$ and $(G-M)_{ba_1} (G-M)_{b_1a_2}$ admit analogous estimates. Therefore, it remains to bound the contribution coming from $M_{ba_1} M_{b_1a_2}$. Using \eqref{eq:kappa_3_av_norm}, we estimate
		\begin{equation} \label{5}
			\begin{split}
				\biggl\lvert N^{-5/2}\sum_{ab}\sum_{\alpha_1, \alpha_2}  \kappa(ab, \alpha_1, \alpha_2) M_{ba_1} M_{b_1a_2}\bigl(G\other{B}M\bigr)_{b_2a} \biggr\rvert &\le N^{-1}\vertiii{\kappa}_3^\mathrm{av}\norm{M}^2 \bigl\lVert G\other{B}M \bigr\rVert_{\mathrm{hs}}\\
				&\lesssim N^{\tar/2} \bigl(1+{\beta}^{-1}\bigr)  \langle z \rangle^{-2} N^{-1/2}\Psi \norm{B}_{\mathrm{hs}}, \quad\text{w.v.h.p.}
			\end{split}
		\end{equation}
		Putting back the dropped $|R|$ factors into the estimates \eqref{1}, \eqref{2}, \eqref{3}, \eqref{4} and \eqref{5}
		and using the Young’s inequality to separate these factors into an additive $|R|^p$ term
		with a small multiplicative constant, we see 
		that the second and third order cumulant terms in \eqref{eq:global_av_cumul_expand} can be estimated
		by $\bigl(N^{3\tar/2}  \langle z \rangle^{1/2}  (N\eta)^{-1/2}\Psi \norm{B}_{\mathrm{hs}}\bigr)^p + N^{-p\tar/4} |R|^{p}$. Here we used $\beta \ge N^{-\tar/4}$ from \eqref{eq:globD_def}.
		
		Estimating the terms involving forth and higher order cumulants using simple power counting, similarly to \eqref{eq:avcase2}--\eqref{eq:avcase1}, we deduce that
		\begin{equation}
			\E\bigl[ |R|^{p} \bigr] \lesssim \biggl(N^{3\tar/2}  \langle z \rangle^{1/2}  (N\eta)^{-1/2}\Psi \norm{B}_{\mathrm{hs}}\biggr)^p + N^{-p\tar/4}\E\bigl[ |R|^{p} \bigr].
		\end{equation}
		This concludes the proof of \eqref{eq:global_av}.
	\end{proof}

	\appendix

	\section{Technical Lemmas} \label{app:tech}
	In this appendix, we collect the proofs of several technical lemmas used throughout this paper. 
	\begin{proof} [Proof of Lemma~\ref{lemma:OUsurj}]
		Observe that for any $s\ge 0$ and any initial condition $H$, the distribution of the random matrix $\mathfrak{F}_{\mathrm{zag}}^s [H]$ satisfies
		\begin{equation}
			\mathfrak{F}_{\mathrm{zag}}^s\bigl[ H  \bigr] \overset{d}{=} \E [H]  + \ee^{-s/2}\bigl( H  - \E  H \bigr) + \sqrt{1 - \ee^{-s}}\, \Sigma_H^{1/2}\bigl[W_{\mathrm{G}}\bigr],
		\end{equation}
		where $W_{\mathrm{G}}$ is a standard GUE/GOE random matrix (in the same symmetry class as $H$) independent of $H$. Moreover, if $\Sigma_H \ge c\Sigma_{\mathrm{G}}$ for some constant $0 < c < 1$, then there exists a random matrix $\widehat{W}$ with $\E \widehat{W} = 0$, such that
		\begin{equation} \label{eq:Gauss_component}
			\Sigma_H^{1/2}\bigl[W_{\mathrm{G}}\bigr] \overset{d}{=} \widehat{W} + \sqrt{c}\,\other{W}_{\mathrm{G}},
		\end{equation}
		where $\other{W}_{\mathrm{G}}$ is a GUE/GOE matrix independent of $\widehat{W}$. Therefore, 
		\begin{equation} \label{eq:zag_explicit}
			\mathfrak{F}_{\mathrm{zag}}^s\bigl[H \bigr] \overset{d}{=} \widehat{H}^s + \sqrt{c}\sqrt{1 - \ee^{-s}}\, \other{W}_{\mathrm{G}}, \quad 
			\widehat{H}^s \overset{d}{:=} \E [H] + \ee^{-s/2}\bigl(H - \E H\bigr) + \sqrt{1 - \ee^{-s}} \, \widehat{W},
		\end{equation}
		where $\other{W}_{\mathrm{G}}$ is independent of $\widehat{H}^s$. 
		Hence, \eqref{eq:init_cond_for_zig} follows immediately from \eqref{eq:OU_flow} and \eqref{eq:zag_explicit} for $\mathfrak{H}_{c,t}(H) $  defined as
		\begin{equation} 
			\mathfrak{H}_{c,t}(H) := \ee^{t/2}\biggl(\E H + \ee^{-s(t)/2}\bigl(H - \E H\bigr) + \sqrt{1 - \ee^{-s(t)}} \widehat{W} \biggr),
		\end{equation}
		where the random matrix $\widehat{W}$ independent of $H$ satisfies \eqref{eq:Gauss_component}, and $s(t) \equiv s_c(t)$ is defined in \eqref{eq:init_cond_reverse}. 
		
		The estimate \eqref{eq:s(t)_est} is a direct consequence of \eqref{eq:init_cond_reverse}. This concludes the proof of Lemma~\ref{lemma:OUsurj}.
	\end{proof}
	
	\begin{proof} [Proof of Lemma~\ref{lemma:abvD_t}]
		Fix a time $0 \le t \le T$ and let $z_s$ denote the solution of \eqref{eq:z_flow} that satisfies $z_t \in \abvD_{t}$.
		It follows from \eqref{eq:z_flow} and \eqref{eq:m_flow} that for all $s \in [0,t]$,
		\begin{equation} \label{eq:drhoeta}
			\rd \bigl(\eta_s\rho_s(z_s)\bigr) =   - \pi \rho_s(z_s)^2   \rd s \le 0, 
		\end{equation}
		where we denote $\eta_s := \Im z_s$. A similar computation reveals that
		\begin{equation} \label{eq:drho^-1eta}
			\rd \bigl(\rho_s(z_s)^{-1}\eta_s\bigr) =  -\bigr(\rho_s(z_s)^{-1} \eta_s + \pi\bigl) \rd s \le -\pi \rd s,
		\end{equation}
		since $\rho_s^{-1}(z)\Im z \ge 0 $ for all $z \in \mathbb{C}$.
		Moreover, it follows from \eqref{ass:Mbdd} that $|\rd z_s / \rd s| \lesssim C'$ for all $0 \le s \le T$ and all $z_t \in \abvD_t$, hence, using the estimates \eqref{eq:drhoeta} and \eqref{eq:drho^-1eta}, we deduce that $z_s \in \abvD_{s}$ for all $s \in [0,t]$. This concludes the proof of Lemma~\ref{lemma:abvD_t}.
	\end{proof}
	
	\begin{proof} [Proof of Lemma~\ref{lemma:Mt}]
		Clearly, for terminal times $0 \le T\lesssim 1$, the solutions to \eqref{eq:ASevol} satisfy $\norm{A_t - A_T} \lesssim T-t$ and $\norm{\mathcal{S}_t - \mathcal{S}_T}_{\norm{\cdot} \to \norm{\cdot}} \lesssim T-t$, for all $0 \le t \le T$. Therefore, for some sufficiently small threshold $T_*\sim 1$, the first bound in \eqref{eq:imM} follows immediately from  Assumption~\ref{ass:Mbdd} and the stability of the MDE against small perturbations of the data pair, see Section 10 in \cite{AEK2020}. 
		Moreover, it follows from the fullness Assumption~\ref{ass:full}, that, by possibly shrinking the threshold $T_*$, we can guarantee that $\mathcal{S}_t[X] \sim \langle X \rangle$ for any Hermitian matrix $X\ge 0$.
		Hence, the second bound in \eqref{eq:imM} follows from Proposition 3.5 in \cite{AEK2020} and the first bound in \eqref{eq:imM}. This concludes the proof of Lemma~\ref{lemma:Mt}.
	\end{proof}
	
	\begin{proof} [Proof of Lemma~\ref{lem:monotone}]
		Throughout the proof, we consider the time $s\in [0,s_{\mathrm{final}}]$ to be fixed, and drop it from the superscript of $G^s$. The uniformity of all estimates in $s$ follows trivially from the assumptions of Lemma~\ref{lem:monotone}.
		
		First, we prove the second estimate in \eqref{eq:monotone}. 
		The map $\eta \mapsto \eta^2/(x^2 + \eta^2)$ is increasing in $\eta > 0$ for any $x\in\mathbb{R}$, hence it follows by spectral decomposition of $\Im G$ that
		\begin{equation} \label{eq:ImG_monot}
			\eta_1 \Im G(E+\ii\eta_1) \le \eta_0 \Im G(E+\ii\eta_0),
		\end{equation}
		in the sense of quadratic forms. Therefore, the second estimate in \eqref{eq:monotone} follows immediately from \eqref{eq:bootstrap_init}. 
		
		Next, we prove the first estimate in \eqref{eq:monotone}. Using the Schwarz inequality and the Ward identity,  we deduce that for all $0 < \eta < \eta_0$,
		\begin{equation} \label{eq:Gdeta}
			\biggl\lvert\frac{\rd  }{\rd \eta}\bigl(G(E+\ii\eta)\bigr)_{\bm u \bm v}\biggr\rvert \lesssim \frac{\bigl\lvert \bigl(\Im G(E+\ii\eta)\bigr)_{\bm u \bm u} \bigl(\Im G(E+\ii\eta)\bigr)_{\bm v \bm v}\bigr\rvert^{1/2}%
			}{\eta} \lesssim \frac{\eta_0}{\eta^2}\rho(E+\ii\eta_0),
		\end{equation}
		where in the second step we used the monotonicity of the maps $\eta \mapsto \eta \Im G (E+\ii\eta)$ and $\eta \mapsto \eta \rho(E+\ii\eta)$, and the second bound in \eqref{eq:monotone} established above. Integrating the bound \eqref{eq:Gdeta} from $\eta_1$ to $\eta_0$, we obtain 
		\begin{equation} \label{eq:G_monot}
			\bigl\lvert \bigl(G(E+\ii\eta_1)\bigr)_{\bm u \bm v} \bigr\rvert \lesssim \bigl\lvert \bigl(G(E+\ii\eta_0)\bigr)_{\bm u \bm v} \bigr\rvert + \frac{\eta_0}{\eta_1}\rho(E+\ii\eta_0).
		\end{equation}
		Since $\rho(E+\ii\eta_0) \lesssim 1$, the first estimate in \eqref{eq:monotone} follows immediately from \eqref{eq:bootstrap_init} and \eqref{eq:G_monot}. This concludes the proof of Lemma~\ref{lem:monotone}.
	\end{proof}
	
	\begin{remark}[Local Laws below the Scale] \label{rem:monot_below}
		Assume that $\rho(E+\ii\eta_1)N\eta_1 \le N^{\scl}$ and $\rho(E+\ii\eta_0)N\eta_0 = N^{\scl}$,
		in particular $\eta_1\le \eta_0$. Using \eqref{eq:ImG_monot} with \eqref{eq:ll_template} at $z:=E+\ii\eta_0$ as an input, we obtain the very-high-probability bound
		\begin{equation} \label{eq:monot_im_below}
			\bigl(\Im G(E+\ii\eta_1)\bigr)_{\bm u \bm u} \lesssim \frac{\eta_0}{\eta_1}\rho(E+\ii\eta_0) \lesssim \frac{\rho(E+\ii\eta_0) N \eta_0}{N\eta_1} \lesssim \frac{N^\scl}{N\eta_1}.
		\end{equation}
		Using Lemma~\ref{lemma:stab_op} and the identity $\rd M(z)/\rd z = \mathcal{B}^{-1}(z)[M(z)^2]$, that follows by taking the $z$-derivative of \eqref{eq:MDE}, we conclude that $\norm{\rd M(z)/\rd z} \lesssim |\rho(z)/\Im z|$. Hence, differentiating $(G(E+\ii\eta)- M(E+\ii\eta))_{\bm u \bm v}$ with respect to $\eta$, similarly to \eqref{eq:Gdeta}, we can deduce that 
		\begin{equation} \label{eq:monot_iso}
			\bigl\lvert\bigl(G(E+\ii\eta_1)- M(E+\ii\eta_1)\bigr)_{\bm u \bm v}\bigr\rvert \lesssim \frac{N^{\scl}}{N\eta_1}, \quad \text{w.v.h.p}.
		\end{equation} 
		Analogous reasoning also applies to averaged bounds. Indeed,
		\begin{equation} \label{eq:av_deriv}
			\biggl\lvert \frac{\rd}{\rd \eta} \bigl\langle\bigl(G(E+\ii\eta)- M(E+\ii\eta)\bigr)B\bigr\rangle \biggr\rvert \lesssim \frac{\bigl\lvert \bigl\langle\Im G(E+\ii\eta)\bigr\rangle \bigl\langle \Im G(E+\ii\eta)BB^*\bigr\rangle\bigr\rvert^{1/2} + \rho(E+\ii\eta)\norm{B}_{\mathrm{hs}}}{\eta}.
		\end{equation}
		Therefore, by integrating \eqref{eq:av_deriv} in $\eta$ and using \eqref{eq:ll_template}, we can deduce that 
		\begin{equation} \label{eq:monot_hs}
			\bigl\lvert \bigl\langle \bigl(G(E+\ii\eta_1)- M(E+\ii\eta_1)\bigr)B \bigr\rangle \bigr\rvert \lesssim \frac{N^{\scl}}{N\eta_1}\norm{B}_{\mathrm{hs}}, \quad \text{w.v.h.p}.
		\end{equation}
		These results show that the local laws \eqref{eq:ll_template}  hold at $z= E+\ii\eta_1$,  for any $0 < \eta_1 \le \eta_0$, once they hold at $E+\ii\eta_0$ with $\eta_0$ satisfying $\rho(E+\ii\eta_0) N \eta_0= N^{\scl}$.
	\end{remark}
	
	\begin{proof} [Proof of Lemma~\ref{lemma:rho_t_props}]
		First, we prove \eqref{eq:Delta_s_comp}.
		Let $\sigma_t$ be function defined in \eqref{eq:sigma}, corresponding to the solution $M_t$ of the time-dependent MDE \eqref{eq:MDEt}.
		It follows from Lemma 5.5 in \cite{AEK2020} that $\sigma_t$ admits a uniformly $1/3$-H\"{o}lder regular extension $\overline{\mathbb{H}}$. Moreover, it follows from Lemma 7.16 in \cite{AEK2020} that $|\sigma_t(\mathfrak{e}^-_{t})| \sim |\sigma_t(\mathfrak{e}^+_{t})| \sim \Delta_t^{1/3}$ and it follows from Theorem 7.7 (ii.b) in \cite{AEK2020} that	$\sigma_t(\mathfrak{e}^-_{t}) <0$ and $\sigma_t(\mathfrak{e}^+_{t}) > 0$. Therefore, there exists a point $x_t \in (\sigma_t(\mathfrak{e}^-_{t}), \sigma_t(\mathfrak{e}^+_{t}))$ satisfying $\sigma_t(x_t) = 0$. For any $0 \le s \le t$, let $x_s := \varphi_{s,t}(x_t)$ as defined in \eqref{eq:phi_map}. It follows from \eqref{eq:m_flow} that $\sigma_s(x_s) = 0$ for any $s\in [0,t]$.
		
		Furthermore, $1/3$-H\"{o}lder regularity of $\rho_t$ in $t$ implies that there exists $c \sim 1$ such that for times $s$ satisfying $0 \le t-s \le c\Delta_t^{1/3}$, the density $\rho_s$ has a gap in the support around $x_s$ of size $\Delta_s > 0$, let $\mathfrak{e}^-_{s}$ and $\mathfrak{e}^+_{s}$ denote its endpoints. From $1/3$-H\"{o}lder regularity of $\sigma_s$, we infer that $\dist(x_s, \mathfrak{e}^{\pm}_{s}) \sim \Delta_s$. 
		
		On the other hand, the map $h_s : x \mapsto \lim_{\eta\to+0}\rho_s(x+\ii\eta)^{-1}\eta$ is also $1/3$-H\"{o}lder regular, uniformly in $s$, hence $h_s(x_s) \sim \dist(x_s, \mathfrak{e}^{\pm}_{s})^{1/2}\Delta_s^{1/6} \sim \Delta_s^{2/3}$ by \eqref{eq:rho_comp}. Along the trajectories of \eqref{eq:z_flow}, $h_s(x_s) \sim h_t(x_t) + (t-s)$ for all $s$ satisfying $0 \le t-s \le c\Delta_t^{1/3}$, therefore \eqref{eq:Delta_s_comp} holds for all $0 \le t-s \le c\Delta_t^{1/3}$.
		In particular, $\Delta_{t - c \Delta^{1/3}} \gtrsim \Delta_t + \Delta_t^{1/2} $, which implies that \eqref{eq:Delta_s_comp} holds for all $0 \le s \le t$. This concludes the proof of \eqref{eq:Delta_s_comp}.
		
		Next, we prove \eqref{eq:end_s_flow}. A similar relation for the evolution of the gaps under the free semicircular flow was studied in Section 5.1 of \cite{Cusp1}. To keep the present paper reasonably self-contained, we present a complete proof for the evolution under the characteristic flow \eqref{eq:z_flow}. 
		Observe that it follows immediately from \eqref{eq:MDEt} that the density $\rho_s(x)$ satisfies
		\begin{equation} \label{eq:rho_evol}
			\frac{\rd}{\rd x} \rho_s(x) = \frac{\rd }{\rd x} \biggl( \frac{x}{2} \rho_s(x)  + \langle \Re M_s(x)\rangle \rho_s(x) \biggr), \quad x \in \mathbb{R}, \quad 0 \le s \le t.
		\end{equation}
		Consider the mass of $\rho_s$ that lies to the left of the point $x_s$. Equation \eqref{eq:rho_evol} implies that
		\begin{equation}
			\frac{\rd}{\rd s}\int_{-\infty}^{x_s} \rho_s(x)\rd x = 0, \quad 0 \le s \le t,
		\end{equation}
		where we used that $\rho_s(x_s) = 0$. Therefore, the mass of the band of $\rho_s$ to the left of $\mathfrak{e}^-_{s}$ is constant $0 \le s \le t$. For any $ r > 0$, define $\gamma_s(r)$ implicitly by
		\begin{equation}
			\int_{-\infty}^{\gamma_s(r)}\rho_s(x)\rd x = \int_{-\infty}^{x_s} \rho_s(x)\rd x - r.
		\end{equation}
		Note that by the definition of the edge point $\mathfrak{e}^-_{s}$ and the structure theorem \cite[Theorem 7.2 (ii)]{AEK2020} for $\rho_s$, there exists a constant $\other{c} > 0$ such that $\rho_s(\gamma_s(r)) > 0$ for all $0 \le r \le \other{c}$ and all $0 \le s \le t$. Moreover, $\gamma_s(r) < \mathfrak{e}^-_s \le x_s$. Therefore, it follows from \eqref{eq:rho_evol} that (a similar equation for the free semicircular flow was obtained in Section 4.1 of \cite{Cusp2})
		\begin{equation} \label{eq:gamma_s_evol}
			\frac{\rd }{\rd s} \gamma_s(r) = -\frac{1}{2} \gamma_s(r) - \bigl\langle \Re M_s\bigl(\gamma_s(r)\bigr) \bigr\rangle, \quad 
			0 \le r \le \other{c}, \quad 0 \le s \le t.
		\end{equation}
		The evolution equation \eqref{eq:end_s_flow} for $\mathfrak{e}^-_s$ follows by taking the limit $r\to 0$ in \eqref{eq:gamma_s_evol}. An analogous argument that considers the mass of $\rho_s$ to the right of $x_s$ implies \eqref{eq:end_s_flow} for $\mathfrak{e}^+_s$. 
		
		Next, we prove the first estimate in \eqref{eq:eta_ll_Delta}. By taking the imaginary parts of \eqref{eq:z_flow} and \eqref{eq:m_flow}, we obtain
		\begin{equation} \label{eq:eta_speed}
			\eta_s \sim \eta_t + \rho_t(t-s), \quad \rho_s \sim \rho_t, \quad 0 \le s \le t.
		\end{equation}
		Moreover, it follows form the comparison relation for $\rho_t$ from \eqref{eq:rho_comp}, that
		\begin{equation} \label{eq:rhot_size}
			\rho_t \sim \eta_t (\kapd_t + \eta_t)^{-1/2} (\Delta_t + \kapd_t + \eta_t)^{-1/6} \lesssim \eta_t^{1/2} \Delta_t^{-1/6},
		\end{equation}
		where we used $0 \le \kapd_t \le \Delta_t$ and the assumption that $\eta_t \lesssim N^{-\arb}\Delta_t$.
		Therefore, using \eqref{eq:Delta_s_comp}, \eqref{eq:eta_speed} and \eqref{eq:rhot_size}, we obtain
		\begin{equation}
			\eta_s \lesssim \eta_t^{1/2}\Delta_t^{-1/6}\bigl(\Delta_t + (t-s)^{3/2}\bigr)^{2/3} \lesssim N^{-\arb/2}\Delta_t^{1/3}\Delta_s^{2/3} \lesssim N^{-\arb/2}\Delta_s, \quad 0 \le s \le t,
		\end{equation}
		hence the first bound in \eqref{eq:eta_ll_Delta} is established. To prove the second relation
		in \eqref{eq:eta_ll_Delta}, observe that it suffices to show that for all $0 \le s \le t$ and all $\eta \lesssim N^{-\arb/2}\Delta_s$, we have the bound
		\begin{equation} \label{eq:above_gap_flows_out}
			\pm\Re\bigl[F^{\pm}_s(\eta) - F^{\pm}_s(+0)\bigr] > 0, \quad F^{\pm}_s(\eta) := -\frac{1}{2}(\mathfrak{e}^{\pm}_s + \ii\eta) - \langle M_s(\mathfrak{e}^{\pm}_s + \ii\eta) \rangle.
		\end{equation}
		Indeed, \eqref{eq:above_gap_flows_out} implies that along \eqref{eq:z_flow}, for all points at level $\eta \lesssim N^{-\arb/2}\Delta_s$ above the ends $\mathfrak{e}^\pm_s$ of the gap, their projection onto the real line moves away from the gap, for all times $0 \le s \le t$. Hence no trajectory $z_s = E_s + \ii\eta_s$ satisfying $E_t \in (\mathfrak{e}^-_t, \mathfrak{e}^+_t)$ and $\eta_t \lesssim N^{-\arb}\Delta_t$ can violate \eqref{eq:eta_ll_Delta}.
		To see that \eqref{eq:above_gap_flows_out} holds, note that the Stieltjes representation for $\langle M_s\rangle$ and the universal shape (see, e.g., \cite[Eqs.~(2.4a)--(2.4e)]{Cusp1} for precise formulas)
		of the density $\rho_s$ in the vicinity of its singularities $\mathfrak{e}^\pm_s$ yields
		\begin{equation}
			\begin{split}
				-\Re\bigl[F^{-}_s(\eta) - F^{-}_s(+0)\bigr] &= \frac{1}{\pi}\int_{\mathbb{R}}\frac{-\eta^2 }{x(x^2 + \eta^2)}   \rho_s(\mathfrak{e}^-_s + x)\rd x \ge C\Delta_s^{-1/6}\eta^{1/2}  + \mathcal{O}\bigl(\Delta_s^{-5/3}\eta^2\bigr) + \mathcal{O}(\eta^2) > 0,
			\end{split}		
		\end{equation}
		where in the last line we used the assumption $\eta \lesssim N^{-\arb/2}\Delta_s \lesssim N^{-\arb/2}$. The computation  for $F^{+}_s$ is completely analogous. This concludes the proof of \eqref{eq:eta_ll_Delta}.
		
		Next, we prove \eqref{eq:eta_ll_d_preserved}. Using the comparison relations \eqref{eq:eta_speed}
		for $\rho_s$ and $\rho^{-1}_s\eta_s$, together with the bound   $\eta_s \lesssim N^{-\arb/2} \Delta_s$, and the assumption 
		$\kapd_t(z_t) \gtrsim N^{\arb}\eta_t\,$, we deduce that
		\begin{equation}
			1 + \frac{\kapd_s}{\eta_s} \sim \frac{\rho_t^{-2}\eta_t + (t-s)}{\Delta_t^{1/3} + (t-s)^{1/2}} \gtrsim \min\biggl\{ 1+ \frac{\kapd_t }{\eta_t}, \frac{\kapd_t^{1/2}\Delta_t^{1/2}}{\eta_t} \biggr\} \gtrsim \frac{\kapd_t}{\eta_t},
		\end{equation}
		which implies \eqref{eq:eta_ll_d_preserved} immediately.

		Finally, we prove \eqref{eq:kapd_s_bound}. 	
		Without loss of generality, we can assume $z_s = \mathfrak{e}^-_s  + y_s + \ii \eta_s$ with $0 \le y_s \le (1-C_1)\Delta_s$ for some $ 1 \lesssim C_1 < 3/4$. Considering the difference of the real part of \eqref{eq:z_flow} and \eqref{eq:end_s_flow}, we obtain
		\begin{equation} \label{eq:d_w_s}
			\begin{split}
				\frac{\rd}{\rd s} y_s + \frac{1}{2}y_s  &=  \Re\bigl\langle M_s(\mathfrak{e}^-_s) - M_s(\mathfrak{e}^-_s+y_s) + M_s(\mathfrak{e}^-_s+y_s)  - M_s(z_s) \bigr\rangle \\
				&= -y_s \frac{1}{\pi}  \int_{\mathbb{ R} }\frac{\rho_s(\mathfrak{e}^-_s+x)\rd x}{ x (x-y_s)} - \eta_s^2 \frac{1}{\pi}  \int_{\mathbb{ R} }\frac{\rho_s(\mathfrak{e}^-_s+x)\rd x}{(y_s -x )\big( (y_s-x)^2 +\eta_s^2\bigr)},
			\end{split}
		\end{equation}
		where in the second line we used the Stieltjes representation for $\langle M_s \rangle$. Using the universal shape of the density $\rho_s$ near the singularities $\mathfrak{e}^\pm_s$, we conclude that uniformly in $0\le s\le t \le T$,
		\begin{equation} \label{eq:w_int_comps}
			y_s \int_{\mathbb{ R} }\frac{\rho_s(\mathfrak{e}^-_s+x)\rd x}{ x (x-y_s)} \gtrsim \frac{y_s^{1/2}}{\Delta_s^{1/6}}, \quad 
			\eta_s^2 \biggl\lvert \int_{\mathbb{ R} }\frac{\rho_s(\mathfrak{e}^-_s+x)\rd x}{(y_s-x)\big( (x-y_s)^2 +\eta_s^2\bigr)} \biggr\rvert \lesssim \frac{\eta_s^2}{\Delta_s^{1/6} (y_s +\eta_s)^{3/2}}. 
		\end{equation}
		Since $\eta_s \lesssim N^{-\arb/2}\kapd_s$ and $\kapd_s \le y_s$, we conclude from \eqref{eq:d_w_s} and \eqref{eq:w_int_comps} that 
		\begin{equation}
			\frac{\rd}{\rd s} y_s \lesssim - \frac{y_s^{1/2}}{\Delta_s^{1/6} },
		\end{equation}
		which, together with \eqref{eq:Delta_s_comp}, implies \eqref{eq:kapd_s_bound} for some constant 
		$\mathfrak{c}>0$.
		This concludes the proof of Lemma~\ref{lemma:rho_t_props}.
	\end{proof}
	
	\begin{proof}[Proof of Lemma~\ref{lemma:exclusion}]
		We restrict our considerations to the %
		event $\Omega$ on which the assumed bound \eqref{eq:excl_template} holds.
		
		First, we prove \eqref{eq:no-spec}. 
		Assume, to the contrary, that there is an eigenvalue $\lambda$ of $H$ such that $\lambda \in [\mathfrak{e}^{-}_t+f(t), \mathfrak{e}^+_t-f(t)]$, then 
		\begin{equation}
			\bigl\langle \Im G(\lambda + \ii \eta) \bigr\rangle \ge \frac{1}{N\eta}, \quad \eta > 0.
		\end{equation}	
		On the other hand, choosing $\eta$ implicitly such that $\rho_t(\lambda+\ii\eta) N \eta = N^{-\zeta/2}$, implies that  $\lambda + \ii\eta \in \subD_t$ and  $\eta \sim \kapd_t(\lambda)^{1/4} \eta_{\mathfrak{f},t}^{3/4}N^{-{\zeta}/4}$ . Therefore, using the assumed bound \eqref{eq:excl_template} with data $(\subD_{t}, \sscl+\ell\arb, \Omega)$ yields
		\begin{equation}
			\bigl\langle\Im G(\lambda + \ii \eta) \bigr\rangle \lesssim \frac{\rho_t(\lambda + \ii \eta )N\eta + N^{-\zeta} (\log N)^{\gamma} }{N\eta} \lesssim \frac{ N^{-\zeta/2}}{N\eta} .
		\end{equation}
		Therefore, we conclude by contradiction that \eqref{eq:no-spec} holds on $\Omega$. This concludes the proof of Lemma~\ref{lemma:exclusion}.	
	\end{proof}

	\section{Polynomially Decaying Metric Correlation Structure} \label{app:tree}
	In this section, we verify the last condition in Assumption~\ref{ass:cumulants}~(i) for the ensemble in Example~\ref{ex:tree}.
	More precisely, we show that \eqref{eq:kappa_3_av_norm} holds 
	under the assumption that (recall \eqref{eq:tree_kappak} from Example~\ref{ex:tree})
	\begin{equation} \label{eq:tree_decay}
		\bigl\lvert \kappa(\alpha_1, \alpha_2, \alpha_3) \bigr\rvert \le C_3 
		\prod_{e \in \mathfrak{T}_{\mathrm{min}} } \frac{1}{1 + d(e)^s} ,
	\end{equation}
	for some\footnote{  The estimate \eqref{eq:tree_target} below can be proved under the relaxed summability condition $s > 3/2$. However, $s > 2$ in \eqref{eq:tree_kappak} is still necessary for \eqref{eq:summcum}--\eqref{eq:summcum3c}. }  $s > 2$, where $\mathfrak{T}_{\mathrm{min}}$ is a minimal spanning tree
	in a complete graph with vertices $\alpha_1, \alpha_2, \alpha_3$ and edge weights induced by distance $d$, defined in \eqref{eq:label_dist}.
	That is, out goal is to show that, for all $X,Y,Z \in \mathbb{C}^{N\times N}$, the estimate
	\begin{equation} \label{eq:tree_target}
		N^{-3/2}\sum_{\alpha_1, \alpha_2, \alpha_3} \bigl\lvert \kappa(\alpha_1, \alpha_2, \alpha_3) \bigr\rvert |X_{b_1a_2}|\,|Y_{b_2a_3}|\,|Z_{b_3a_1}| \lesssim C_3 \norm{X}\norm{Y}\norm{Z}_{\mathrm{hs}},  \qquad \alpha_j := (a_j,b_j),
	\end{equation}
	holds for some absolute implicit constant, where $C_3$ is the constant from \eqref{eq:tree_decay}.

	We estimate the contribution of the case when $(\alpha_1,\alpha_3) \in \mathfrak{T}_{\mathrm{min}}$ and $d(\alpha_1,\alpha_3) = |a_1-b_3|+|b_1-a_3|$ in full detail. It is straightforward to check that in all other cases, using the trivial bounds $|X_{b_1a_2}| \le \norm{X}$, $|Y_{b_2a_3}| \le \norm{Y}$ is sufficient. 
	Indeed, since $s > 2$, the indices $b_1, a_2, b_2, a_3$ can be summed up after using the norm bounds on $X$ and $Y$; then for the remaining $(a_1,b_3)$ sum, we use $N^{-3/2}\sum_{a_1,b_3} |Z_{a_1b_3}| \lesssim \norm{Z}_{\mathrm{hs}}$.
	
	Therefore, it suffices to bound
	\begin{equation} \label{eq:X_tree}
		\mathfrak{X} \equiv\mathfrak{X}(X,Y,Z) :=	C_3N^{-3/2}\sum_{\alpha_1, \alpha_3}\frac{|Z_{b_3a_1}| }{1 + \bigl(|a_1-b_3|+|b_1-a_3|\bigr)^s}\, \sum_{\alpha_2}\frac{|X_{b_1a_2}|\,|Y_{b_2a_3}|}{1 + \bigl(|a_1-a_2|+|b_1-b_2|\bigr)^s},
	\end{equation}
	where we assumed for concreteness that $\mathfrak{T}_{\mathrm{min}} = \{(\alpha_1,\alpha_2), (\alpha_1,\alpha_3)\}$ and $d(\alpha_1,\alpha_3) = |a_1-a_2|+|b_1-b_2|$ (other cases are identical).
	First, we use the Schwarz inequality in the $b_2$ summation, to obtain
	\begin{equation} \label{eq:b2sum}
		\sum_{b_2}\frac{|Y_{b_2a_3}|}{1 + \bigl(|a_1-a_2|+|b_1-b_2|\bigr)^s} \lesssim \frac{1}{1+|a_1-a_2|^{s-1/2}}\sqrt{\sum_{b_2}|Y_{b_2a_3}|^2} \lesssim \frac{\norm{Y}}{1+|a_1-a_2|^{s-1/2}}. 
	\end{equation}
	Plugging \eqref{eq:b2sum} into the expression for $\mathfrak{X}$ in \eqref{eq:X_tree} and performing the summation in $a_3$, we obtain the estimates
	\begin{equation}
		\begin{split}
			\mathfrak{X} &\lesssim C_3\norm{Y}N^{-3/2}\sum_{a_1, b_3}\frac{|Z_{b_3a_1}| }{1+|b_3-a_1|^{s-1}}\, \sum_{a_2}\frac{1}{1+|a_1-a_2|^{s-1/2}}\,\sum_{b_1}|X_{b_1a_2}| \\
			&\lesssim C_3\norm{X}\norm{Y}N^{-1}\sum_{a_1, b_3}\frac{|Z_{b_3a_1}| }{1+|b_3-a_1|^{s-1}}\, \sum_{a_2}\frac{1}{1+|a_1-a_2|^{s-1/2}} \lesssim C_3\norm{X}\norm{Y}\norm{Z}_{\mathrm{hs}},
		\end{split}
	\end{equation}
	where in the second step we used Schwarz inequality in $b_1$, and in the ultimate step we use the fact that $s>2$ to first sum the convergent series in $a_2$, and then apply Schwarz in $(a_1,b_3)$. This yields the desired \eqref{eq:tree_target}.
	
	\vspace{5pt}
	{\bf Acknowledgment.} We thank Giorgio Cipolloni for many productive discussions.
	
	\vspace{5pt}
	{\bf Data availability.} There is no data associated  to this work.
	
	\vspace{5pt}
	{\bf Competing interests.} The authors have no conflict of interest to disclose.
	
	\printbibliography
\end{document}